\documentclass[11pt,reqno]{amsart}
\usepackage[latin1]{inputenc}
\usepackage{amssymb}
\usepackage{mathrsfs}
\usepackage{amsfonts}
\usepackage{amsfonts,amssymb,amsmath,amsthm}
\usepackage{url}
\usepackage{enumerate}
\usepackage[pdftex,bookmarksnumbered]{hyperref}
\usepackage{graphicx,xcolor}
\usepackage{amsfonts}
\usepackage{}
\usepackage{mathrsfs}
\usepackage{bbm}
\usepackage{url}
\usepackage{enumerate}
\usepackage{graphicx}
\usepackage[pdftex,bookmarksnumbered]{hyperref}
\numberwithin{equation}{section}

\newtheorem{theorem}{Theorem}[section]
\newtheorem{lemma}[theorem]{Lemma}
\newtheorem{proposition}[theorem]{Proposition}

\theoremstyle{definition}
\newtheorem{definition}[theorem]{Definition}
\newtheorem{remark}[theorem]{Remark}
\numberwithin{equation}{section}

\usepackage[left=2.0cm, right=2.0cm, top=2.0cm, bottom=2.0cm]{geometry}

\usepackage{graphicx}%
\usepackage{color}
\usepackage{cite}
\usepackage{hyperref}
\hypersetup{
	colorlinks = true,%
	citecolor = blue,
	filecolor=red,%
	linkcolor = [rgb]{0.65,0.0,0.0},%
	anchorcolor = red,
	pagecolor = red,
	urlcolor= [rgb]{0.65,0.0,0.0}
	linktocpage=true,
	pdfpagelabels=true,
	bookmarksnumbered=true,
}

\DeclareMathOperator{\Ho}{H}

\DeclareMathOperator{\id}{Id}
\DeclareMathOperator{\diam}{diam}
\DeclareMathOperator{\vol}{vol}
\DeclareMathOperator{\Ca}{Cap}

 \DeclareMathOperator{\Co}{Cov}
\DeclareMathOperator{\A}{A}

\DeclareMathOperator{\Dim}{\mathrm{\mathtt{dim}}}

\DeclareMathOperator{\cat}{cat}

\DeclareMathOperator{\ess}{ess}

\DeclareMathOperator{\ind}{ind}

\newcommand{\ds}{\displaystyle}

\author{Alexandru Krist\'aly}
\address{Department of Economics\\
	Babe\c s-Bolyai University\\
	400591 Cluj-Napoca, Romania \&  Institute of Applied Mathematics\\
 \'Obuda University\\
 1034 Budapest, Hungary}
  \email{alex.kristaly@econ.ubbcluj.ro; kristaly.alexandru@nik.uni-obuda.hu}

\author{Zhongmin Shen}
\address{
Department of Mathematical Sciences\\
Indiana University-Purdue University Indianapolis\\
Indiana, U.S.A.}
\email{zshen@math.iupui.edu}

\author{Lixia Yuan}
\address{
School of Mathematics and Physics\\
Shanghai Normal University\\
Shanghai, China}
\email{yuanlixia@shnu.edu.cn}

\author{Wei Zhao}
\address{
Department of Mathematics\\
East China University of Science and Technology\\
Shanghai, China}
\email{szhao\underline{ }wei@yahoo.com}

\begin{document}

\title[]{Nonlinear spectrums  of Finsler manifolds}

\begin{abstract}
In this paper we investigate the spectral problem in Finsler geometry. Due to the nonlinearity of the Finsler-Laplacian operator, we introduce \textit{faithful dimension pairs} by means of which the spectrum of a compact reversible Finsler metric measure manifold is defined.
Various upper and lower bounds of such eigenvalues are provided in the spirit of Cheng, Buser and Gromov, which extend in several aspects the results of Hassannezhad, Kokarev and Polterovich.  Moreover, we construct several faithful dimension pairs based on Lusternik-Schnirelmann category, Krasnoselskii genus and essential dimension, respectively; however, we also show that the Lebesgue covering dimension pair is not faithful. As an application, we show that the Bakry-\'Emery spectrum of a closed weighted Riemannian manifold can be characterized by the faithful Lusternik-Schnirelmann dimension pair.
\end{abstract}
\maketitle

\section{Introduction}

According to S.-S. Chern \cite{Chern}, 'Finsler geometry is just Riemannian geometry without the quadratic restriction'. Chern's statement is fairly confirmed as most of the well-known results from Riemannian geometry -- by suitable modifications -- have their Finslerian accompanying, e.g. Hopf-Rinow, Hadamard-Cartan and Bonnet-Myers theorems as well as Rauch and Bishop-Gromov comparison principle,  see Bao, Chern and Shen \cite{BCS}. However, genuine differences occur between the two geometries; let us recall just three of them. First, unlike the Hopf classification in Riemannian geometry, no full characterization is available for Finsler manifolds having constant flag curvature; in fact, various subclasses of Finsler manifolds seem to play a crucial role in such a description (as Minkowski, Berwald, Landberg, Randers spaces), see e.g. Shen \cite{Shen-1, Shen-2}. Second, unlike in inner spaces,  affine 2-disks in normed Minkowski spaces  are not area-minimizing among rational rational chains having the same boundary, see Burago and  Ivanov \cite{BI}.
Another unexpected phenomenon arises in the theory of  Sobolev spaces; indeed, while Sobolev spaces over complete Riemannian manifolds have the expected properties (separability, reflexivity, embeddings, etc), see Hebey \cite{H},  it turns out that Sobolev spaces over non-compact Finsler manifolds should not have even a vector space structure, see Krist\'aly and Rudas \cite{K-R}.

The aim of the present paper is to investigate the spectral problem on compact reversible Finsler manifolds.  The main difficulty relies on the \textit{nonlinearity}  of the Finsler-Laplace operator  unless the Finsler manifold is Riemannian. To be more precise, let us consider a Finsler metric measure manifold $(M,F,d\mathfrak{m})$ (shortly, FMMM), i.e., $(M,F)$ is a reversible Finsler manifold endowed with a smooth measure $d\mathfrak{m}$. Let  $(x^i)$ be a local coordinate for $M$ and $(x^i,\eta^i)$ be the induced   coordinates for $T^*M$. Set $d\mathfrak{m} = \sigma(x) dx^1 \dots dx^n$ and $ g^{*ij} (x,\eta) := \frac{1}{2}[F^{*2}]_{\eta^i\eta^j} (x, \eta)$, where $F^*$ is the co-Finsler metric on the cotangent bundle $T^*M$, see Section \ref{section2}.
The \textit{Finsler-Laplace operator} $\Delta$ on $(M,F,d\mathfrak{m})$  is given by
\[
\Delta u(x) = \frac{1}{\sigma(x)} \frac{\partial}{\partial x^i} \Big [ \sigma (x) g^{*ij} (x, du)\frac{\partial u}{\partial x^j} (x) \Big ],\text{ if }du(x)\neq0,\ u\in C_0^\infty(M).
\]
The dependence of $g^{*ij} (x, \eta)$ by $\eta$ clearly implies the nonlinearity of $u\mapsto \Delta u$, unless $(M,F)$ is Riemannian, see e.g. Shen \cite[Example 3.2.1]{Sh1}.
 The \textit{spectrum} of $(M, F, d\mathfrak{m})$ is defined to be the set of numbers
  $\lambda$ such that the nonlinear
  equation
\begin{equation*}
-\Delta u = \lambda u,\tag{1.1} \label{Laplacian}
\end{equation*}
has a nontrivial solution; in such a case,  $\lambda$ is an {\it eigenvalue} of $\Delta$ or $(M, F, d\mathfrak{m})$. From the Morse-theoretical point of view of the spectrum, see e.g. Gromov \cite{G}, equation (\ref{Laplacian}) is precisely the Euler-Lagrange equation of the {\it canonical energy functional}  $E$ given by
\begin{equation*}
E (u): =E_F (u)= \frac{ \ds\int_M [F^{*} (du)]^2   d\mathfrak{m} }{ \ds\int_M u^2   d\mathfrak{m}} ,\ \ \forall\,u\in\mathscr{X}\backslash\{0\},
\end{equation*}
where  $\mathscr{X}$ is the Sobolev space consisting of  $H^{1}$ functions on $M$ (with $u|_{\partial M }=0$ if $\partial M \not= \emptyset$); therefore, the spectrum of $(M, F, d\mathfrak{m})$ is the set of critical points of $E$.
We notice that the spectral problem on Riemannian manifolds has been intensively studied, see e.g. Chavel \cite{Ch}; in particular, the Beltrami-Laplace operator  $u\mapsto \Delta u$ in (\ref{Laplacian}) is linear and the approach of Gromov \cite{G} can be fully applied in order to state qualitative results for the spectrum of compact Riemannian manifolds.


%
%


Following the abstract idea of Gromov \cite{G},  the nonlinear character of the Finsler-Laplace operator $\Delta$ on a generic compact FMMM $(M,F,d\mathfrak{m})$ heavily motivates the introduction of a  \textit{dimension-like function} $\Dim $  on a collection  $\mathscr{C}$  of certain subsets of
$$\mathcal{S} := \left\{ u\in \mathscr{X} : \ \ds\int_M u^2 d\mathfrak{m} =1 \right\}$$
in order to capture an infinite sequence of eigenvalues of $\Delta$.
To do so, for every positive integer $k$, set
\begin{equation*}
 \lambda_k := \sup\Big \{ \lambda \geq 0 : \ \Dim E^{-1}[0, \lambda ] < k \Big \}, \tag{1.2}\label{spectrum2}
 \end{equation*}
where $\Dim E^{-1}[0, \lambda ] := \sup\{ \Dim (A): \ A \in \mathscr{C}, \; A \subset E^{-1}[0, \lambda ] \}$; the set $\{ \lambda_k \}_{k=1}^{\infty}$ is called the  $({\mathscr C}, \Dim)$-{\it spectrum}.  As expected, the set of eigenvalues  $\{ \lambda_k\}_{k=1}^{\infty}$  defined by (\ref{spectrum2}) might not be the set of all critical values of $E$.
Even more,  a generic $({\mathscr C}, \Dim)$-spectrum may have a completely different behavior w.r.t. the spectrum of $(M,F, d\mathfrak{m})$, see e.g. Proposition \ref{flawex1} for a nontrivial example where the $({\mathscr C}, \Dim)$-spectrum is a singleton. Accordingly, a challenging question is to identify dimension pairs $({\mathscr C}, \Dim)$ whose spectrum inherits the expected features of the spectrum of $(M,F, d\mathfrak{m})$. A possible way is to introduce {\it faithful dimension pairs} $(\mathscr{C}, \Dim)$, see Definition \ref{faithful}, which requires that the $({\mathscr C}, \Dim)$-spectrum and the Courant spectrum coincide for \textit{every} (test) Riemannian metric $g$ acting on $M$, the measure being the canonical one $d\vol_{{g}}$. It turns out that faithful dimension pairs occur quite often; we construct several ones based on Lusternik-Schnirelmann category, Krasnoselskii genus and essential dimension, respectively, see Section \ref{examplest}. Our first result establishes a close relationship between  the spectrum of $(M,F, d\mathfrak{m})$ and the $(\mathscr{C}, \Dim)$-spectrum of a faithful dimension pair; to state it, we consider the Sobolev space
\[
\mathscr{X}_0=\left\{u\in H^1(M):u|_{\partial M}=0\ \text{if}\ \partial M \not= \emptyset  \ \text{or}\  \ds\int_M u \; d\mathfrak{m} =0\ \text{if}\ \partial M = \emptyset\right\}.
\]

\begin{theorem}\label{th1.1}
Let $(M,F,d\mathfrak{m})$ be a compact {\rm FMMM}. For any faithful dimension pair $(\mathscr{C},{\rm \Dim})$, every number $\lambda_k$ in its  spectrum
 belongs to the spectrum of $(M, F, d\mathfrak{m})$, or equivalently,
there exists $u\in \mathscr{X}_0\backslash\{0\}$ or $u=\text{\rm const.}\neq0$ with
\[
-\Delta u=\lambda_k u \text{ in the weak sense}.
\]
 Moreover,
the spectrum $\{\lambda_k\}_{k=1}^{\infty}$  has the following properties$:$
\begin{align*}
&0=\lambda_1<\lambda_2\leq  \ldots \leq \lambda_k\leq\ldots \nearrow +\infty,\text{ if }\partial M=\emptyset;\\
&0<\lambda_1\leq\lambda_2\leq  \ldots \leq \lambda_k\leq\ldots \nearrow +\infty,\text{ if }\partial M\neq\emptyset,
\end{align*}
where the first positive eigenvalue is given by
\[
\left\{
\begin{array}{llll}
 \lambda_2=\inf_{u\in \mathscr{X}_0\backslash\{0\}}E(u),&\text{ if }\partial M=\emptyset;\\
 \\
 \lambda_1=\inf_{u\in \mathscr{X}_0\backslash\{0\}} E(u),&\text{ if }\partial M\neq\emptyset.
\end{array}
\right.
\]
\end{theorem}

In the sequel, our interest is to provide upper and lower bound estimates for the eigenvalues associated with a fixed dimension pair. First, we provide a Cheng type estimate for generic dimension pairs, i.e., the  eigenvalues $\lambda_k$'s are bounded from above by a term involving bounds of the weighted Ricci curvature ${\bf Ric}_N$ (cf. Ohta and Sturm \cite{O}) and diameter of the FMMM.

\begin{theorem}\label{th1.2}Given $N\in [n,\infty)\cap \mathbb N$, $K\in \mathbb{R}$ and $d>0$,
let $(M, F, d\mathfrak{m})$ be an $n$-dimensional
closed  {\rm FMMM} with
\[
 {\bf Ric}_N \geq (N-1)K, \ \ \ \ \diam(M)= d.
\]
Then there exists $C_1=C_1(N)>0$  depending only on $N$ such that for any dimension pair $(\mathscr{C},\Dim)$ the corresponding eigenvalues $\lambda_k$'s  satisfy
\[
\lambda_{k}\leq  \frac{(N-1)^2}{4}|K|+C_1(N)\left(\frac{k}{d}\right)^2.
\]
\end{theorem}

Theorem \ref{th1.2} extends the estimates of Cheng \cite[Corollary 2.3]{C} and Hassannezhad,  Kokarev and Polterovich \cite[Theorem 1.3.1]{HKP} to Finsler manifolds.   The above estimate is asymptotically optimal, i.e. one cannot replace $\left(\frac{k}{d}\right)^2$ by $\left(\frac{k}{d}\right)^{2-\varepsilon}$ for any $\varepsilon>0$; indeed, in the $n$-dimensional unit sphere $\mathbb S^n$ with its canonical metric, for any faithful dimension pair we have $\lambda_k=k(k+n-1)$, $k\in \mathbb N.$ Moreover, Theorem \ref{th1.2} also handles the case in Proposition \ref{flawex1}, where the $({\mathscr C}, \Dim)$-spectrum contains only one element.

Unlike in the Riemannian setting,  various measures can be introduced on a Finsler manifold  whose behavior may be genuinely different. Two such frequently used measures are the   Busemann-Hausdorff measure $d\mathfrak{m}_{BH}$ and  Holmes-Thompson measure $d\mathfrak{m}_{HT}$, see Alvarez-Paiva and  Berck\cite{AlB} and Alvarez-Paiva and Thompson \cite{AlT}. These two measures become the canonical Riemannan measure whenever the Finsler metric is Riemannian. Let $\Lambda_F\geq 1$ be the uniformity constant of $(M,F)$, with  $\Lambda_F= 1$  if and only if $F$ is Riemannian (cf. Egloff \cite{E}). The following result provides a Gromov type estimate, see \cite{G3, G2}.

\begin{theorem}\label{GromovtypBH}Given $K\in \mathbb R$ and $d>0$,
let $(M,F,d\mathfrak{m})$ be an $n$-dimensional closed  {\rm FMMM} with
\[
\mathbf{Ric}\geq (n-1)K,  \ \ \ \ \diam(M)= d,
\]
where $d\mathfrak{m}$ is either the Busemann-Hausdorff measure or the Holmes-Thompson measure.
Then there  exists  $C_2=C_2(n)>0$  depending only on $n$ such that for any faithful dimension pair $(\mathscr{C},\Dim)$ the corresponding eigenvalues $\lambda_k$'s  satisfy
\textcolor[rgb]{0.00,0.00,0.00}{\[
\lambda_{k+1}\geq \frac{C_2^{1+d\sqrt{|K|}}}{\Lambda^{24n}_F d^2}\, k^\frac{2}{n},\ \forall\,k\in \mathbb{N}.
\]}

\end{theorem}

We notice that the faitfulness of the dimension pair in Theorem \ref{GromovtypBH} is indispensable; see again  Proposition \ref{flawex1}. For a closed Riemannian manifold (endowed with its canonical measure), Theorem \ref{GromovtypBH} reduces to the estimate given by Gromov \cite[Appendix C]{G2} and   Hassannezhad,  Kokarev and Polterovich \cite[Theorem 1.2.1]{HKP}, while  Weyl's asymptotic law (see e.g. Chavel \cite[p.9]{Ch}) implies the asymptotic optimality of the latter estimate. Moreover,
Theorem \ref{GromovtypBH} can be extended to arbitrary measures, see Theorem \ref{thefisrttheorem}, where a weaker estimate is obtained on the right hand side of the above inequality containing quantitative information on the \textit{distortion} of  $(M,F,d\mathfrak{m})$.  In addition, for some special faithful dimension pairs, we obtain  better estimates which are not only independent of the uniformity constant $\Lambda_F$ but also valid for arbitrary measures, see Theorem \ref{thesecondtheorem}.

We also provide a Buser type estimate; hereafter, $\mathfrak{i}_M$ stands for the injectivity radius of $(M,F)$.
\begin{theorem}\label{Busertype}
Given $K\in \mathbb R$ and $V>0$, let $(M,F,d\mathfrak{m})$ be an $n$-dimensional closed {\rm FMMM} with
\[
\mathbf{Ric}\geq (n-1)K,  \ \ \ \ \mathfrak{m}(M)=V,
\]
 where $d\mathfrak{m}$ is either the Busemann-Hausdorff measure or the Holmes-Thompson measure.
Then there exist  $C_3=C_3(n)>0$ and $C_4=C_4(n)>0$ both depending only on $n$ such that for  any faithful dimension pair $(\mathscr{C},\Dim)$ the corresponding eigenvalues $\lambda_k$'s  satisfy
\[
\lambda_{k+1}\geq \frac{C_3}{\Lambda^{32n}_F }\, \left(\frac{k}{V}\right)^\frac{2}{n},\ \forall\,k\geq C_4 \max\left\{  \mathfrak{i}_M^{-n}, |K|^{\frac{n}{2}}   \right\} \Lambda_F^{5n^2}V.
\]
\end{theorem}

As an application, we show that for every closed weighted Riemannian manifold $(M,g, e^{-f}d\vol_g)$ the Lusternik-Schnirelmann spectrum is precisely the spectrum of the Bakry-\'Emery Laplacian, see Theorem \ref{wegitedLS}; the proof is based on the fact that $(M,g, e^{-f}d\vol_g)$ can be viewed as an FMMM $(M,F,d\mathfrak{m})$ with the metric $F=\sqrt{g}$ and measure $d\mathfrak{m}=e^{-f}d\vol_g$, respectively.

The paper is organized as follows.
In Section \ref{section2} we recall/prove those notions/results which are indispensable in our study (Finsler geometry, Sobolev spaces, energy functionals). In Section \ref{section3} we introduce the spectrum of the dimension pairs and we construct several faithful dimension pairs. In Section  \ref{section4} we prove the Cheng type upper estimate (proof of Theorem \ref{Chees}), while in Section \ref{section5} lower bound estimates are given for the eigenvalues (proofs of Theorems  \ref{thefisrttheorem}, \ref{compBuse} and \ref{thesecondtheorem}). In Section \ref{section6} we prove Theorem \ref{wegitedLS} by joining the Lusternik-Schnirelmann spectrum with the spectrum of the Bakry-\'Emery Laplacian (proof of Theorem \ref{wegitedLS}). In Section   \ref{section7} we prove some  technical results which are used throughout the previous sections.



\section{Preliminaries}\label{section2}
\subsection{Elements from Finsler geometry}\label{finslerknowdege} In this section, we recall some definitions and properties about Finsler manifolds;  see Bao, Chern and Shen \cite{BCS} and Shen \cite{Sh1} for more details.

\subsubsection{Finsler manifolds}

 Let $M$ be a connected
$n$-dimensional smooth manifold and $TM=\bigcup_{x \in M}T_{x}
M $ be its tangent bundle. The pair $(M,F)$ is a \textit{reversible Finsler
	manifold} if   $F:TM\to [0,+\infty)$ satisfies
the conditions:

\smallskip

(a) $F\in C^{+\infty}(TM\setminus\{ 0 \});$

(b) $F(x,\lambda y)=|\lambda| F(x,y)$ for all $\lambda\in \mathbbm{R}$ and $(x,y)\in TM;$

(c) $g_{ij}(x,y)=[\frac12F^{2}%
]_{y^{i}y^{j}}(x,y)$ is positive definite for all $(x,y)\in
TM\setminus\{ 0 \}$, where $F(x,y):=F(y^i\frac{\partial}{\partial x^i}|_x)$.

\smallskip

\noindent The Euler theorem yields $F(x,y)=\sqrt{g_{ij}(x,y)y^iy^j}$ for any $y\in TM\backslash\{0\}$.
Moreover, $(g_{ij}(x,y))$ can be defined at $y=0$ if and only if it is independent of $y$, in which case $F$ is Riemannian.

 Set $S_xM:=\{y\in T_xM:F(x,y)=1\}$ and $SM:=\cup_{x\in M}S_xM$. The  {\it uniformity constant} $\Lambda_F$ (cf. Egloff \cite{E}) is defined by
\[
\Lambda_F:=\underset{X,Y,Z\in SM}{\sup}\frac{g_X(Y,Y)}{g_Z(Y,Y)},\ \ \ \ \text{ where }g_X(Y,Y)=g_{ij}(x,X)\,Y^iY^j.
\]
Clearly, ${\Lambda_F} \geq 1$ with equality if and only if $F $ is Riemannian.

The {\it average Riemannian metric} $\hat{g}$ on $M$ induced by $F$ is defined as
\[
\hat{g}(X,Y):=\frac{1}{\nu(S_xM)}\ds\int_{S_xM}g_y(X,Y)d\nu_x(y),\ \forall\,X,Y\in T_xM,\tag{2.1}\label{averageRieman}
\]
where $\nu(S_xM)=\ds\int_{S_xM}d\nu_x(y)$, and $d\nu_x$ is the canonical Riemannian measure on $S_xM$ induced by $F$.
Simple estimates yield
\[
\Lambda^{-1}_F\cdot F^2(X)\leq \hat{g}(X,X)\leq \Lambda_F \cdot F^2(X), \forall\,X\in TM.\tag{2.2}\label{averagenorm}
\]

The {\it co-Finsler $($dual$)$ metric} $F^*$ on $M$ is
defined by
\begin{equation*}
F^*(\eta):=\underset{X\in T_xM\backslash \{0\}}{\sup}\frac{\eta(X)}{F(X)}, \ \
\forall \eta\in T_x^*M,
\end{equation*}
which is  a Finsler metric on $T^*M$.
The {\it Legendre transformation} $\mathfrak{L} : TM \rightarrow T^*M$ is defined
by
\begin{equation*}
\mathfrak{L}(X):=\left \{
\begin{array}{lll}
g_X(X,\cdot), & \ \ \ \text{ if }X\neq0, \\
\\
0, & \ \ \ \text{ if }X=0.%
\end{array}
\right.
\end{equation*}
In particular, $F^*(\mathfrak{L}(X))=F(X)$.
Given $f \in C^1(M)$, the
{\it gradient} of $f$ is defined as $\nabla f = \mathfrak{L}^{-1}(df)$. Thus,  $df(X) = g_{\nabla f} (\nabla f,X)$. We remark that $\nabla$ is usually nonlinear, i.e., $\nabla (f+h)\neq \nabla f+\nabla h$.

Let $\zeta:[0,1]\rightarrow M$ be a Lipschitz continuous path. The length of $\zeta$ is defined by
\[
L_F(\zeta):=\int^1_0 F(\dot{\zeta}(t))dt.
\]
Define the {\it distance function} $d_F:M\times M\rightarrow [0,+\infty)$ by
$d_F(x_1,x_2):=\inf L_F(\sigma)$,
where the infimum is taken over all
Lipschitz continuous paths $\zeta:[a,b]\rightarrow M$ with
$\zeta(a)=x_1$ and $\zeta(b)=x_2$.
Given $R>0$, the {\it $R$-ball} centered at $p$ is defined as $B_p(R):=\{x\in M:\, d_F(p,x)<R\}$.

A smooth curve $t\mapsto \gamma(t)$ in $M$ is called a (constant speed) \textit{geodesic} if it satisfies
\[
\frac{d^2\gamma^i}{dt^2}+2G^i\left(\frac{d\gamma}{dt}\right)=0,
\]
where
\begin{align*}
G^i(y):=\frac14 g^{il}(y)\left\{2\frac{\partial g_{jl}}{\partial x^k}(y)-\frac{\partial g_{jk}}{\partial x^l}(y)\right\}y^jy^k\tag{2.3}\label{geoedesiccon}
\end{align*}
is the geodesic coefficient.
We always use $\gamma_y(t)$ to denote  the geodesic with $\dot{\gamma}_y(0)=y$.

A reversible Finsler manifold $(M,F)$ is {\it complete} if  every geodesic $t\mapsto \gamma(t)$, $0< t<1$, can be extended to a geodesic defined on $-\infty< t<+\infty$.
The \textit{cut value} $i_y$ of $y\in S_xM$ is defined by
\[
i_y:=\sup\{t: \text{ the geodesic }\gamma_y|_{[0,t]} \text{ is globally minimizing}  \}.
\]
The \textit{injectivity radius} at $x$ is defined as $\mathfrak{i}_x:=\inf_{y\in S_xM} i_y$.  According to Bao, Chern and Shen \cite{BCS}, if $(M,{F})$ is  complete, then ${\mathfrak{i}_x}>0$ for any point $x\in M$. The \textit{injectivity radius of} $M$ is defined by
$\mathfrak{i}_M:=\inf_{x\in M} \mathfrak{i}_x$; if $M$ is compact, then $\mathfrak{i}_M>0$.
The \textit{cut locus} of $x$ is defined as
\[
\text{Cut}_x:=\left\{\gamma_y(i_y):\,y\in S_xM \text{ with }i_y<+\infty \right\}.
\]
In particular, $\text{Cut}_x$ is closed and has null measure.

\subsubsection{Measures and curvatures}

A triple $(M,F, d\mathfrak{m})$ is called an {FMMM} (i.e., {\it Finsler metric measure manifold}), if $(M,F)$ is a reversible Finsler manifold endowed with a smooth measure $d\mathfrak{m}$.
In a local coordinate system $(x^i)$, use $\sigma(x)$ to denote the density function
of $d\mathfrak{m}$, i.e.,
\[
d\mathfrak{m}=:\sigma(x)dx^1 \dots  dx^n.\tag{2.4}\label{density2.2}
\]
In particular,
the \textit{Busemann-Hausdorff measure} $d\mathfrak{m}_{BH}$ and the \textit{Holmes-Thompson measure} $d\mathfrak{m}_{HT}$ (cf. \cite{AlB,AlT}) are defined by
\begin{align*}
&d\mathfrak{m}_{BH}:=\frac{\vol(\mathbb{B}^{n})}{\vol(B_xM)}dx^1\dots dx^n,\\
 &d\mathfrak{m}_{HT}:=\left(\frac1{\vol(\mathbb{B}^{n})}\ds\ds\int_{B_xM}\det g_{ij}(x,y)dy^1 \dots dy^n \right) dx^1 \dots dx^n,
\end{align*}
where $B_xM:=\{y\in T_xM: F(x,y)<1\}$ and $\mathbb{B}^{n}$ is the usual Euclidean $n$-dimensional unit ball.

Given a $C^2$-function $f$, set $\mathcal {U}=\{x\in M:\, df|_x\neq0\}$. The \textit{Laplacian} of $f\in C^2(M)$ is defined on $\mathcal {U}$ by
\begin{align*}
\Delta f:=\text{div}(\nabla f)=\frac{1}{\sigma(x)}\frac{\partial}{\partial x^i}\left(\sigma(x)g^{*ij}(df|_x)\frac{\partial f}{\partial x^j}\right),\tag{2.5}\label{formofORlAPLA}
\end{align*}
where $(g^{*ij})$ is the fundamental tensor of $F^*$ and $\sigma(x)$ is defined in (\ref{density2.2}).  As in Ohta and Sturm \cite{O}, we define
the {\it distributional Laplacian} of $u\in H^1_{\text{loc}}(M)$
in the weak sense by
\[
\ds\ds\int_M v{\Delta} u d\mathfrak{m}=-\ds\ds\int_M\langle\nabla u, dv\rangle d\mathfrak{m} \text{ for all }v\in C^\infty_0(M),
\]
where $\langle\nabla u, dv\rangle:= dv(\nabla u)$ at $x\in M$ denotes the canonical pairing between $T_x^*M$ and $T_xM.$

Define the {\it distortion} $\tau$ and the {\it S-curvature} $\mathbf{S}$ of $(M,F,d\mathfrak{m})$ as
\begin{equation*}
\tau(y):=\log \frac{\sqrt{\det g_{ij}(x,y)}}{\sigma(x)}, \ \ \mathbf{S}(y):=\left.\frac{d}{dt}\right|_{t=0}[\tau(\dot{\gamma}_y(t))],\ \text{ for $y\in T_xM\backslash\{0\}$},
\end{equation*}
where $\gamma_y(t)$ is a geodesic with $\dot{\gamma}(0)=y$.

\begin{lemma}[Yuan and Zhao {\cite{YZ}}]\label{uniformBuHo}Let $(M,F,d\mathfrak{m})$ be an $n$-dimensional {\rm FMMM} with finite uniformity constant $\Lambda_F$.
If $d\mathfrak{m}$ is the Busemann-Hausdorff measure or the Holmes-Thompson measure, then $e^{\tau(y)}\in [\Lambda_F^{-n},\Lambda_F^{n}]$ for any $y\in TM\backslash\{0\}$.
\end{lemma}

The {\it Riemannian curvature} $R_y$ of $F$ is a family of linear transformations on tangent spaces. More precisely, set
$R_y:=R^i_k(y)\frac{\partial}{\partial x^i}\otimes dx^k$, where
\begin{align*}
R^i_{\,k}(y)&:=2\frac{\partial G^i}{\partial x^k}-y^j\frac{\partial^2G^i}{\partial x^j\partial y^k}+2G^j\frac{\partial^2 G^i}{\partial y^j \partial y^k}-\frac{\partial G^i}{\partial y^j}\frac{\partial G^j}{\partial y^k},
\end{align*}
where $G^i$'s are the geodesic coefficients defined in (\ref{geoedesiccon}).
The  {\it Ricci curvature} of $y\neq 0$
is defined by
$\mathbf{Ric}(y):= \frac{R^i_{i}(y)}{F^2(y)}$.  According to Ohta and Sturm \cite{O}, given $y\in SM$,
 the {\it weighted Ricci curvature} is defined by
\begin{align*}\mathbf{Ric}_N(y)=\left\{
\begin{array}{lll}
\mathbf{Ric}(y)+\left.\frac{d}{dt}\right|_{t=0}\mathbf{S}(\gamma_y(t))-\frac{\mathbf{S}^2(y)}{N-n}, && \text{ for }N\in (n,+\infty),\\
\\
\underset{L\downarrow n}{\lim}\mathbf{Ric}_L(y), && \text{ for }N=n,\\
\\
\mathbf{Ric}(y)+\left.\frac{d}{dt}\right|_{t=0}\mathbf{S}(\gamma_y(t)),  && \text{ for }N=+\infty.
\end{array}
\right.
\end{align*}
In particular,  bounding
$\mathbf{Ric}_n$ from below makes sense only if $\mathbf{S}=0$.

\subsubsection{Laplacian and volume comparison theorems}

If $M$ is complete, then  there exists a  polar coordinate system at every point of $M$ (cf. Zhao and Shen \cite{ZS}). Fixing an arbitrary point $p\in M$,
let $(r,y)$ be the {\it polar coordinate system} at $p$ and write
\[
d\mathfrak{m}=: \hat{\sigma}_p(r,y)\,dr\, d\nu_p(y),\tag{2.6}\label{volumeforminpolar}
\]
where $r$ is the distance from $p$ and $d\nu_p(y)$ is the Riemannian measure on $S_pM$  induced by $F$.

For any fixed $y\in S_pM$, we have
\[
\Delta r=\frac{\partial}{\partial r}\log( \hat{\sigma}_p(r,y)), \text{ for }0<r<i_y.\tag{2.7}\label{Laplacainr}
\]
 In particular,
\[
\lim_{r\rightarrow 0^+}\frac{ \hat{\sigma}_p(r,y)}{r^{n-1}}=e^{-\tau(y)}.\tag{2.8}\label{v14-2.1}
\]

In this paper,  $A_{n,K}(r)$ (resp. $V_{n,K}(r)$) denotes the area (resp., volume) of sphere (resp., ball) with radius $r$ in the Riemannian space form of constant curvature $K$, that is,
\[
A_{n,K}(r)=\vol(\mathbb{S}^{n-1})\mathfrak{s}^{n-1}_K(r),\ V_{n,K}(r)=\vol(\mathbb{S}^{n-1})\int^r_0\mathfrak{s}^{n-1}_K(t)dt,\tag{2.9}\label{volume2.5K}
\]
where $\mathfrak{s}_K$ is the unique
solution to $f'' + Kf = 0$ with $f(0) = 0$ and $f'(0) = 1$.
For the Ricci curvature, we have the following result; see Zhao and Shen \cite[Theorem 1.2, Remark 3.5]{ZS} for the proof.
\begin{lemma}\label{distorlemma-0}
Let $(M,F,d\mathfrak{m})$ be an $n$-dimensional complete {\rm FMMM} and let $(r,y)$ be the polar coordinate system at $p$.
\begin{itemize}
	\item[(i)] If $\mathbf{Ric}\geq (n-1)K$, then for any $y\in S_pM$, the function
	\[
	r\mapsto f_y(r):=\frac{\hat{\sigma}_p(r,y)}{e^{-\tau(\dot{\gamma}_y(r))}\mathfrak{s}^{n-1}_K(r)}
	\]
	is monotonically non-increasing  and converges to $1$ as $r\rightarrow 0^+.$
	\item[(ii)] If $\mathbf{Ric}\geq (n-1)K$ and $|\tau|\leq \log \Theta$, then
	\[
	\frac{\mathfrak{m}(B_p(R))}{\mathfrak{m}(B_p(r))}\leq \Theta^2
	\frac{V_{n,K}(R)}{V_{n,K}(r)},\ \forall\, 0<r\leq R.
	\]
\end{itemize}
\end{lemma}

For the weighted Ricci curvature, Ohta and  Sturm \cite{O} obtained the following result.

\begin{lemma}\label{Rho}
Let $(M, F, d\mathfrak{m})$ be an $n$-dimensional complete {\rm FMMM}. If for some  $N \in [n,+\infty)$ and $K\in \mathbb R$, the weighted Ricci
curvature satisfies $\mathbf{Ric}_N\geq (N-1)K$, then the Laplacian of the
distance function $r(x) = d_F(p, x)$ from any given point $p\in M$ can be estimated as
\[
\Delta r\leq \frac{d}{dr}\left(\log\mathfrak{s}^{N-1}_K(r)\right),
\]
which holds pointwisely on $M\backslash (\text{Cut}_p\cup \{p\})$ and in the sense of distributions on $M\setminus\{p\}$.

Hence,
for any $x\in M$ and $0<r\leq R$,
\[
\frac{\mathfrak{m}(B_x(R))}{\mathfrak{m}(B_x(r))}\leq\frac{V_{N,K}(R)}{V_{N,K}(r)}\leq e^{(N-1)R\sqrt{|K|}}\left( \frac{R}{r} \right)^N.
\]
\end{lemma}

Moreover, we have an extension of the so-called "segment inequality" of Cheeger and  Colding \cite[Theorem 2.11]{CC}; see Zhao \cite[Theorem 3.1, Remark 3.2]{Z22} for its proof.

\begin{theorem}\label{CheegerColding}
Given $N\in [n,+\infty)$ and  $K\in \mathbb R$, let $(M,F,d\mathfrak{m} )$ be an $n$-dimensional complete  {\rm FMMM} with $\mathbf{Ric}_N\geq (N-1)K$.
Let $A_i$, $i=1,2$ be two bounded open subsets and let $W$ be an open subset such that for each two $x_i\in A_i$, a normal minimal geodesic $\gamma_{x_1x_2}$ from $x_1$ to $x_2$ is contained in $W$. Thus, for any non-negative integrable function $f$ on $W$, we have
\begin{align*}
&\ds\int_{A_1\times A_2}\left(\int^{d_F(x_1,x_2)}_0 f(\gamma_{x_1x_2}(s))ds\right) d\mathfrak{m}_{\times}\\
\leq& C(N,K,d)\left[ \mathfrak{m}(A_1)\diam(A_2)+\mathfrak{m}(A_2)\diam(A_1)\right]\ds\int_W f d\mathfrak{m},
\end{align*}
where $d\mathfrak{m}_{\times}$ is the product measure induced by $d\mathfrak{m}$, $d:=\sup_{x_1\in A_1,\, x_2\in A_2}d_F(x_1,x_2)$ and
\[
C(N,K, d)= \sup_{0<\frac12r\leq s\leq r\leq d}\left(\frac{\mathfrak{s}_{K}(r)}{\mathfrak{s}_{K}(s)}\right)^{N-1}\leq 2^{N-1} e^{(N-1)\frac{\sqrt{|K|}d}{2}}.
\]
\end{theorem}

\subsection{Sobolev spaces and energy functionals}
Let $(M,F,d\mathfrak{m})$ be a compact  FMMM with or without boundary $\partial M$.  Define a   norm $\|\cdot\|_{\Ho}$ on $C^\infty(M)$ with respect to $d\mathfrak{m}$ by
\[
\|u\|_{\Ho}:=\|u\|_{L^2}+\|F^*(d u)\|_{L^2}=\left(\ds\int_M u^2 d\mathfrak{m}\right)^{\frac12}+\left(\ds\int_M F^{*2}(du) d\mathfrak{m}\right)^{\frac12}.
\]
Now set
\begin{align*}
H^1(M):=\overline{C^\infty(M)}^{\|\cdot\|_{\Ho}},\ \ \mathscr{X}:=\overline{C^\infty_0(M)}^{\|\cdot\|_{\Ho}},\ \ \mathscr{X}_0:=\left \{ u\in \mathscr{X} \ :  \ds\int_M u d\mathfrak{m} =0 \ \text{if} \ \partial M= \emptyset \right\}.
\end{align*}
Since $M$ is compact, both $H^1(M)$ and $\mathscr{X}$ are independent of the choices of $F$ and $d\mathfrak{m}$; in particular, $H^1(M)$ is the standard Sobolev space in the sense of Hebey \cite[Definition 2.1]{H}. However, when $M$ is not compact, $H^1(M)$ need not be even a vector space, see Krist\'aly and Rudas \cite{K-R}.

The {\it canonical energy functional} $E$ (i.e., Rayleigh quotient) on $\mathscr{X}\backslash\{0\}$ is defined as
\[
E(u):=
 \frac{\ds\int_MF^{*2}(du)d\mathfrak{m}}{\ds\int_Mu^2 d\mathfrak{m}}=\frac{\|F^*(du)\|_{L^2}^2}{\|u\|_{L^2}^2},\  \forall\,u\in \mathscr{X}\backslash\{0\}.\tag{2.10}\label{Enefunction}
\]
Given $u\in \mathscr{X}\backslash\{0\}$, for any $v\in \mathscr{X}$, we have
\begin{align*}
&DE(u) (v):=\langle v, \,DE(u)  \rangle:=\left.\frac{d}{dt}\right|_{t=0}E(u+tv)=-2\frac{\ds\int_Mv(\Delta u+E(u)\,  u)d\mathfrak{m}}{\ds\int_Mu^2 d\mathfrak{m}}.
\end{align*}
Hence, $DE(u)$ is a linear functional on $\mathscr{X}$. In particular,
 $DE(u)=0$   if and only if
\[
-\Delta u= E(u)u \text{ {in the weak sense}}.
\]

\begin{proposition}\label{srongDE} Let $(M,F,d\mathfrak{m})$ be a compact {\rm FMMM}. Then for any $u\in \mathscr{X}\backslash\{0\}$, $DE(u)$ is a bounded functional and  $u\mapsto DE(u)$
is continuous; hence,
$E\in C^1(\mathscr{X}\backslash\{0\})$.
\end{proposition}
\begin{proof}[Sketch of the proof]

Given $u\in \mathscr{X}\backslash\{0\}$, H\"older's inequality furnishes
\begin{align*}
\|DE(u)\|&=\sup_{v\neq0}\left|\frac{\langle v, \,DE(u) \rangle}{\|v\|_{\Ho}}\right|
\leq \,2\,\frac{\|F^*(dv)\|_{L^2}\|F(\nabla{u})\|_{L^2}+E(u)\cdot \|{u}\|_{L^2}\|v\|_{L^2}}{\|u\|^2_{L^2}\|v\|_{\Ho}}\\
\leq &\,2\,\frac{\max\{\sqrt{E(u)},\, E(u)\}}{\|u\|_{L^2}}.\tag{2.11}\label{v14-2.3}
\end{align*}
Moreover, due to Ge and Shen \cite[(11)]{GS}, a partition of unity argument yields  a  constant $C=C(M)>0$ depending only on $M$ such that
\[
\|F(\nabla u-\nabla v)\|_{L^2}\leq C\cdot \|F^*(d u-d v)\|_{L^2},\ \ \forall u,v\in H^1(M).\tag{2.12}\label{v14-2.4}
\]
Now a direct calculation together with (\ref{v14-2.3}) and (\ref{v14-2.4}) furnishes
\[
\lim_{k\rightarrow \infty}\|u_{k}-u\|_{\Ho}=0\Longrightarrow\lim_{k\rightarrow \infty}\|DE(u_k)-DE(u)\|=0,
\]
where $u_k\neq 0$.
Thus $DE(u)$ is continuous at $u$.
\end{proof}

Recall the following (P.-S.) condition.
\begin{proposition}[Ge and Shen \cite{GS}]\label{PScondition}Given any   $0<\delta<+\infty$,
if $\{u_k\}$ is a sequence in $\mathscr{X}\backslash\{0\}$ with
\[
 \|u_k\|_{L^2}=1,\ E(u_k)\leq \delta,\ \|DE(u_k)\|\rightarrow0,\tag{P.-S.}\label{P.-S.}
\]
then there exists a $($strongly$)$ convergent subsequence in $\mathscr{X}\backslash\{0\}$.
\end{proposition}

\begin{definition}\it
Given any eigenvalue  $\lambda\geq 0$, the {\rm eigenset} $\mathfrak{K}_{\lambda}$ corresponding to $\lambda$ is defined as
\[
\mathfrak{K}_{\lambda}:=\{u\in \mathscr{X}:\,  \|u\|_{L^2}=1,\, {E}(u)=\lambda, \, D {E}(u)=0\}.\tag{2.13}\label{4.10101}
\]
\end{definition}

\begin{lemma}\label{funddomain}
 $\mathfrak{K}_{\lambda}$  is  compact.
\end{lemma}
\begin{proof}  Given a sequence $\{u_k\}\subset \mathfrak{K}_{\lambda}$, the (\ref{P.-S.}) condition yields that a subsequence $\{u_{k_l}\}$ strongly converge to $u\in \mathscr{X}$. Now Proposition  \ref{srongDE} yields that
$\|u\|_{L^2}=1,\ E(u)= \lambda,\  DE(u)=0$, i.e., $ u\in \mathfrak{K}_{\lambda}$.
Hence, $\mathfrak{K}_{\lambda}$ is compact.
\end{proof}

In the sequel,  $X$ is called   a {\it Banach-Finsler manifold} if $X$ is a Finsler manifold in the sense of Palais (cf. Palais \cite[Definition 2.10, Definition 3.5]{PS} and Struwe \cite[p. 77]{S}); see also Definition \ref{PSFinslerde} (see Appendix \ref{B-Fsec}).

Now let $T\mathscr{X}$ denote the tangent bundle of $\mathscr{X}$ and let $\|\cdot\|$ be the trivial metric structure on $T\mathscr{X}$ induced by $\|\cdot\|_{\Ho}$. Thus,
$(\mathscr{X},\|\cdot\|)$ is a $C^\infty$-Banach-Finsler manifold. Let us introduce the set
\[
{\mathcal {S}}:=\{u\in \mathscr{X}:\,\|u\|_{L^2}=1\}.
\]
In the sequel, the set $\mathcal {S}$ will be our
main object of study rather than $\mathscr{X}$ or $\mathscr{X}_0$. First,
we have the following important result, whose proof will be given in Appendix \ref{B-Fsec}.
\begin{proposition}\label{compS}
 $(\mathcal {S},\|\cdot\|\,|_{T\mathcal {S}})$ is a complete $C^\infty$-Banach-Finsler  manifold and an {\rm AR} $($i.e., absolute retract$)$. Moreover, $i^*E$ is a $C^1$-function on $\mathcal {S}$, where $i:\mathcal {S}\hookrightarrow \mathscr{X}$ is the inclusion.
\end{proposition}

The following lemma is based on the homogeneity of $E$.
\begin{lemma}\label{criciequation}
 A function $u\in \mathcal {S}$ is a critical point of $E$ if and only if $u$ is a critical point of $i^*E$, where $i:\mathcal {S}\hookrightarrow \mathscr{X}$ is the inclusion. In particular, either $u=\pm(\mathfrak{m}(M))^{-\frac12}$ or $u\in \mathscr{X}_0\backslash\{0\}$.
\end{lemma}

\begin{remark}\rm
 Ge and  Shen \cite{GS} proved that if $DE(u)=0$, then $u\in C^{1,\alpha}(M)$ for some $0<\alpha<1$.
\end{remark}

According to Lemma \ref{criciequation}, there is no difference between $E$ and $i^*E$ from the point of view of critical points in $\mathcal {S}$; so by abuse of notation, we will  use $E$ to denote $i^*E$ in the rest of paper.

A standard argument concerning pseudo-gradient vector fields together with Propositions \ref{PScondition} and \ref{compS}  yields the following  result; we omit its proof since it is the same as  Struwe \cite[Chapter II, Theorem 3.11]{S}.
\begin{lemma}[Homotopy  Lemma]\label{maintheorem2-0}Let $(M,F,d\mathfrak{m})$ be a compact {\rm FMMM}.
Let $\lambda\geq 0$, $\epsilon>0$ and  let $O\subset \mathcal {S}$ be any open neighborhood of the eigenset
$\mathfrak{K}_{\lambda}$ $($see $(\ref{4.10101})).$ Then there exist a number $\epsilon_0\in (0,\epsilon)$ and a continuous $1$-parameter family
of homeomorphisms $ \Phi(\cdot,t)$ of $\mathcal {S}$, $0 \leq t < +\infty$, with the following properties$:$

\begin{itemize}
	\item[(i)] $\Phi(u,t) = u$, if one of the following conditions hold
	\[
	\text{\rm (1)}\, t=0;\ \ \text{\rm (2)}\, DE(u)=0;\ \ \text{\rm (3)}\, |E(u) - \lambda| \geq \epsilon;
	\]
	\item[(ii)] $t\mapsto E(\Phi(u, t))$ is non-increasing  for every $u\in \mathcal {S};$
	\item[(iii)] $\Phi(E_{\lambda+\epsilon_0}\backslash O,1)\subset E_{\lambda-\epsilon_0}$, and $\Phi(E_{\lambda+\epsilon_0}, 1) \subset E_{\lambda-\epsilon_0}\cup O$, where   $E_\delta:=\{u\in \mathcal {S}:\, E(u)<\delta\}$, $\delta>0;$
	\item[(iv)] $\Phi(-u,t)=-\Phi(u,t)$ for every $t\geq 0$ and $u\in \mathcal {S};$
	\item[(v)] $\Phi:\mathcal {S}\times[0,\infty)\rightarrow \mathcal {S}$ has the semi-group property, i.e.,
	$
	\Phi(\cdot,s)\circ \Phi(\cdot,t)=\Phi(\cdot,s+t)$ for every $s,t\geq 0.$
\end{itemize}
\end{lemma}

\section{dimension pairs and eigenvalues}\label{section3}

\subsection{Spectrum of a dimension pair}
Since the Laplacian of a non-Riemannian Finsler manifold is nonlinear (cf. \cite{GS,Sh1}),
it is impossible to define the higher order eigenvalues by the traditional way.
Inspired by Gromov \cite{G}, we carry out a systematic study of eigenvalues by dimension-like functions. In addition, our results complement in several aspects those obtained in Riemannian geometry.

\medskip

\noindent\textbf{Notations.} We will use the following notations throughout the paper:

\begin{itemize}
		\item[(1)] $\mathbb{R}^+:=[0,\infty)$, $\mathbb{N}:=\{0,1,2,\ldots\}$ and $\mathbb{N}^+:=\{1,2,\ldots\}$;
		\item[(2)] $\Dim_C(\cdot)$ denotes the Lebesgue covering dimension (cf. Hurwicz and Wallman \cite{HW});
		\item[(3)] A  homeomorphism $h:\mathcal {S}\rightarrow \mathcal {S}$ is called an  APH (i.e., antipode preserving homeomorphism) whenever $h$ satisfies $h(-u)=-h(u)$ for all $u\in \mathcal {S}$;
		\item[(4)] 	Given a compact FMMM $(M,F, d\mathfrak{m})$, for  any $u,v\in L^2(M)$ we set
		\[
		(u,v)_{L_2}:=\ds\int_M u\,v d\mathfrak{m},\ \ \ \ \|u\|_{L^2}:=\ds\int_M  u^2d\mathfrak{m}.\tag{3.1}\label{innernormL2}
		\]
\end{itemize}

Now we introduce the notion of dimension pairs.

\begin{definition}\label{defdim1}\it An {\rm optional family} $\mathscr{C}$ is a  collection of subsets of $\mathcal {S}$ satisfying the following conditions:

\begin{itemize}
	\item[(i)] $\emptyset\in \mathscr{C};$
	\item[(ii)] Given $k\in \mathbb{N}^+$, for any $k$-dimensional vector subspace $V\subset \mathscr{X}$, one has $V\cap \mathcal {S}\in \mathscr{C};$
	\item[(iii)] For every APH $h:\mathcal {S}\rightarrow \mathcal {S}$, $h(A)\in \mathscr{C}$ for all $A\in \mathscr{C}$.
\end{itemize}
Given an optional family $\mathscr{C}$, a {\rm dimension-like function  $\Dim:\mathscr{C}\to \mathbb{N}\cup\{ +\infty\}$}
satisfies the following  conditions:
\begin{itemize}
	\item[(D1)] {\rm $\Dim$}$(A)\geq 0$ for any $A\in \mathscr{C}$ with equality if and only if $A=\emptyset;$
	\item[(D2)] For any $A_1,A_2\in \mathscr{C}$ with $A_1\subset A_2$, $ \Dim(A_1)\leq \Dim(A_2);$
	\item[(D3)] Given $k\in \mathbb{N}^+$, for any $k$-dimensional vector subspace $V\subset \mathscr{X}$, $\Dim(V\cap \mathcal {S})\geq k;$
	\item[(D4)] For every APH $h:\mathcal {S}\rightarrow \mathcal {S}$, $\Dim(h(A))=\Dim(A)$ for all $A\in \mathscr{C}$.
\end{itemize}
 $(\mathscr{C},\Dim)$ is a {\rm   dimension pair}, if $\mathscr{C}$ is an  optional family and $\Dim$ is a   dimension-like function on $\mathscr{C}$.
 \end{definition}

\begin{remark}\rm Since the inverse of an APH is still an APH,
(D4) is equivalent to the following:

\smallskip

\noindent (D4') For every APH $h:\mathcal {S}\rightarrow \mathcal {S}$, $\Dim(h(A))\geq\Dim(A)$ for all $A\in \mathscr{C}$.

\end{remark}

The spectrum for a dimension pair is defined as follows.
\begin{definition}\label{spect2}\it
Let $(M,F, d\mathfrak{m})$ be a compact {\rm FMMM}.
Given a dimension pair $(\mathscr{C},\Dim)$, the corresponding {\rm eigenvalues} are defined as
\[
\lambda_k:=\sup\left\{  \lambda\in \mathbb{R}^+\cup\{+\infty\}:\, \Dim E^{-1}[0,\lambda]<k \right\},\ \forall\,k\in \mathbb{N}^+,
\]
where
\[
\Dim E^{-1}[0,\lambda]:=\sup\left\{ \Dim(A):\, A\in \mathscr{C},\ A\subset E^{-1}[0,\lambda]\right\}.
\]
The collection $\{\lambda_k\}_{k=1}^\infty$ is called the $(\mathscr{C},\Dim)$-{\rm spectrum}.
\end{definition}

\begin{remark}\rm
In \cite{G}, Gromov defined a dimension-like function $\Dim$ as   a function on a collection of sets $\mathscr{C}$ only satisfying Property (D2).
In this paper, we require that both  an optional family $\mathscr{C}$ and a dimension-like function $\Dim$ satisfy   further properties which provide qualitative properties of the $(\mathscr{C},\Dim)$-spectrum.
\end{remark}

First we have the following min-max principle.
\begin{theorem}[Min-max Principle]\label{minmaxth} Let $(M,F, d\mathfrak{m})$ be a compact  {\rm FMMM}. Given a dimension pair $(\mathscr{C},\Dim)$,  set
\[
\mathscr{C}_k:=\{A\in \mathscr{C}:\, \Dim\,(A)\geq k\},\ \forall\,k\in \mathbb{N}^+.
\]
Then the corresponding eigenvalue satisfies the min-max principle, i.e.,
\[
\lambda_k=\underset{A\in \mathscr{C}_k}{\inf}\underset{u\in A}{\sup}{E(u)},\ \forall\,k\in \mathbb{N}^+.
\]
In particular, $\lambda_k$ is finite for every $k \in \mathbb{N}^+$.
\end{theorem}
\begin{proof} For any $k\in \mathbb{N}^+$,
 Definition \ref{defdim1}  implies $\mathscr{C}_k\neq\emptyset$. Hence $\hat{\lambda}_k:={\inf}_{A\in \mathscr{C}_k}{\sup}_{u\in A}{E(u)}$ is well-defined. We  show first that $\lambda_k=\hat{\lambda}_k$.
In fact, if $\lambda \in \mathbb{R}^+\cup\{+\infty\}$ satisfies $\Dim E^{-1}[0,\lambda]\geq k$, then Definition \ref{spect2} yields
 $\lambda_k\leq \lambda$, which implies
 \begin{align*}
  \lambda_k&\leq \inf\left\{\lambda\in \mathbb{R}^+\cup\{+\infty\}:\, \Dim E^{-1}[0,\lambda]\geq k\right\}\\
  &=\inf\left\{\lambda\in \mathbb{R}^+\cup\{+\infty\}:\,\exists A\in  \mathscr{C}\text{ with }A\subset E^{-1}[0,\lambda] \text{ and } \Dim(A)\geq k\right\}\leq \hat{\lambda}_k.
  \end{align*}
If $\lambda_k=+\infty$, then clearly $\lambda_k\geq \hat{\lambda}_k$. Now suppose
$\lambda_k<+\infty$. Thus, for any $\epsilon>0$, Definition \ref{spect2} furnishes $\Dim E^{-1}[0,\lambda_k+\epsilon]\geq k$, i.e., there exists $A\subset E^{-1}[0,\lambda_k+\epsilon]$ with $A\in \mathscr{C}_k$, which implies $\hat{\lambda}_k={\inf}_{A\in \mathscr{C}_k}{\sup}_{u\in A}{E(u)} \leq \lambda_k+\epsilon$. The arbitrariness of $\epsilon>0$ implies $\hat{\lambda}_k\leq \lambda_k$, thus  $\lambda_k=\hat{\lambda}_k$.

We now prove that $\lambda_k$ is finite. Let $\hat{g}$ be the average Riemannian metric induced by $F$, see (\ref{averageRieman}).
 Denote by $(\cdot,\cdot)$ and $\|\cdot\|_1$  the standard inner product and norm on $H^{1}(M)$  induced by $\hat{g}$, respectively, i.e.,
\[
(u,v):=\ds\int_M uv \,d\vol_{\hat{g}}+\ds\int_M  {\hat{g}}(du,dv)\,d\vol_{\hat{g}},\ \ \ \ \|u\|_1:=\sqrt{(u,u)}.\tag{3.2}\label{3.1}
\]
Since $M$ is compact,  the topology of $(\mathscr{X},\|\cdot\|_1)$ coincides with the one of $(\mathscr{X},\|\cdot\|_{\Ho})$; in particular,  $E$ is continuous in the topology of $(\mathscr{X}, \|\cdot\|_1)$.

Let $\{\lambda^{\Delta_{\hat{g}}}_i\}_{i=1}^\infty$ be the usual spectrum of the Beltrami-Laplacian $\Delta_{\hat{g}}$ and $\{f_i\}_{i=1}^\infty$ be the corresponding eigenfunctions with $\|f_i\|_{L^2}=1$. According to  Craioveanu, Puta and Rassias \cite[p.134]{C4}, for any $u\in \mathscr{X}$, there exist a sequence of constants $\{a_i\}$ such that $u=\sum_{i=1}^\infty a_i f_i$ with
\[
\|u\|^2_{L^2}=\sum_{i=1}^\infty a^2_i,\ \|u\|^2_1=\sum_{i=1}^\infty (1+\lambda^{\Delta_{\hat{g}}}_i)a^2_i<+\infty.\tag{3.3}\label{new new 5.2}
\]
Now set $V:=\text{Span}\{f_1,\ldots,f_k\}$. Due to (\ref{new new 5.2}),
 $V\cap \mathcal {S}\subset \mathscr{X}$ is compact in  $(\mathscr{X}, \|\cdot\|_1)$. Since $V\cap \mathcal {S}\in \mathscr{C}_k$, see (D3), the min-max characterization  furnishes
$
\lambda_k=\underset{A\in \mathscr{C}_k}{\inf}\underset{u\in A}{\sup}{E(u)}\leq \sup_{u\in V\cap \mathcal {S}}E(u)<+\infty.
$
\end{proof}
\begin{remark}\rm \label{infeigen}
If $\Dim$ does not satisfy (D3) in Definition \ref{defdim1}, $\mathscr{C}_k$ could be empty, in which case $\lambda_k=+\infty$.
\end{remark}

\begin{theorem}\label{properties1} Let $(M,F, d\mathfrak{m})$ be a compact {\rm FMMM}. Given a dimension pair $(\mathscr{C},\Dim)$, the corresponding spectrum $\{\lambda_k\}_{k=1}^\infty$ satisfy the following properties$:$

\begin{itemize}
	\item[(i)]  \textbf{\rm (\bf Monotonicity)}
	\begin{align*}
	0= \lambda_1\leq\lambda_2\leq  \ldots \leq \lambda_k\leq\ldots, &\text{ if }\partial M=\emptyset;\\
	0<\lambda_1\leq\lambda_2\leq  \ldots \leq \lambda_k\leq\ldots, &\text{ if }\partial M\neq\emptyset.
	\end{align*}
	In particular, the first eigenvalue is
	\[
	\lambda_1=\inf_{u\in \mathcal {S}} E(u)=\inf_{u\in \mathscr{X}\backslash\{0\}}E(u).
	\]
	\item[(ii)] \textbf{\rm (\bf Riemannian case)} If $F$ is Riemannian and $d\mathfrak{m}$ is the canonical Riemannian measure, then
	\[
	\lambda_k\leq \lambda^\Delta_k,\ \forall\,k\in \mathbb{N}^+,
	\]
	where $\lambda_k^\Delta$ is the usual $k^{\text{th}}$-eigenvalue of the Beltrami-Laplacian $\Delta$ in the Riemannian case.
	\item[(iii)] \textbf{\rm (\bf Existence of eigenfunction)}  For each
	$k\in \mathbb{N}^+$,
	the eigenfunction $u$ corresponding to the eigenvalue $\lambda_k$ always exists, i.e., there exist $u\in \mathscr{X}\backslash\{0\}$ with
	$\Delta u+\lambda_k u=0 \text{ in the weak sense}$.
	In particular, the eigenfunction $u$ satisfies
	\[
	\left\{
	\begin{array}{llll}
	u=\text{const}.\neq0,\ &\text{ if }\lambda_k=0;\\
	\\
	u\in \mathscr{X}_0\backslash\{0\},\ &\text{ if }\lambda_k>0.
	\end{array}
	\right.
	\]
\end{itemize}

\end{theorem}
\begin{proof}
{(i)} For convenience, set
$
\lambda_1^*:=\inf_{u\in \mathcal {S}}E(u)=\inf_{u\in \mathscr{X}\backslash\{0\}} E(u).
$
We claim $\lambda_1=\lambda_1^*$.
First, Theorem \ref{minmaxth} implies $\lambda_1^*\leq \lambda_1$. Furthermore, for each $f\in \mathcal {S}$, we have $A_0:=\{\pm f\} \subset\mathbb{R}f\cap \mathcal {S}\in  \mathscr{C}_1$, which together with the min-max principle yields $\lambda_1=\inf_{A\in \mathscr{C}_1}\sup_{u\in A}E(u)\leq \sup_{u\in A_0}E(u)= E( f)$. Taking the infimum of the right hand side when $f\in \mathcal {S}$, it turns out that
 $\lambda_1\leq\lambda_1^*$.

In the sequel, we study the positivity of $\lambda_1$.
If $\partial M=\emptyset$,
set $A=\{\pm(\mathfrak{m}(M))^{-\frac12}\}\in \mathscr{C}_1$. Thus,
$0\leq \lambda_1\leq \sup_{u\in A}E(u)=0$.
Now suppose  $\partial M\neq\emptyset$.
Let $\hat{g}$ be the average Riemannian metric induced by  $F$. Since $M$ is compact,  there exists a positive constant $C_\mathfrak{m}\geq 1$ such that
\[
C_{\mathfrak{m}}^{-1}\cdot d\vol_{\hat{g}}\leq d\mathfrak{m}\leq C_{\mathfrak{m}}\cdot d\vol_{\hat{g}},\tag{3.4}\label{v14-3.3}
\]
which together with (\ref{averagenorm}) and
the  spectral theory in Riemannian geometry   yields
\[
\lambda_1=\lambda_1^*=  \inf_{u\in \mathscr{X}\backslash\{0\}}  \frac{\ds\int_M F^{*2}(du)d\mathfrak{m}}{\ds\int_M u^2 d\mathfrak{m}}  \geq \frac{1}{\Lambda_F C_\mathfrak{m}^2}\inf_{u\in \mathscr{X}\backslash\{0\}}\frac{\ds\int_M\hat{g}(du,du)d\vol_{\hat{g}}}{\ds\int_Mu^2d\vol_{\hat{g}}}>0.
\]
Since $\mathscr{C}_{k+1}\subset\mathscr{C}_{k}$, the monotonicity of the eigenvalues  follows by Theorem \ref{minmaxth}.

\medskip

{(ii)}
If $F$ is Riemannian, Courant's minimax principle yields
\[
\lambda^\Delta_k=\min_{V\in \mathscr{H}_k}\max_{u\in V\backslash\{0\}}E(u),\tag{3.5}\label{new4.1}
\]
where
$\mathscr{H}_k=\{V\subset \mathscr{X}:\, V\text{ is a linear subspace with }\Dim_C(V)=k\}$.
In particular, for any $\epsilon>0$, there exists a linear space $V$ with $\Dim_C(V)=k$ and  $\max_{u\in V\backslash\{0\}} E(u)<\lambda^\Delta_k+\epsilon$.
Since $\mathcal {S}\cap V\in \mathscr{C}_k$, the min-max principle furnishes
$
\lambda_k\leq \sup_{u\in \mathcal {S}\cap V}E(u)<\lambda^\Delta_k+\epsilon$. The arbitrariness of $\epsilon>0$ implies that $ \lambda_k\leq \lambda^\Delta_k.
$

\medskip

 {(iii)}
We claim that each $\lambda_k$ is a critical value of $E$.
Assume the contrary that  $\lambda_k$ is a regular value, i.e., if $u\in \mathcal {S}$ with $E(u)=\lambda_k$, then $DE(u)\neq0$. Accordingly, the eigenset $\mathfrak{K}_{\lambda_k}$ is empty (cf.\,(\ref{4.10101})). Due to Lemma \ref{maintheorem2-0} ($O=\emptyset$ and $\epsilon=1$),  there exists $\epsilon_0>0$ and a family of APH's $\Phi(\cdot,t):\mathcal {S}\rightarrow \mathcal {S}$, $t\in [0,1],$ such that $\Phi(E_{\lambda_k+\epsilon_0},1)\subset E_{\lambda_k-\epsilon_0}$. For this $\epsilon_0>0$, Theorem \ref{minmaxth} yields an element $A\in \mathscr{C}_k$ with $A\subset E_{\lambda_k+\epsilon_0}$, therefore, $E(\Phi(w,1))< \lambda_k-\epsilon_0$ for every $w\in A$.

By (D4) in Definition \ref{defdim1} one has ${\Phi(A,1)}\in \mathscr{C}_k$ which together with
 Theorem \ref{minmaxth} implies
\[
\lambda_k\leq \sup_{u\in { \Phi(A,1)}}E(u)\leq \lambda_k-\epsilon_0<+\infty,
\]
a contradiction. Therefore, the eigenfunction $u\in \mathcal {S}$ corresponding to $\lambda_k$ does exist; in particular, by Lemma \ref{criciequation} it follows that $u=\text{const}.\neq0$ or $u\in \mathscr{X}_0\backslash\{0\}$.
\end{proof}

\begin{remark}\label{remark-kesobb}\rm According to Chavel \cite[p.9]{Ch}, for a closed Riemannian manifold one has
\[
0=\lambda_1^\Delta<\lambda_2^\Delta\leq \ldots \leq \lambda_k^\Delta\leq \ldots,
\]
in which case  {\it the first eigenvalue} in the classical literature usually means the   first  \textit{positive} eigenvalue, i.e., $\lambda_2^\Delta$. On the other hand, it is easy to check that
\[
\lambda_1^\Delta=0=\inf_{u\in H^{1}(M)\backslash\{0\}}\frac{\ds\int_Mg(\nabla u,\nabla u)d\vol_g}{\ds\int_M u^2d\vol_g}=\inf_{u\in \mathcal {S}}E(u).
\]
Therefore,  Theorem \ref{properties1}/(i) holds in the Riemannian case.
\end{remark}

  Theorem \ref{properties1} implies in particular that
for a compact Riemannian manifold equipped with the canonical Riemannian measure, each eigenvalue of a dimension pair $(\mathscr{C},\Dim)$ is a standard eigenvalue of the Beltrami-Laplacian operator.
 However,  $(\mathscr{C},\Dim)$-spectrum may not contain all the critical values of $E$, see subsection \ref{lebsdiem}.
It should be also remarked that there are dimension pairs such that $0=\lambda_k<\lambda^\Delta_k$,  $k\geq 2$, for every closed Riemannian manifold, see Proposition \ref{flawex1}. In order to avoid such a case, we introduce a "stronger" notion of dimension pairs.

\begin{definition}\label{faithful}\it A dimension pair $(\mathscr{C},\Dim)$  is said to be {\rm faithful}  if
\[
\lambda_k=\lambda_k^\Delta,\ \forall\,k\in \mathbb{N}^+,
\]
 for any compact Riemannian manifold $(M,g)$ endowed with its canonical Riemannian measure $d\vol_g;$
here, $\lambda_k$ is from Definition {\rm \ref{spect2}} considered for the manifold $(M,\sqrt{g},d\vol_{{g}})$, while   $\lambda_k^\Delta$ stands for the usual eigenvalue of the Beltrami-Laplacian $\Delta$ in the Riemannian setting.
\end{definition}

\begin{theorem}\label{properties2}
Let $(M,F,d\mathfrak{m})$ be a compact {\rm FMMM}. For a faithful dimension pair $(\mathscr{C},\Dim)$, the corresponding spectrum satisfies$:$
\begin{itemize}
	\item[(i)] The first {positive} eigenvalue is equal to
	\[
	\left\{
	\begin{array}{llll}
	\lambda_2&=\inf_{u\in \mathscr{X}_0\backslash\{0\}}E(u),\ \text{ if }\partial M=\emptyset;\\
	\lambda_1&=\inf_{u\in \mathscr{X}_0\backslash\{0\}} E(u),\ \text{ if }\partial M\neq\emptyset;
	\end{array}
	\right.
	\]
	\item[(ii)] $\underset{k\rightarrow\infty}{\lim}\lambda_k=+\infty;$
	\item[(iii)] The multiplicity  of each $\lambda_k$ is finite.
\end{itemize}
\end{theorem}
\begin{proof}
Let $\hat{g}$ be the average Riemannian metric induced by $F$ and $k\in \mathbb N^+$. Since $(\mathscr{C},\Dim)$ is a faithful dimension pair, the usual eigenvalue ${\lambda}^{\Delta_{\hat{g}}}_k$ of $(M,\sqrt{\hat g},d\vol_{\hat{g}})$ is equal to
\[
\tilde \lambda^{{\hat{g}}}_k=\inf_{A\in \mathscr{C}_k}\sup_{u\in A}\frac{\ds\int_M \hat{g}(du,du)d\vol_{\hat{g}}}{\ds\int_M u^2 d\vol_{\hat{g}}}.
\]
The latter fact together with (\ref{averagenorm}) and (\ref{v14-3.3}) implies that
$$\lambda_k=\underset{A\in \mathscr{C}_k}{\inf}\underset{u\in A}{\sup}\frac{\ds\int_M F^{*2}(du)d\mathfrak{m}}{\ds\int_M u^2 d\mathfrak{m}} \geq \frac{1}{\Lambda_F\cdot C_\mathfrak{m}^2}\tilde \lambda^{{\hat{g}}}_k=\frac{1}{\Lambda_F\cdot C_\mathfrak{m}^2}{\lambda}^{\Delta_{\hat{g}}}_k.$$
Hence,  $\lambda_k>0$ for $k\geq 2$ and $\ds\lim_{k\rightarrow \infty}\lambda_k=+\infty$ follow from the spectral theory in Riemannian geometry. Since $\lambda_k<+\infty$ for every $k\in \mathbb N^+$ (see Theorem \ref{minmaxth}), the latter limit implies the finiteness of the multiplicity issue; thus properties (ii) and (iii) are verified.

Now we show (i). If $\partial M\neq\emptyset$, Theorem \ref{properties1}/(i) together with $\mathscr{X}=\mathscr{X}_0$ yields
\[
0<\lambda_1=\inf_{u\in \mathcal {S}}E(u)=\inf_{u\in \mathscr{X}\backslash\{0\}}E(u)=\inf_{u\in \mathscr{X}_0\backslash\{0\}} E(u).
\]
When $\partial M=\emptyset$, we recall that
 $\lambda_2>0$.
Thus Theorem \ref{properties1}/(iii) yields an eigenfunction $f\in \mathscr{X}_0\backslash\{0\}$ corresponding to $\lambda_2$. In particular,
 $\lambda_2=E(f)\geq \inf_{u\in \mathscr{X}_0\backslash\{0\}}E(u)$.
 On the other hand, for each $u\in \mathscr{X}_0\backslash\{0\}$, set $V_u=\text{Span}\{1,\, u\}$. Since $A_u:=\mathcal {S}\cap V_u\in \mathscr{C}_2$ and $\ds\int_Mud\mathfrak{m}=0$, it turns out that
\[
\lambda_2\leq \sup_{v\in A_u}E(v)=\sup_{(a,b)\neq (0,0)}\frac{\ds\int_M F^{*2}(d(a+bu))d\mathfrak{m}}{\ds\int_M (a+bu)^2 d\mathfrak{m}}=\sup_{(a,b)\neq (0,0)}\frac{b^2\ds\int_M F^{*2}(du))d\mathfrak{m}}{a^2\mathfrak{m}(M)+b^2 \ds\int_M u^2 d\mathfrak{m}}=E(u).
\]
Therefore, $ \lambda_2\leq \inf_{u\in \mathscr{X}_0\backslash\{0\}}E(u),$
which concludes the proof.
\end{proof}

\vspace{-0.2cm}

\begin{proof}[Proof of Theorem {\rm \ref{th1.1}}]
 Theorem \ref{th1.1} directly follows by Theorems \ref{properties1} and \ref{properties2}, respectively.
\end{proof}

\subsection{Examples of dimension pairs}\label{examplest}
In this subsection we present
 some faithful dimension pairs for which Theorem \ref{properties2} applies.
First, we introduce some notions and notations.

\medskip

  Let $\mathbb{P}(\mathscr{X})$ be the quotient space $\mathcal {S}/\mathbb{Z}_2$. Thus, $\mathfrak{p}:\mathcal {S}\rightarrow \mathbb{P}(\mathscr{X})$ is a $2$-fold covering as $\mathbb{Z}_2$ acts freely and properly discontinuously on $\mathcal {S}$; in particular, $\mathbb{P}(\mathscr{X})$ is a normal ANR (see Proposition  \ref{finalPX}). The following result is trivial.

\begin{proposition}\label{projesec}
$\mathbb{P}(\mathscr{X})$ is homeomorphic to the projective space $(\mathscr{X}\backslash\{0\})/\sim$, where $u\sim v$ if and only if there exists $\mu\neq 0$ such that $u=\mu\cdot v$.
\end{proposition}

Given a $k$-dimensional linear  subspace $V$ of $\mathscr{X}$, $\mathbb{P}(V):=\mathfrak{p}(V\cap \mathcal {S})$ is also used to denote the projective space  induced by $V$. All the maps in this subsection are assumed to be continuous.

\subsubsection{Lusternik-Schnirelmann dimension pair}\label{LSDP}
${}$ \\
In this subsection we  construct two  dimension pairs  by means of the Lusternik-Schnirelmann category. First, we recall the relative Lusternik-Schnirelmann (LS) category on $\mathbb{P}(\mathscr{X})$ (cf. \cite{CLOT,F,S}).
\begin{definition}\label{defls}\it
Given a  subset $A\subset \mathbb{P}(\mathscr{X})$,   the {\rm LS category of $A$  relative to} $\mathbb{P}(\mathscr{X})$, $\cat_{\mathbb{P}(\mathscr{X})}(A)$,  is the smallest possible integer value $k$ such that
 $A$ is covered
by $k$ closed sets $A_j$, $1 \leq j \leq k$, which are contractible in $\mathbb{P}(\mathscr{X})$. If no such finite covering exists we write $\cat_{\mathbb{P}(\mathscr{X})}(A)=+\infty$.
\end{definition}

\begin{definition}\label{defls2}\it
Define two optional families $\mathscr{C}^\alpha$, $\alpha=1,2$ by
\[
\mathscr{C}^1:=\{A\subset \mathcal {S}: A \text{ is closed}\},\ \mathscr{C}^2:=\{A\subset \mathcal {S}: A \text{ is compact}\}.
\]
Given a closed set $A\subset \mathcal {S}$, the {\rm Lusternik-Schnirelmann  dimension of} $A$ is defined by
\[
\Dim_{LS}(A):=\cat_{\mathbb{P}(\mathscr{X})}(\mathfrak{p}(A)),
\]
where $\mathfrak{p}:\mathcal {S}\rightarrow \mathbb{P}(\mathscr{X})$ is the natural projection.
\end{definition}

\begin{remark}\rm
Since $\mathcal {S}$ is contractible (see Proposition  \ref{contr}), it is unsuitable to use the  LS category relative
to $\mathcal {S}$ to define dimension pairs.
\end{remark}

\begin{proposition}\label{lsdimp}
For each $\alpha\in\{1,2\}$,
$(\mathscr{C}^\alpha, \Dim_{LS})$ is a dimension pair.
\end{proposition}
\begin{proof}Given any $\alpha\in \{1,2\}$, we have to show that $(\mathscr{C}^\alpha, \Dim_{LS})$ satisfies properties (D1)-(D4) in Definition \ref{defdim1}.  (D1) and (D2) clearly follow by
Definitions \ref{defls} and \ref{defls2}. Given a $k$-dimensional linear space $V$, since $\mathfrak{p}(V\cap \mathcal {S})=\mathbb{P}(V)$, one has $\Dim_{LS}(V\cap \mathcal {S})=\cat_{\mathbb{P}(\mathscr{X})}(\mathbb{P}(V))=k$, which implies (D3).
Moreover, each APH $h:\mathcal {S} \rightarrow \mathcal {S}$ induces a homeomorphism  $H:\mathbb{P}(\mathscr{X})\rightarrow \mathbb{P}(\mathscr{X})$, i.e., $H([u]):=\mathfrak{p}\circ h(u)$.
Since $\cat_{\mathbb{P}(\mathscr{X})} (\cdot)$ is invariant under homeomorphism (cf. Cornea,  Lupton,
 Oprea and
Tanr\'e \cite[Lemma 1.13/(5)]{CLOT}), one gets
\[
\Dim_{LS}( h(A))=\cat_{\mathbb{P}(\mathscr{X})}(\mathfrak{p}(h(A)))=\cat_{\mathbb{P}(\mathscr{X})} (H(\mathfrak{p}(A)))=\cat_{\mathbb{P}(\mathscr{X})} (\mathfrak{p}(A))=
\Dim_{LS}(A),\ \forall\,A\in \mathscr{C}^\alpha,
\]
which proves property (D4).
\end{proof}

Let $(M,F, d\mathfrak{m})$ be a compact {\rm FMMM} and let $\alpha\in\{1,2\}$.
According to Theorem \ref{minmaxth},
the $k^{th}$ eigenvalue of $(\mathscr{C}^\alpha, \Dim_{LS})$, denoted by $\lambda_k^{LS,\alpha}$, is
\[
\lambda_k^{LS,\alpha}=\inf_{A\in \mathscr{C}^{LS,\alpha}_k}\sup_{u\in A}E(u),
\]
where $\mathscr{C}^{LS,\alpha}_k:=\left\{A\in \mathscr{C}^\alpha:\, \Dim_{LS}(A)\geq k\right\}$. The collection $\{\lambda_k^{LS,\alpha}\}_{k=1}^\infty$ is called the $(\mathscr{C}^\alpha, \Dim_{LS})$-spectrum.


\begin{lemma}\label{muti2}Let $(M,F, d\mathfrak{m})$ be a compact {\rm FMMM}. Given $\alpha\in\{1,2\}$,
if for some $k\in \mathbb{N}^+$,
\[
0\leq \lambda^{LS,\alpha}_k=\lambda^{LS,\alpha}_{k+1}=\cdots=\lambda^{LS,\alpha}_{k+l-1}=\lambda,
\]
i.e., the multiplicity of the
eigenvalue $\lambda$ is $l$, then $\Dim_{LS}(\mathfrak{K}_{\lambda}) \geq l$ $($see $(\ref{4.10101})).$ In particular, there exist at least $l$ linearly independent eigenfunctions corresponding to the eigenvalue $\lambda$.
Moreover, if $l>1$, then $\mathfrak{K}_{\lambda}$ is an infinite set.
\end{lemma}

The proof of Lemma \ref{muti2} will be postponed after Theorem \ref{therealconnection}; this lemma furnishes the following important result.

\begin{theorem}\label{compareRiemannen}
For each $\alpha\in\{1,2\}$,
$(\mathscr{C}^\alpha, \Dim_{LS})$ is a faithful dimension pair.
\end{theorem}
\begin{proof}Let $(M,g,d\vol_g)$ be a compact Riemannian manifold endowed with its canonical measure.
Fix $\alpha\in\{1,2\}$ and $k\in \mathbb{N}^+$ arbitrarily. Due to Theorem \ref{properties1}/(ii), it suffices to show $\lambda_k^\Delta\leq \lambda^{LS,\alpha}_k$.

Theorem \ref{properties1}/(iii) together with the spectral theory in Riemannian geometry implies that for each $j$ with $1\leq j \leq k$, there exists $u_j\in   C^\infty(M)$ such that $E(u_j)=\lambda_j^{LS,\alpha}$ and $-\Delta u_j=\lambda_j^{LS,\alpha} u_j$ in the weak sense.
  If $\lambda^{LS,\alpha}_i\neq\lambda^{LS,\alpha}_j$, then
\[
\ds\int_M u_i\, u_j\,d\vol_g=0,\ \ds\int_M g(\nabla u_i,\nabla u_j)  d\vol_g=0.\tag{3.6}\label{4.1}
\]
If  the multiplicity of the eigenvalue $\lambda$ is $l$, Lemma \ref{muti2} provides at least $l$ linearly independent eigenfunctions $\{u_s\}_{s=1}^l$ corresponding to $\lambda$, which still satisfy (\ref{4.1}) (since $\Delta$ is linear). Accordingly, one always obtains $k$ eigenfunctions $\{u_j\}_{j=1}^k$ such that they are mutually orthogonal (in the sense of (\ref{4.1})) and $E(u_j)=\lambda^{LS,\alpha}_j$.

Now let
$V_k:=\text{Span}\{u_1,\ldots,u_k\}\subset \mathscr{X}$.
Thus, $\Dim_C(V_k)= k$ and then Courant's minimax principle (\ref{new4.1})  together with (\ref{4.1}) yields
\begin{eqnarray*}
\lambda_k^\Delta&\leq&\sup_{u\in V_k\backslash\{0\}}E(u)=\sup_{(a_1,...,a_k)\neq 0_{\mathbb R^k}}\frac{{\ds\sum_{i=1}^k}a_i^2\ds\int_M g(\nabla u_i,\nabla u_i)d\vol_g}{{\ds\sum_{i=1}^k}a_i^2\ds\int_M u_i^2d\vol_g}=\sup_{(a_1,...,a_k)\neq 0_{\mathbb R^k}}\frac{{\ds\sum_{i=1}^k}a_i^2\lambda_i^{LS,\alpha}\ds\int_M  u_i^2d\vol_g}{{\ds\sum_{i=1}^k}a_i^2\ds\int_M u_i^2d\vol_g}\\&\leq&\lambda^{LS,\alpha}_k,
\end{eqnarray*}
which concludes the proof.
\end{proof}

\subsubsection{Krasnoselskii dimension pair}
${}$ \\
We now  use   the  Krasnoselskii genus to construct dimension pairs. We also refer to  Ambrosio,  Honda and Portegies \cite{A} for the  spectrum   defined on $L^2(M)$ by the  Krasnoselskii genus where the Cheeger energy is used instead of the Rayleigh quotient.
According to \cite{K2,S}, we recall the Krasnoselskii genus.
\begin{definition}\label{Krasnodef}\it
Set $\mathscr{G}:=\{A\subset \mathscr{X}:\, A\text{ is closed and }A=-A\}$.
The {\rm Krasnoselskii genus} $\Dim_K:\mathscr{G}\rightarrow \mathbb{N}\cup\{+\infty\}$ is defined by
\[
\Dim_K(A):=\left\{
\begin{array}{llll}
 \inf\{m\in \mathbb{N}:\, \exists h\in C^0(A;\mathbb{R}^m\backslash\{0\}),\,h(-u)=-h(u)\},\\
 \\
  +\infty, \ \text{ if }\{m\in \mathbb{N}:\, \exists h\in C^0(A;\mathbb{R}^m\backslash\{0\}),\,h(-u)=-h(u)\}=\emptyset.
\end{array}
\right.
\]
\end{definition}

The Krasnoselskii genus satisfies the following properties; see Struwe \cite[Charpter II,  Proposition 5.2, Proposition 5.4, Observation 5.5]{S}.
\begin{lemma}\label{bascip}
Let $A, B\in \mathscr{G}$ and $h : \mathscr{X} \rightarrow \mathscr{X}$ be a map with $h(-u)=-h(u)$. Then the following properties hold$:$

\begin{itemize}
	\item[(i)] $\mathrm{\Dim}_K(A) \geq0$ with equality if and only if $A=\emptyset;$
	\item[(ii)] $A\subset B$ implies $\Dim_K(A)\leq \Dim_K(B);$
	\item[(iii)] If $A$ is a  finite collection of antipodal
	pairs $u_i,-u_i$, then $\Dim_K(A)=1;$
	\item[(iv)] Given $k\in \mathbb{N}^+$, for any $k$-dimensional linear space $V\subset \mathscr{X}$, one has $\Dim_{K}(\mathcal {S}\cap V)=k;$
	\item[(v)] $\Dim_K(A) \leq \Dim_K(\overline{h(A)});$
	\item[(vi)] $\Dim_K(A\cup B) \leq  \Dim_K(A) + \Dim_K(B);$
	\item[(vii)] If $A$ is compact and $ 0\notin A$, then $\Dim_K(A)<+\infty$ and there is a symmetric neighborhood $O$ of $A$ in
	$\mathscr{X}$ such that $\overline{O}\in \mathscr{G}$ and $\Dim_K(A)=\Dim_K(\overline{O})$.
\end{itemize}
\end{lemma}

By  Lemma \ref{bascip}/(i)-(v) one easily gets the following result.
\begin{proposition}
Define two optional families $\mathscr{D}^\alpha$, $\alpha=1,2,$ by
\begin{align*}
\mathscr{D}^1:=\{A\subset \mathcal {S}: A \text{ is closed and }A=-A\},\ \mathscr{D}^2:=\{A\subset \mathcal {S}: A \text{ is compact and }A=-A\}.
\end{align*}
Then for each $\alpha\in\{1,2\}$, $(\mathscr{D}^\alpha,\Dim_K)$ is a dimension pair.
\end{proposition}

Let $(M,F, d\mathfrak{m})$ be a compact {\rm FMMM}
and let $\alpha\in\{1,2\}$. In view of Theorem  \ref{minmaxth},
the $k^{th}$ eigenvalue of $(\mathscr{D}^\alpha, \Dim_{K})$, denoted by $\lambda_k^{K,\alpha}$, is equal to
\[
\lambda_k^{K,\alpha}=\inf_{A\in \mathscr{D}^{K,\alpha}_k}\sup_{u\in A}E(u),
\]
where $\mathscr{D}^{K,\alpha}_k:=\{A\in \mathscr{D}^{\alpha}:\, \Dim_K(A)\geq k\}$. The collection $\{\lambda_k^{K,\alpha}\}_{k=1}^\infty$ is called the $(\mathscr{D}^\alpha, \Dim_{K})$-spectrum.


We are going to point out an important relation between the   $(\mathscr{C}^\alpha,\Dim_{LS})$-spectrum  and  the  $(\mathscr{D}^\alpha,\Dim_K)$-spectrum; to do this, we recall the following result.
\begin{lemma}[{Fadell \cite[Theorem (3), p.34]{F}}]\label{imrelation}
Let $\mathcal {E}$ be any contractible paracompact free $G$-space, where
$G$ is a compact Lie group. Let $\Sigma$ denote the collection of closed, invariant subsets
of $\mathcal {E}$ and set $\mathcal {B} = \mathcal {E}/G$.
Then for any $A\in \Sigma$, we have
\[
\cat_\mathcal {B}(A/G)= G\text{-genus } A .
\]
In particular, if $G=\mathbb{Z}_2$, the $G$-genus is precisely the Krasnoselskii genus.
\end{lemma}

\begin{theorem}\label{therealconnection}
For any compact {\rm FMMM}, one has
\[
\lambda^{LS,\alpha}_k=\lambda_k^{K,\alpha},\ \forall\,\alpha\in \{1,2\}, \ \forall\,k\in \mathbb{N}^+.
\]
In particular, $\Dim_{LS}(A)=\Dim_{K}(A)$ for any $A\in \mathscr{D}^\alpha$.
\end{theorem}
\begin{proof}
According to Propositions \ref{contr} and \ref{compS}, $\mathcal {S}$ is a contractible, paracompact and $\mathbb{Z}_2$-free space. Fix $\alpha\in\{1,2\}$ and $k\in \mathbb{N}^+$ arbitrarily.
 Given $A\in \mathscr{D}^{K,\alpha}_k$,  $A$ is $\mathbb{Z}_2$-invariant and $\mathfrak{p}(A)=A/\mathbb{Z}_2$. Thus,
 Lemma \ref{imrelation}  yields (by setting $\mathcal {E}:=\mathcal {S}$ and $G:=\mathbb{Z}_2$)
\[
\Dim_{LS}(A)=\cat_{\mathbb{P}(\mathscr{X})}(\mathfrak{p}(A))=\Dim_K(A)\geq k,
\]
which implies $A\in \mathscr{C}^{LS,\alpha}_k$ and hence, $\lambda_k^{LS,\alpha}\leq \lambda_k^{K,\alpha}$.

On the other hand, for any $A\in \mathscr{C}^{LS,\alpha}_k$, set $A':=A\cup-A$. Lemma \ref{imrelation} yields that
\[
\Dim_K(A')=\cat_{\mathbb{P}(\mathscr{X})}(A'/\mathbb{Z}_2)=\cat_{\mathbb{P}(\mathscr{X})}(\mathfrak{p}(A))=     \Dim_{LS}(A)\geq k,
\]
which implies $A'\in \mathscr{D}^{K,\alpha}_k$.
Since $F$ is reversible, we have
$$
\lambda_k^{K,\alpha}\leq \sup_{u\in A'}E(u)=\sup_{u\in A}E(u).$$ Taking the infimum w.r.t  $A\in \mathscr{C}^{LS,\alpha}_k$, it turns out that  $ \lambda_k^{K,\alpha}\leq \lambda_k^{LS,\alpha},
$
which  concludes the proof.
\end{proof}

Theorems \ref{therealconnection} and \ref{compareRiemannen} immediately imply the following result.
\begin{theorem}
	For each $\alpha\in\{1,2\}$,
	$(\mathscr{D}^\alpha, \Dim_{K})$ is a faithful dimension pair.
\end{theorem}

Due to Theorem \ref{therealconnection}, we give a simple proof of Lemma  \ref{muti2}.

\begin{proof}[Proof of Lemma  {\rm \ref{muti2}}] On account of Theorem \ref{therealconnection}, it suffices to show that Lemma  \ref{muti2} holds for the Krasnoselskii dimension pairs.

Fix $\alpha\in\{1,2\}$.
Since $F$ is reversible and $\mathfrak{K}_{\lambda}$ is compact (see Lemma \ref{funddomain}), we have $\mathfrak{K}_{\lambda}\in \mathscr{D}^\alpha$. Then Lemma \ref{bascip}/(vii) yields a  symmetric neighborhood $O$ of $\mathfrak{K}_{\lambda}$ with $$\Dim_K(\overline{O})=\Dim_K(\mathfrak{K}_{\lambda})<+\infty.$$ Set $\epsilon=1$ and let $\epsilon_0$ (resp., $\Phi(\cdot,t)$) be the constant (resp., the family of APH's) in the Homotopy Lemma (Lemma \ref{maintheorem2-0}). By the assumption on $\lambda$, one can
choose  $A\in \mathscr{D}^\alpha$ with $\Dim_{K}(A)\geq k+l-1$ and $\sup_{u\in A}E(u)<\lambda+\epsilon_0$. Homotopy Lemma together with the min-max principle (Theorem \ref{minmaxth}) then yield
\[
\Phi(A,1)\subset \overline{E_{\lambda-\epsilon_0}\cup O},\ \Dim_K(\overline{E_{\lambda-\epsilon_0}})\leq k-1.
\]
Now it follows by Lemma \ref{bascip}/(vi) that
\begin{align*}
\Dim_K(\mathfrak{K}_{\lambda})&=\Dim_K(\overline{O})\geq  \Dim_K(\overline{E_{\lambda-\epsilon_0}\cup O})-\Dim_K(\overline{E_{\lambda-\epsilon_0}})\\
&\geq \Dim_K(\Phi(A,1))-k+1= \Dim_K(A)-k+1\geq l.
\end{align*}
Recall that $(\mathscr{X},(\cdot,\cdot))$ is a complete Hilbert space, where $(\cdot,\cdot)$ is defined by  (\ref{3.1}). In particular, $\mathfrak{K}_{\lambda}$ is still compact  with $\Dim_K(\mathfrak{K}_{\lambda})\geq l$ in $(\mathscr{X},(\cdot,\cdot))$.
Now
let $\{u_1,\ldots, u_s\}$ be a maximal set of mutually orthogonal vectors in $\mathfrak{K}_{\lambda}$,
set $V := \text{span} \{u_1,\ldots, u_s\}\approx\mathbb{R}^s$,  and let $\pi:\mathscr{X}\rightarrow V$ be the  orthogonal projection onto $V$. Since $h:=\pi|_{\mathfrak{K}_{\lambda}}:\mathfrak{K}_{\lambda}\rightarrow \mathbb{R}^s\backslash\{0\}$ is a map with $h(-u)=-h(u)$, we have $s\geq \Dim_K(\mathfrak{K}_{\lambda})\geq l$. The cardinality ${\rm card} \mathfrak{K}_{\lambda}=+\infty$ directly follows by Lemma \ref{bascip}/(iii) whenever $l>1$.
\end{proof}

\subsubsection{Essential dimension pair}
${}$ \\
Inspired by Gromov \cite{G}, we utilize the essential dimension to define dimension pairs.
In the sequel,
a subset $A\subset \mathbb{P}(\mathscr{X})$ is said to be {\it contractible in $\mathbb{P}(\mathscr{X})$ onto a subset $B\subset \mathbb{P}(\mathscr{X})$} if there exists a map $h:A\times [0,1]\rightarrow\mathbb{P}(\mathscr{X})$ with $h(\cdot,0)=\id_A$ and $h(A,1)=B$.
For simplicity, such an $h$ is called a {\it homotopy}.

\begin{definition}[Gromov\cite{G}]\label{GROMESS}\it
Given a closed nonempty set $A\subset \mathbb{P}(\mathscr{X})$,  the {\rm essential dimension of} $A$ is defined by
\begin{align*}
\text{ess}(A):=&\text{the smallest integer $i$ such that $A$ is contractible in $\mathbb{P}(\mathscr{X})$ onto }\\
 &\text{a subset $B\subset \mathbb{P}(\mathscr{X})$ with $\Dim_C(B)=i$},
\end{align*}
and set ${\ess}(\emptyset):=-1$.
\end{definition}


Now we define the {\it essential  dimension} of a closed set $A\subset \mathcal {S}$ as
\[
\Dim_{ES}(A):=\ess(\mathfrak{p}(A))+1.
\]

\begin{lemma}\label{Gropro}
Given the closed subsets $A, B\subset \mathcal {S}$,  we have$:$
\begin{itemize}
	\item[(i)] $\Dim_{ES}(A)\geq 0$ with equality if and only if $A=\emptyset;$
	\item[(ii)] If $A\subset B$, then $\Dim_{ES}(A)\leq \Dim_{ES}(B);$
	\item[(iii)] $\Dim_{ES}(A\cup B)\leq \Dim_{ES}(A)+\Dim_{ES}(B);$
	\item[(iv)] For any APH $h:\mathcal {S}\rightarrow \mathcal {S}$, $\Dim_{ES}(A)=\Dim_{ES}(h(A));$
	\item[(v)] Given $k\in \mathbb{N}^+$, for any $k$-dimensional linear space $V \subset \mathscr{X}$, $\Dim_{ES}(V\cap \mathcal {S} )=k.$
\end{itemize}
\end{lemma}
\begin{proof}
(i) and (ii) follow directly by the definition  and (iii) follows from Gromov \cite[0.4B1]{G}, i.e., for any $\mathfrak{A},\mathfrak{B}\subset \mathbb{P}(\mathscr{X})$,
$\ess(\mathfrak{A}\cup \mathfrak{B})\leq \ess (\mathfrak{A})+\ess(\mathfrak{B})+1$.
To prove (iv), set $H([u]):=\mathfrak{p}\circ h(u)$. It is easy to check that $H:\mathbb{P}(\mathscr{X})\rightarrow\mathbb{P}(\mathscr{X})$ is a homeomorphism   with $H\circ \mathfrak{p}=\mathfrak{p}\circ h$.
Since $\ess$ is invariant under homeomorphisms, we get
\[
\Dim_{ES}(h(A))=\ess(\mathfrak{p}(h(A)))+1=\ess(H(\mathfrak{p}(A)))+1=\ess(\mathfrak{p}(A))+1=\Dim_{ES}(A).
\]
Property  (v) follows directly by $\ess(\mathbb{P}(V))=\Dim_C(\mathbb{P}(V))=k-1$, see Gromov \cite[0.4B/(v)]{G}.
\end{proof}

%
%
%

Lemma \ref{Gropro} immediately yields the following result.
\begin{proposition}
For each $\alpha\in\{1,2\}$,
$(\mathscr{C}^\alpha, \Dim_{ES})$ is a dimension pair.
\end{proposition}

Let $(M,F, d\mathfrak{m})$ be a compact {\rm FMMM} and let $\alpha\in\{1,2\}$.
On account of Theorem \ref{minmaxth}, the $k^{th}$ eigenvalue of $(\mathscr{C}^\alpha,  \Dim_{ES})$, denoted by $\lambda_k^{ES,\alpha}$, equals to
\[
\lambda_k^{ES,\alpha}=\inf_{A\in \mathscr{C}^{ES,\alpha}_k}\sup_{u\in A}E(u),
\]
where $\mathscr{C}^{ES,\alpha}_k:=\{A\in \mathscr{C}^{\alpha}:\, \Dim_{ES}(A)\geq k\}$. The collection $\{\lambda_k^{ES,\alpha}\}_{k=1}^\infty$ is called the $(\mathscr{C}^\alpha, \Dim_{ES})$-spectrum.


Now we  show the following result.
\begin{theorem}
For each $\alpha\in\{1,2\}$,
$(\mathscr{C}^\alpha, \Dim_{ES})$ is a faithful dimension pair.
\end{theorem}
\begin{proof}Fix $\alpha\in\{1,2\}$ and let $(M,g, d\vol_g)$ be a compact Riemannian manifold equipped with its canonical measure. Due to Theorem \ref{properties1}, it suffices to show $ \lambda^\Delta_{k }\leq \lambda^{ES,\alpha}_k$; the proof is divided into two steps.

\textbf{Step 1.} Let $\lambda$ be an eigenvalue of $(\mathscr{C}^\alpha,\Dim_{ES})$.  We claim that
 there exists an open neighbourhood $O$ of the eigenset $ \mathfrak{K}_{\lambda}$ in $\mathcal {S}$ such that  $\Dim_{ES}(\overline{O})= \Dim_{ES}(\mathfrak{K}_{\lambda})<+\infty$.

Since the metric is Riemannian, the eigenspace of $\lambda$, say $V_\lambda$, is a (finite) $s$-dimensional linear space which is spanned by the eigenfunctions $u_1,\ldots,u_s$ satisfying (\ref{4.1}). Since $\Delta$ is linear, it turns out that $\mathfrak{K}_{\lambda}=\mathcal {S}\cap V_\lambda$ and hence,  Lemma \ref{Gropro}/(v) implies $\Dim_{ES}(\mathfrak{K}_{\lambda})=s$ (i.e., $ {\ess}(\mathfrak{p}(\mathfrak{K}_{\lambda}))=s-1$).

Recall that $(\mathscr{X},(\cdot,\cdot))$   is a separable Hilbert space, see (\ref{3.1}). For a fixed $\epsilon\in (1/4,1/2)$, define
\[
O':=\{u+\rho:\,u\in V_\lambda,\, \rho\in V_\lambda^\perp, \, \sqrt{(\rho,\rho)}<\epsilon\}\subset \mathscr{X}.
\]
Using a   complete orthonormal basis, it is easy to check that $O'$ is an open neighbourhood of $V_\lambda$ in $\mathscr{X}$. Thus, ${{O'}}\cap \mathcal {S}$ is a open neighbourhood of $ \mathfrak{K}_{\lambda}$ in $\mathcal {S}$. In the sequel, we show  $\Dim_{ES}(\overline{{{O'}}\cap \mathcal {S}})=s$, i.e., $O:={{O'}}\cap \mathcal {S}$ verifies our claim.

Note that for each $v\in \overline{{O'}\cap \mathcal {S}}$, the representation  $v=u+\rho$ is unique, where $u\in V_\lambda\backslash\{0\}$ and $ \rho\in V_\lambda^\perp$ with $\sqrt{(\rho,\rho)}<\epsilon$. Hence, we can define a homotopy $H: \overline{{O'}\cap \mathcal {S}} \times[0,1]\rightarrow \mathcal {S}$ by
\[
H(u+\rho,t):=\frac{u+(1-t)\rho}{\|u+(1-t)\rho\|_{L^2}}.
\]
Since $H(-v,t)=-H(v,t)$, $H$ induces
a homotopy $H':\mathfrak{p}(\overline{{O'}\cap \mathcal {S}})\times [0,1]\rightarrow \mathbb{P}(\mathscr{X})$ defined by
\[
H'(\mathfrak{p}(v),t):=\mathfrak{p}\circ H(v,t), \text{ for }v\in \overline{{O'}\cap \mathcal {S}},\ t\in [0,1].
\]
It turns out that $\mathfrak{p}(\overline{{O'}\cap \mathcal {S}})$ is contractible onto $\mathfrak{p}(\mathfrak{K}_{\lambda})$ by means of $H'$.  Thus Definition \ref{GROMESS} yields
$$
s-1= {\ess}(\mathfrak{p}(\mathfrak{K}_{\lambda}))\leq  {\ess}(\mathfrak{p}(\overline{{O'}\cap \mathcal {S}}))\leq {\ess}(H'(\mathfrak{p}(\overline{{O'}\cap \mathcal {S}}),1))= {\ess}(\mathfrak{p}(\mathfrak{K}_{\lambda}))=s-1,$$ which implies $ \Dim_{ES}(\overline{{{O'}}\cap \mathcal {S}})=s.
$
Therefore, the claim holds with the choice $O:={O'}\cap \mathcal {S}$.

\textbf{Step 2.} Suppose that for some $i\in \mathbb{N}^+$,
 $\lambda^{ES,\alpha}_i=\lambda^{ES,\alpha}_{i+1}=\cdots=\lambda^{ES,\alpha}_{i+l-1}=:\lambda$,
i.e., the multiplicity of the
eigenvalue $\lambda$ is $l$. Using the  open neighbourhood $O$ of $ \mathfrak{K}_{\lambda}$  constructed in Step 1 and the same argument as in the proof of  Lemma  \ref{muti2}, one can show that $\Dim_{ES}(\mathfrak{K}_{\lambda})\geq l$. By recalling $\Dim_{ES}(\mathfrak{K}_{\lambda})=s$ from Step 1,
we get $l$ linearly independent eigenfunctions $u_i$'s corresponding to $\lambda$.
The rest of the proof is the same as
in Theorem \ref{compareRiemannen}.
\end{proof}

We have shown that   the $(\mathscr{C}^\alpha,\Dim_{LS})$-spectrum  is exactly the  $(\mathscr{D}^\alpha,\Dim_K)$-spectrum (see Theorem \ref{therealconnection}). In order to investigate
the relationship between the $(\mathscr{C}^\alpha,\Dim_{ES})$-spectrum  and the $(\mathscr{C}^\alpha,\Dim_{LS})$-spectrum, we recall the following results, see Cornea,  Lupton,
Oprea and
Tanr\'e  \cite[Remark 1.12, Lemma 1.13]{CLOT} for the proofs.
\begin{lemma}\label{ANRlemmals}
Let $X$ be a normal {\rm  ANR}. For any closed subset $A\subset X$, we have

\begin{itemize}
	\item[(i)] $\cat_X(A)-1\leq \Dim_C(A);$
	\item[(ii)] For any homotopy $h: A\times I\rightarrow X$, $\cat_X(A)\leq \cat_X(h(A,1))$.
\end{itemize}
\end{lemma}

\begin{remark}\rm
The LS category defined in \cite{CLOT} is smaller than the one in Definition \ref{defls} by the factor $1$. So the first statement in \cite{CLOT} reads as $\cat_X(A)\leq \Dim_C(A)$ and the homotopy we defined before is called a {\it deformation}.
\end{remark}

\begin{theorem}\label{comparsiones}For any $k\in \mathbb{N}^+$, we have
\[
\min\{\lambda_k^{LS,1},\,\lambda_k^{ES,1}\}\leq \lambda^{ES,2}_k\leq \lambda^{LS,2}_k.
\]
\end{theorem}
\begin{proof}It suffices to show $\lambda^{ES,2}_k\leq \lambda^{LS,2}_k$. Given $A\in \mathscr{C}^{LS,2}_k$, let $B$ be the homotopic image of  $\mathfrak{p}(A)$ with $\Dim_C(B)=\ess(\mathfrak{p}(A))$. Note that $B$ is compact (closed) and $\mathbb{P}(\mathscr{X})$ is a normal ANR (see Proposition  \ref{finalPX}).
Then Lemma \ref{ANRlemmals}/(i) yields
\[
\Dim_{ES}( A)=\Dim_{C}(B)+1\geq \cat_{\mathbb{P}(\mathscr{X})}(B).
\]
Moreover,
 it follows by  Lemma \ref{ANRlemmals}/(ii) that
 \[
 \cat_{\mathbb{P}(\mathscr{X})}(B)\geq \cat_{\mathbb{P}(\mathscr{X})}(\mathfrak{p}(A))=\Dim_{LS}(A).
 \]
 Accordingly, one has $\Dim_{ES}(A)\geq \Dim_{LS}(A)\geq k$, thus $A\in \mathscr{C}^{ES,2}_k$, which implies  $\lambda^{ES,2}_k\leq \lambda^{LS,2}_k$.
\end{proof}

\subsubsection{Lebesgue covering dimension pair}\label{lebsdiem}
\begin{definition}\it Let
$\mathscr{C}^\alpha$, $\alpha=1,2$ be two optional families, i.e.,
\[
\mathscr{C}^1:=\{A\subset \mathcal {S}: A \text{ is closed}\},\ \mathscr{C}^2:=\{A\subset \mathcal {S}: A \text{ is compact}\}.
\]
Given a closed set $A\subset \mathcal {S}$, we define the {\rm modified Lebesgue covering dimension of} $A$ by
\[
\Dim_{MC}(A):=\left\{
\begin{array}{llll}
 \Dim_C(A)+1,&\text{ if }A\neq\emptyset,\\
 \\
 0,&\text{ if }A=\emptyset.
\end{array}
\right.
\]
\end{definition}
\noindent Since $\mathscr{X}$ is a  separable metric space,  one has $\Dim_C(A)=\ind(A)$ for any subset  $A\subset \mathcal {S}\subset \mathscr{X}$, where $\ind(\cdot)$ denotes the  inductive dimension.
Consequently, we have the following result.
\begin{lemma}\label{someimd} Given $\alpha\in \{1,2\}$,
let $A, B\in \mathscr{C}^{\alpha}$ and $h:\mathcal {S}\rightarrow \mathcal {S}$ be a   homeomorphism. The following properties hold$:$

\begin{itemize}
	\item[(i)] $\Dim_{MC}(A)\geq 0$ with equality if and only if $A=\emptyset;$
	\item[(ii)] If $A\subset B$, then $\Dim_{MC}(A)\leq \Dim_{MC}(B);$
	\item[(iii)] $\Dim_{MC}(A\cup B)\leq \Dim_{MC}(A)+\Dim_{MC}(B);$
	\item[(iv)] $\Dim_{MC}(A)=\Dim_{MC}(h(A));$
	\item[(v)] Given $k\in \mathbb{N}^+$, for any $k$-dimensional linear space $V \subset \mathscr{X}$, one has $\Dim_{MC}(\mathcal {S}\cap V )=k$.
\end{itemize}
\end{lemma}
\noindent In particular, $(\mathscr{C}^\alpha,\Dim_{MC})$, $\alpha=1,2$ are dimension pairs. By Theorem \ref{minmaxth}, we obtain
the $(\mathscr{C}^\alpha,\Dim_{MC})$-spectrum, that is,
\[
\lambda^{MC,\alpha}_k=\inf_{A\in \mathscr{C}^{MC,\alpha}_{k}}\sup_{u\in  A}E(u), \ \text{ for } k\in \mathbb{N}^+,
\]
where $\mathscr{C}^{MC,\alpha}_{k}:=\{A\in \mathscr{C}^{\alpha}:\,  \Dim_{MC}(A)\geq k\}$.

\begin{proposition}\label{flawex1}For each $\alpha\in \{1,2\}$, one has $\lambda_k^{MC,\alpha}=\lambda_1^{MC,\alpha}$ for every $k\in \mathbb{N}^+$. Hence $(\mathscr{C}^\alpha,\Dim_{MC})$ is not faithful.
\end{proposition}
\begin{proof}Let $(M,g,d\vol_g)$ be a compact Riemannian manifold equipped with its canonical measure and fix $\alpha\in\{1,2\}$.
Theorem \ref{properties1}/(i) implies $\lambda^{MC,\alpha}_1=\lambda^{\Delta}_1$.
For $k\geq 2$, let $u_i\in \mathcal {S}$, $i=1,\ldots,k$ be the  eigenfunctions corresponding to $\lambda^\Delta_i$
with (\ref{4.1}).
Set
$W_k:=\text{Span}\{u_1,\ldots,u_k\}$.
Let $(\cos\alpha_1,\ldots,\cos\alpha_k)$ denote the direction  cosines of a nonzero vector  in $(W_k,(\cdot,\cdot)_{L^2})$ with respect to $\{u_i\}$, i.e., $\cos \alpha_i:=(\cdot, u_i)_{L^2}/\|\cdot\|_{L^2}$, where $(\cdot,\cdot)_{L^2}$ is defined in (\ref{innernormL2}).
Set
\[
A_n:=\left\{ \sum_{i=1}^k\cos \alpha_i\cdot u_i: \, -\frac{\pi}{2^n}\leq \alpha_1\leq \frac{\pi}{2^n}\right\}.
\]
Clearly, $A_n\in \mathscr{C}_k^{MC,\alpha}$ and since $\sum_{i=1}^k\cos^2\alpha_i=1$, we have
\begin{align*}
\lambda_1^{MC,\alpha}\leq \lambda^{MC,\alpha}_k\leq \sup_{u\in A_n}E(u)
\leq \cos^2\alpha_1\cdot \lambda^{\Delta}_1+\sin^2\alpha_1\cdot \lambda_k^\Delta\rightarrow \lambda^\Delta_1=\lambda^{MC,\alpha}_1 \text{ as }n\rightarrow \infty,
\end{align*}
which implies $\lambda^{MC,\alpha}_1=\cdots=\lambda^{MC,\alpha}_k$. In particular, $\lim_{k\rightarrow\infty}\lambda_k^{MC,\alpha}\neq+\infty$ and the statement follows by using Theorem \ref{properties2}/(ii).
\end{proof}

\section{Upper bounds for eigenvalues}\label{section4}
Let $(M,F, d\mathfrak{m})$ be an $n$-dimensional complete FMMM and let
 $(\mathscr{C},\Dim)$ be a dimension pair.
Given $r_0>0$ and $p\in M$,
$(\overline{B_p(r_0)}, F, d\mathfrak{m})$ is an $n$-dimensional compact FMMM and the  corresponding Banach space  $\mathscr{X}$  is denoted by $\mathscr{X}(B_p(r_0))$.
Now  denote by $\lambda_1(B_p(r_0))$ the first eigenvalue  with respect to $(\mathscr{C},\Dim)$ on $(\overline{B_p(r_0)}, F, d\mathfrak{m})$. According to Theorem \ref{properties1}/(i), one has
\[
\lambda_1(B_p(r_0))=\inf_{u\in \mathscr{X}(B_p(r_0))\backslash\{0\} }E(u).
\]

On the other hand, given $N\in [n,+\infty)\cap \mathbb N$,
let $B^N_K(r_0)$ denote a geodesic ball with radius $r_0$ in the $N$-dimensional Riemannian space form of constant sectional curvature $K$, and let $\lambda^\Delta_1(B^N_K(r_0))$ be the usual first eigenvalue of the Beltrami-Laplacian on the compact Riemannian manifold $\overline{B^N_K(r_0)}$.

Inspired by Cheng \cite{C}, we have the following lemma.
\begin{lemma}\label{Chenglemma}
Given $N\in [n,+\infty)\cap \mathbb N$, let $(M,F,d\mathfrak{m})$ be an $n$-dimensional   complete {\rm FMMM} with $\mathbf{Ric}_N\geq (N-1)K$. Then for every  dimension pair $(\mathscr{C},\Dim)$, we have
\[
\lambda_1(B_p(r_0))\leq   \lambda^\Delta_1(B^N_K(r_0)),\ \forall\,p\in M.
\]
\end{lemma}
\begin{proof}
Let $\varphi$ be a nonnegative eigenfunction corresponding to $\lambda^\Delta_1(B^N_K(r_0))$, which is always a radial function.
Let $r(x):=d_F(p,x)$; thus, $f(x):=\varphi\circ r(x)\in \mathscr{X}(B_p(r_0))$.
Let
$\mathcal {S}(B_p(r_0)):=\{u\in \mathscr{X}(B_p(r_0)): \|u\|_{L^2}=1\}$.
Clearly, $A:=\{\pm  f/\|f\|_{L^2}\}\subset \mathcal {S}(B_p(r_0))$. Since $\Dim(A)\geq 1$, Theorem \ref{minmaxth} yields
\[
\lambda_1(B_p(r_0))\leq \frac{\ds\int_{B_p(r_0)} F^{*2}(df)\, d\mathfrak{m}}{\ds\int_{B_p(r_0)}f^2\, d\mathfrak{m}}.\tag{4.1}\label{v14-4.1}
\]
Let $(r,y)$ be the polar coordinate system at $p$ and set $a(y):=\min\{i_y, r_0\}$ for any $y\in S_pM$. Thus, $f(r,y):=f(x)=\varphi(r)$ and $\partial f/\partial r=d\varphi/dr<0$. Moreover, since we have the eikonal equation $F^*(dr)=1$ (cf. Shen \cite{Sh1}), (\ref{volumeforminpolar}) yields
\begin{align*}
\ds\int_{B_p(r_0)}f^2\, d\mathfrak{m}&=\ds\int_{S_pM}d\nu_p(y)\ds\int_{0}^{a(y)}\varphi^2(r) \,\hat{\sigma}_p(r,y) dr,\tag{4.2}\label{finter}\\
\ds\int_{B_{p}(r_0)}F^{*2}(df)\, d\mathfrak{m}
&=\ds\int_{S_pM}d\nu_p(y)\ds\int_{0}^{a(y)}\left( \frac{d \varphi}{dr} \right)^2   \hat{\sigma}_p(r,y) dr.\tag{4.3}\label{dfinter}
\end{align*}
A direct calculation furnishes
\begin{align*}
\ds\int_{0}^{a(y)}\left( \frac{d \varphi}{dr} \right)^2   \hat{\sigma}_p(r,y) dr=\left. \varphi\frac{d\varphi}{dr}\hat{\sigma}_p(r,y)\right|^{a(y)}_0-\int^{a(y)}_0\frac{1}{\hat{\sigma}_p(r,y)}\frac{\partial}{\partial r}\left[ \frac{d\varphi}{dr}\hat{\sigma}_p(r,y) \right]\varphi\hat{\sigma}_p(r,y)dr.\tag{4.4}\label{5.1newsee}
\end{align*}
Relation (\ref{Laplacainr})   together with Lemma \ref{Rho} and $d\varphi/d r<0$ implies
\begin{eqnarray*}
\frac{1}{\hat{\sigma}_p(r,y)}\frac{\partial}{\partial r}\left[ \frac{d\varphi}{dr}\hat{\sigma}_p(r,y) \right]&=&\frac{d^2\varphi}{dr^2}+\frac{d\varphi}{dr}\cdot\frac{\partial}{\partial r}\log \hat{\sigma}_p(r,y)\\
&\geq &\frac{d^2\varphi}{dr^2}+\frac{d\varphi}{dr}\cdot\frac{d}{d r}\left( \log\mathfrak{s}_K^{N-1}(r) \right)\\&=&-\lambda^\Delta_1(B^N_K(r_0))\cdot\varphi,
\end{eqnarray*}
which combined with (\ref{5.1newsee}) and (\ref{v14-2.1})   yields
\begin{align*}
\ds\int_{0}^{a(y)}\left( \frac{d \varphi}{dr} \right)^2   \hat{\sigma}_p(r,y) dr
\leq\lambda^\Delta_1(B^N_K(r_0))\ds\int_{0}^{a(y)}\varphi^2(r) \,\hat{\sigma}_p(r,y) dr.
\end{align*}
Integrating the above inequality over $S_pM$, by (\ref{finter}) and (\ref{dfinter}),  one gets
\[
\ds\int_{B_{p}(r_0)}F^{*2}(df)\, d\mathfrak{m} \leq \lambda^\Delta_1(B^N_K(r_0))\ds\int_{B_p(r_0)}f^2\, d\mathfrak{m},\tag{4.5}\label{7.4}
\]
which together with (\ref{v14-4.1}) concludes the proof.
\end{proof}

According to Kronwith \cite{K3}, we introduce   convex Finsler manifolds.
\begin{definition}\it
Let $(M,F)$ be a   complete reversible Finsler manifold and let $C$ be a subset in $M$. $C$ is called {\rm convex} if for any $p,q\in \overline{C}$, there exists a minimal geodesic in $M$ from $p$ to $q$, which is contained in $\overline{C}$.
An $n$-dimensional compact reversible Finsler manifold $M$
 $($with or without
 boundary$)$  is said to be {\rm convex} if there are an $n$-dimensional  complete reversible Finsler manifold $W$
   and an isometric imbedding $i: M \hookrightarrow W$ such
that $i(M)$ is a convex subset of $W$.
\end{definition}

We have the following Cheng type estimate.
\begin{theorem}\label{Chees}Given $N\in [n,+\infty)\cap \mathbb N$, $K\in \mathbb{R}$ and $d>0$,
let $(M,F,d\mathfrak{m})$ be an $n$-dimensional compact convex {\rm FMMM} with $\mathbf{Ric}_N\geq (N-1)K$ and $\diam(M)=d$. Thus, for any dimension pair $(\mathscr{C},\Dim)$, the corresponding spectrum satisfies
\[
\lambda_k\leq  \lambda^\Delta_1\left(B^N_K\left(\frac{d}{2k }\right)\right), \ \forall k\in \mathbb{N}^+.
\]
Moreover, 
there exists a constant $C=C(N)>0$  depending only on $N$ such that
\[
\lambda_{k}\leq  \frac{(N-1)^2}{4}|K|+C(N)\left(\frac{k}{d}\right)^2, \  \forall\,k\in \mathbb{N}^+.
\]
\end{theorem}
\begin{proof}
Let $k\in \mathbb N^+$ and choose $k$ distinct points $\{p_i\}_{i=1}^k$ in $M$ such that $B_i:=B_{p_i}(d/(2k ))$ are pairwise disjoint. Let $r_i(x):=d_F(p_i,x)$ and let
 $\varphi$ be the first eigenfunction corresponding to $\lambda_1^{\Delta}(B^N_K(d/(2k )))$. Now we define the functions $f_i$, $i=1,\ldots k$ on $M$ by
\[
f_i(x):=\left\{
\begin{array}{llll}
 \varphi\circ r_i(x),\ &\text{ if }x\in B_i;\\
 \\
 0,\ &\text{ if }x\in M\backslash B_i.
\end{array}
\right.
\]
Clearly, $f_i\in \mathscr{X}$.
Set $V_k:=\text{Span}\{f_1,\ldots,f_k\}$.
Since $\Dim(V_k\cap \mathcal {S})\geq k$, Theorem \ref{minmaxth} furnishes
\[
\lambda_k\leq \sup_{u\in V_k\cap \mathcal {S}}E(u)=\sup_{u\in V_k\backslash\{0\}}E(u).\tag{4.6}\label{5.4}
\]
Since   $\text{supp}(f_i)\subset \overline{B_i}$ and $B_i$ are pairwise disjoint sets, $i=1,\ldots k$, for any choice $u=\sum_{i=1}^k a_i f_i\in V_k $ with $a_i\in \mathbb R$, the estimate (\ref{7.4})  yields
\begin{eqnarray*}
\ds\int_M F^{*2}(du)d\mathfrak{m}&=&\ds\int_M\sum_{i,j=1}^ka_i\,a_j\, g_{du}(df_i,df_j)d\mathfrak{m}
 = \sum_{i=1}^k a^2_i\ds\int_{B_i} F^{*2}(df_i)d\mathfrak{m}\\
 \\&\leq&  \lambda^\Delta_1\left(B^N_K\left(\frac{d}{2k }\right)\right)\sum_{i=1}^k a^2_i\ds\int_{B_i} f^2_id\mathfrak{m}\\&=& \lambda^\Delta_1\left(B^N_K\left(\frac{d}{2k }\right)\right)\ds\int_M u^2d\mathfrak{m},
\end{eqnarray*}
which together with (\ref{5.4}) gives $\lambda_k\leq \lambda^\Delta_1\left(B^N_K\left(\frac{d}{2k }\right)\right).$ It remains to use the estimate of Cheng \cite[p.294]{C} for $\lambda^\Delta_1\left(B^N_K\left(\frac{d}{2k }\right)\right)$ which concludes the proof.
\end{proof}

\begin{proof}[Proof of Theorem {\rm \ref{th1.2}}]
Since a closed reversible Finsler manifold is always compact and  convex,
Theorem \ref{th1.2} follows by Theorem \ref{Chees}.
\end{proof}

\section{Lower bounds for eigenvalues}\label{section5}
In this section we study the lower bounds of eigenvalues for  faithful dimension pairs $(\mathscr{C},\Dim)$ on a closed FMMM $(M,F,d\mathfrak{m})$.
Due to Theorem \ref{properties1}, the first eigenvalue of $(\mathscr{C},\Dim)$ is always zero. For convenience, we use $\overline{\lambda}_k$ to denote the {\it $k^{th}$ positive eigenvalue} of  $(\mathscr{C},\Dim)$, i.e., $\overline{\lambda}_k$ is the $(k+1)^{th}$ eigenvalue (see Theorem \ref{properties2}). As we already pointed out, in order to provide lower bounds for the eigenvalues, we necessarily have to deal with faithful dimension pairs, see Proposition \ref{flawex1}.

\subsection{General faithful dimension pairs}
We first study lower bounds of eigenvalues of  general faithful dimension pairs by means of the Ricci curvature $\mathbf{Ric}$, the distortion $\tau$ and the uniformity constant $\Lambda_F$, respectively.

\subsubsection{Dirichlet region and Cheeger's constant}
We naturally extend the concepts of Dirichlet region  (cf. Buser \cite{B}) and Cheeger's constant (cf. Cheeger \cite{Cheeger}) to Finsler geometry, both playing key roles in our arguments.

\begin{definition}\label{Dirdef}\it Let $(M,F)$ be a closed reversible Finsler manifold. Given $r>0$,
a sequence of points $\{p_i\}_{i=1}^m$ is called a {\rm complete $r$-package} if $\{B_{p_i}(r)\}_{i=1}^m$ is   a maximal family of disjoint $r$-balls in $(M,F)$.
The {\rm Dirichlet regions} corresponding to a complete $r$-package $\{p_i\}$ is defined as
\[
D_i:=\{q\in M: d_F(p_i,q)\leq d_F(p_j,q), \text{ for all }j=1,\ldots,m \}, \ 1\leq i\leq m.
\]
\end{definition}

Let $D$ be a domain and let $p\in M$. $D$ is called {\it starlike with respect to $p$} if
 each minimizing geodesic from $p$ to an arbitrary point $q\in D$ is always contained in $D$. We have the following lemma whose
proof will be provided in Appendix \ref{propertiesDirdomAb}.
\begin{lemma}\label{propertiesDirdom}
Let $\{D_i\}_{i=1}^m$ be the Dirichlet regions corresponding to $\{p_i\}_{i=1}^m$ defined as in Definition {\rm \ref{Dirdef}}. Then $\{D_i\}_{i=1}^m$ is a  covering of $M$ with $\mathfrak{m} (D_i\cap D_j)=0$ for any $i\neq j$. In particular, for each $i\in\{1,\ldots,m\}$,
$B_{p_i}\left({r}  \right)\subset D_i\subset B_{p_i}(2r)$, $\mathfrak{m}(\text{int}(D_i))=\mathfrak{m}(D_i)$ and $\text{int}(D_i)$ is starlike with respect to $p_i$,
where $\text{int}(D_i)$ denotes the interior of $D_i$.
\end{lemma}

Let $\Ca_M(r)$ be the {maximum number of disjoint $r$-balls in $M$} and  $\Co_M(r)$ be the {minimum number of $r$-balls it takes to cover $M$}.
Given  a   complete $r$-package $\{p_i\}_{i=1}^m$, it is easy  to see
\[
m=\Ca_M(r)\leq  \Co_M(r)\leq \Ca_M\left(\frac{r}{2}\right).\tag{5.1}\label{coveringDir}
\]

Let
$i:\Gamma{\hookrightarrow}M$ be a smooth hypersurface
embedded in $(M,F,d\mathfrak{m})$. For each $x\in\Gamma$, there exist a $1$-form
$\omega(x)\in T^*_xM$ satisfying
$i^*(\omega(x))=0$ and $F^*(\omega(x))=1$. Then $\textbf{n}(x):=\mathfrak{L}^{-1}(\omega_\pm(x))$ is a unit normal
vector on $\Gamma$. The induced measure on $\Gamma$ is defined by $dA=i^*(\textbf{n}\rfloor d\mathfrak{m})$ (cf. Shen \cite{Sh1}).
The Cheeger's constant can be defined in the Finsler setting; see \cite{GS,Sh1,YZ,ZS} for more details.

\begin{definition}\label{Cheergerconstant}\it Let $(M,F,d\mathfrak{m})$ be an $n$-dimensional complete {\rm FMMM}.
Given an open subset $D\subset M$,
its {\rm Cheeger constant}
${\mathbbm{h}}(D)$ is defined by
\[
\mathbbm{h}(D)=\underset{\Gamma}{\inf}\frac{A(\Gamma)}{\min\{\mathfrak{m}(D_1),\mathfrak{m}(D_2)\}},
\]
where $\Gamma$ varies over compact $(n-1)$-dimensional submanifolds
of $M$ which divide $D$ into disjoint open subsets $D_1$, $D_2$
of $D$ with common boundary $\partial D_1 \cap \partial D_2=\Gamma$.
\end{definition}

{\begin{lemma}\label{Th1old}Given  $\Theta\geq 1$,
let $(M,F,d\mathfrak{m})$ be an $n$-dimensional complete {\rm  FMMM} with $|\tau|\leq \log \Theta$.
For any  open subset $D\subset M$   of $M$, we have
\[
\inf_{\left\{f\in C^\infty(M) :\,f|_D\neq0,\  \int_D fd\vol_{\hat{g}}=0\right\}}\frac{\ds\int_D \hat{g}(d f,df)d\vol_{\hat{g}}}{\ds\int_Df^2d\vol_{\hat{g}}} \geq\frac{1}{4\Lambda_F^{1+n}\Theta^2}{\mathbbm{h}^2(D)},
\]
where   $d\vol_{\hat{g}}$ is the canonical Riemannian measure induced by the average Riemannian metric $\hat{g}$ and $\Lambda_F$ is the uniformity constant.
\end{lemma}}
\begin{proof}The proof is divided into two steps.

\textbf{Step 1.} We first provide a quantitative form of (\ref{v14-3.3}), i.e.,
\[
\left(\Lambda_F^\frac{n}2\Theta\right)^{-1}d\vol_{\hat{g}}\leq d\mathfrak{m}\leq \Lambda_F^\frac{n}2\Theta\, d\vol_{\hat{g}}.\tag{5.2}\label{volumecomparison}
\]
In order to do this,
 choose a local coordinate system $(x^i)$ around an arbitrary point $x$ and write
\[
d\mathfrak{m}:=\sigma(x)\, dx^1\dots dx^n,\ d\vol_{\hat{g}}:=\sqrt{\det \hat{g}(x)}\,dx^1\dots dx^n.
\]
We are going to estimate
$f(x):=\frac{\sigma(x)}{\sqrt{\det \hat{g}(x)}}$.
For any $y\in S_xM$,
express
\[
f(x)=\frac{\sigma(x)}{\sqrt{\det g(x,y)}}\frac{\sqrt{\det g(x,y)}}{\sqrt{\det \hat{g}(x)}}=e^{-\tau(y)}\sqrt{\frac{\det g(x,y)}{\det\hat{g}(x)}}.
\]
Select a $\hat{g}$-orthonormal basis $\{E_i\}$ at $x$ such that each $E_i$ is an eigenvector of $(g_{ij}(x,y))$. Thus, (\ref{averagenorm}) implies
\[
\Lambda^{-\frac{n}2}_F\leq\sqrt{\frac{\det g(x,y)}{\det\hat{g}(x)}}\leq \Lambda^{\frac{n}2}_F,
\]
which combined with the definition of the distortion $\tau$ of $M$, we obtain the estimate  (\ref{volumecomparison}).

\textbf{Step 2.} Given $f\in C^\infty(M)$ with $f|_D \neq0$ and $\ds\int_D fd\vol_{\hat{g}}=0$. Let $\alpha$ be a median of $f|_D$, i.e.,
\[
\mathfrak{m}(\{x\in D: f(x)\geq \alpha\})\geq \frac12\mathfrak{m}(D), \ \mathfrak{m}(\{x\in D: f(x)\leq \alpha\})\geq \frac12\mathfrak{m}(D).
\]
It is easy to see that such an $\alpha$ always exists.
Set $f_+:=\max\{f-\alpha,0\}$ and
$f_-:=\min\{f-\alpha,0\}$. By the definition of median, one can check that for any $t>0$,
\[
\mathfrak{m}(\{x\in D: f_+^2(x)\geq t\})\leq \frac12\mathfrak{m}(D), \ \mathfrak{m}(\{x\in D: f_-^2(x)\geq t\})\leq \frac12\mathfrak{m}(D).
\]
The above inequalities together with the layer cake representation (see Lieb and Loss \cite[Theorem 1.13]{LL}) and Lemma \ref{chele} yield
\begin{eqnarray*}
{\mathbbm{h}}(D)\int_D|f-\alpha|^2d\mathfrak{m}&=&{\mathbbm{h}}(D)\int_D(f^2_++f^2_-)d\mathfrak{m}\\
&=&{\mathbbm{h}}(D)\left(\int_0^\infty\mathfrak{m}(\{x\in D: f_+^2(x)\geq t\})dt+\int_0^\infty\mathfrak{m}(\{x\in D: f_-^2(x)\geq t\})dt\right)\\
&\leq& \int_DF^*(df^2_+)d\mathfrak{m}+\int_DF^*(df^2_-)d\mathfrak{m}=2\int_Df_+F^*(df_+)+(-f_-)F^*(-df_-)d\mathfrak{m}\\
&\leq &2  \int_D|f-\alpha|F^*(df)d\mathfrak{m}\leq 2 \left(\int_D|f-\alpha|^2d\mathfrak{m}\right)^{\frac12}\left(\int_DF^{*2}(df)d\mathfrak{m}\right)^{\frac12}.
\end{eqnarray*}
Accordingly, we have
\[
\int_DF^{*2}(df)d\mathfrak{m}\geq \frac{{\mathbbm{h}^2}(D)}{4}\int_D|f-\alpha|^2d\mathfrak{m},
\]
which together with (\ref{volumecomparison}) yields
\begin{align*}
\int_D\hat{g}(df,df)d\vol_{\hat{g}}\geq \frac{1}{\Lambda_F^{1+n}\Theta^2}\frac{{\mathbbm{h}^2(D)}}{4}\int_D|f-\alpha|^2d\vol_{\hat{g}}.
\end{align*}
Since $\ds\int_Dfd\vol_{\hat{g}}=0$, it follows that
$
\underset{\alpha\in \mathbb{R}}{\inf}\ds\int_D|f-\alpha|^2d\vol_{\hat{g}} \geq \int_Df^2d\vol_{\hat{g}},
$ which ends the proof.
\end{proof}

Inspired by Buser \cite{B}, we have the following estimate; since the proof is similar to Buser's original, we postpone its proof to Appendix \ref{cheegerconst}.
\begin{lemma}\label{infact}
\label{imle}Given $K\in \mathbb R$ and $\Theta\geq 1$,
let $(M,F,d\mathfrak{m})$ be an $n$-dimensional complete {\rm  FMMM} with
\[
\mathbf{Ric}\geq (n-1)K,\ |\tau|\leq \log \Theta.
\]
Suppose that  $D\subset M$ is a starlike open set with respect to a point $p$ such that
 $B_p(r)\subset D\subset B_p(R)$.   Then
\begin{align*}
\mathbbm{h}(D)\geq \frac{C^{1+\sqrt{|K|}R} }{\Theta^4}\frac{r^{n-1}}{R^n},
\end{align*}
where  $C=C(n)<1$ is a positive constant depending only  on $n$.
\end{lemma}

\subsubsection{Gromov type  estimate}

\begin{lemma}\label{firslem}Given $\Theta\geq 1$,
let $(M,F,d\mathfrak{m})$ be an $n$-dimensional closed {\rm FMMM} with $|\tau|\leq \log \Theta$.
Then for any faithful dimension pair  $(\mathscr{C},\Dim)$,
the $k^{th}$ positive eigenvalue satisfies
\begin{align*}
\overline{\lambda}_k\geq  \frac{1}{\Lambda^{2n+2}_F\Theta^4} \sup_{\{u_i\}_{i=1}^{k}\in L^2(M)}\inf_{\left\{f\in \mathscr{X}:\,\textcolor[rgb]{0.00,0.00,0.00}{\langle f,u_i\rangle_{L^2}}=0 \right\}}E(f),\tag{5.3}\label{new5.1CC**}
\end{align*}
where the supremum is taken over any set of $k$ functions $\{u_i\}_{i=1}^{k}\in L^2(M)$, the infimum is taken  over all functions $f$ which are perpendicular to $u_i$ in the sense that $\textcolor[rgb]{0.00,0.00,0.00}{\langle u_i,f\rangle_{L^2}:=\ds\int_M u_ifd\vol_{\hat{g}}}=0$.
\end{lemma}
\textcolor[rgb]{0.00,0.00,0.00}{\begin{remark}\rm
		We notice that the inner produce $\langle \cdot,\cdot\rangle_{L^2}$ is w.r.t. the measure $d\vol_{\hat{g}}$, which is not the same as $( \cdot,\cdot)_{L^2}$ where the measure is  $d\mathfrak{m}$, see (\ref{innernormL2}).
\end{remark}}
\begin{proof}
Theorem \ref{minmaxth}
together with relation (\ref{volumecomparison}) and Definition \ref{faithful} yields
\[
\overline{\lambda}_k=\inf_{A\in \mathscr{C}_{k+1}}\sup_{u\in A} E(u)\geq \frac{1}{\Lambda_F^{n+1}\Theta^2}\inf_{A\in \mathscr{C}_{k+1}}\sup_{u\in A}\frac{\ds\int_M\hat{g}(du,du)d\vol_{\hat{g}}}{\ds\int_Mu^2d\vol_{\hat{g}}}=\frac{\overline{\lambda}^{\Delta_{\hat{g}}}_{k}}{\Lambda_F^{n+1}\Theta^2},\tag{5.4}\label{5.1newCCC}
\]
where  $\overline{\lambda}^{\Delta_{\hat{g}}}_k$ denote the $k^{\rm th}$  positive eigenvalue of $(M,\hat{g},d\vol_{\hat{g}})$, i.e., $\overline{\lambda}^{\Delta_{\hat{g}}}_k={\lambda}^{\Delta_{\hat{g}}}_{k+1}$; here, we explored the fact that $(\mathscr{C},\Dim)$ is a faithful dimension pair.

The max-min theorem in Riemanian geometry (cf. Chavel \cite[p.17]{Ch}) together with (\ref{volumecomparison}) yields
\begin{align*}
\overline{\lambda}^{\Delta_{\hat{g}}}_k&\geq \sup_{\{u_i\}_{i=1}^{k}\in L^2(M)}\inf_{\left\{f\in \mathscr{X}:\,\textcolor[rgb]{0.00,0.00,0.00}{\langle f,u_i\rangle_{L^2}}=0 \right\}}\frac{\ds\int_M\hat{g}(df,df)d\vol_{\hat{g}}}{\ds\int_Mf^2d\vol_{\hat{g}}}
\geq \frac{1}{\Lambda_F^{n+1}\Theta^2}\sup_{\{u_i\}_{i=1}^{k}\in L^2(M)}\inf_{\left\{f\in \mathscr{X}:\,\textcolor[rgb]{0.00,0.00,0.00}{\langle f,u_i\rangle_{L^2}}=0 \right\}}E(f),
\end{align*}
which combined with (\ref{5.1newCCC}) yields (\ref{new5.1CC**}).
\end{proof}

Now we have a Gromov type lower estimate for eigenvalues.

\begin{theorem}\label{thefisrttheorem}Given $K\in \mathbb R$, $\Theta\geq 1$ and $d>0$,
let $(M,F,d\mathfrak{m})$ be an $n$-dimensional closed {\rm FMMM} with
\[
\mathbf{Ric}\geq (n-1)K, \ |\tau|\leq \log\Theta,\ \diam(M)=d.
\]
Then there is a positive constant  $\mathfrak{C}_1=\mathfrak{C}_1(n)$ such that for any faithful dimension pair $(\mathscr{C},\Dim)$, the $k^{th}$ positive eigenvalue satisfies
\textcolor[rgb]{0.00,0.00,0.00}{\[
\overline{\lambda}_{k}\geq \frac{\mathfrak{C}_1^{1+d\sqrt{|K|}}}{\Lambda^{4n+4}_F\Theta^{18}\cdot d^2}\, k^\frac{2}{n},\ \forall\,k\in \mathbb{N}^+.
\]}
\end{theorem}
\begin{proof}  Without loss of generality, we can assume $K< 0$; moreover, due a scaling, we may consider $K=-1$. Given any $\varepsilon>0$, let $\{p_i\}_{i=1}^m$ be a complete $\varepsilon$-package and let $\{D_i\}_{i=1}^m$ be the  Dirichlet regions.
Denote by ${u_i}$  the characteristic function of $D_i$. For each $f\in \mathscr{X}$ with \textcolor[rgb]{0.00,0.00,0.00}{$ \ds\int_{D_i}fd\vol_{\hat{g}}  =\langle f,{u_i}\rangle_{L^2}=0$}, $i=1,\ldots,m$, Lemmas \ref{propertiesDirdom} and \ref{Th1old} with relation (\ref{volumecomparison}) provide
\textcolor[rgb]{0.00,0.00,0.00}{\begin{align*}
\ds\int_M F^{*2}(df)d\mathfrak{m}=&\sum_{i=1}^m\ds\int_{\text{int}(D_i)}F^{*2}(df)d\mathfrak{m}\geq \frac{1}{\Lambda_F^{1+\frac{n}2}\Theta}\sum_{i=1}^m\ds\int_{\text{int}(D_i)} \hat{g}(df,df)  d\vol_{\hat{g}}      \geq \sum_{i=1}^m\frac{\mathbbm{h}^2(\text{int}(D_i))}{4 (\Lambda_F^{1+n}\Theta^2)^2}\ds\int_{\text{int}(D_i)} f^2 d\mathfrak{m}\\
\geq& \frac{\underset{1\leq i\leq m}{\min} \mathbbm{h}^2\left(\text{int}(D_i)\right)}{4(\Lambda_F^{1+n}\Theta^2)^2}\sum_{i=1}^m\ds\int_{\text{int}(D_i)} f^2 d\mathfrak{m}=\frac{\underset{1\leq i\leq m}{\min} \mathbbm{h}^2\left(\text{int}(D_i)\right)}{4(\Lambda_F^{1+n}\Theta^2)^2}\ds\int_M f^2 d\mathfrak{m}.
\end{align*}}
The latter relation combined with Lemma \ref{firslem} furnishes
\textcolor[rgb]{0.00,0.00,0.00}{\begin{align*}
\overline{\lambda}_{m}\geq  \frac{ \min\{\mathbbm{h}^2(\text{int}(D_i)),\,i=1,\ldots,m \}}{4(\Lambda^{1+n}_F\Theta^2)^4}.
\end{align*}
}

Now (\ref{coveringDir}) implies $m=\Ca_M(\varepsilon)\leq \Co_M(\varepsilon)=:k$ and Lemma \ref{propertiesDirdom} yields
\[
\left\{
\begin{array}{llll}
B_{p_i}(\varepsilon)\subset D_i\subset B_{p_i}(2 \varepsilon),  && \text{ if }0<\varepsilon\leq \frac{d}{2 },\\
 \\
B_{p_i}(\varepsilon)\subset D_i\subset B_{p_i}(d),  && \text{ if }\varepsilon\geq \frac{d}{2 },
\end{array}
\right.
\]
which together with Lemma \ref{infact} implies the existence of  a positive constant $C_1=C_1(n)<1$ such that
\textcolor[rgb]{0.00,0.00,0.00}{\[
\overline{\lambda}_{k}\geq \overline{\lambda}_{m}\geq \frac{C_1^{1+d}}{\Lambda^{4n+4}_F\Theta^{16}}\frac{1}{\varepsilon^2},\ \forall\, \varepsilon>0.\tag{5.5}\label{3.3newcC*}
\]}
Furthermore, (\ref{coveringDir}) and Lemma  \ref{distorlemma-0}/(ii) yield
\[
k=\Co_M(\varepsilon)\leq\Ca_M\left(\frac{\varepsilon}{2}\right)\leq \frac{\mathfrak{m}(M)}{\ds\min_{i=1,...,k}\mathfrak{m}\left(B_{p_i}(\frac{\varepsilon}{2})\right)}=\ds\max_{i=1,...,k}\frac{\mathfrak{m}(B_{p_i}(d))}{\mathfrak{m}(B_{p_i}(\frac{\varepsilon}{2}))}\leq
 \Theta^{2}\frac{V_{n,-1}\left( d \right)}{V_{n,-1}\left(\frac{\varepsilon}{2}\right)},\ \forall\,0<\varepsilon\leq d.
\]
Thus there is a number $C_2=C_2(n)>0$ such that
\[
\left\{
\begin{array}{llll}
\varepsilon\leq C_2 \, \Theta^{\frac2n}  \,d\,e^d\,k^{-\frac1n}, &&\text{ if } 0<\varepsilon\leq d;\\
 \\
k=1, &&\text{ if }\varepsilon\geq d,
\end{array}
\right.
\]
which combined with (\ref{3.3newcC*}) yields
\textcolor[rgb]{0.00,0.00,0.00}{\[
\overline{\lambda}_{k}\geq \frac{\mathfrak{C}_1^{1+d}}{\Lambda^{4n+4}_F\Theta^{18}\cdot d^2} k^\frac{2}{n}, \ \forall\, k\in \mathbb{N}^+,
\]}
which concludes the proof.
\end{proof}

\begin{proof}[Proof of Theorem \ref{GromovtypBH}] According to Lemma \ref{uniformBuHo}, Theorem \ref{GromovtypBH} directly follows by Theorem \ref{thefisrttheorem}.
\end{proof}

\subsubsection{Buser type estimate}
In order to give a Buser type lower bounds of eigenvalues, we need the following Croke type inequality.
\begin{lemma}[{Zhao and Shen \cite[Proposition 6.3]{ZS}}]\label{Crokineq}
Let $(M,F,d\mathfrak{m})$ be an $n$-dimensional closed {\rm FMMM}, where $d\mathfrak{m}$ is either the Busemann-Hausdorff measure or the Holmes-Thompson measure.
Then there is a  constant $C=C(n)>0$  depending only on $n$ such that
for any $p\in M$, we have
\[
\mathfrak{m}\left(B_p(r)\right)\geq  \frac{C}{ \Lambda_F^{5n^2}}\,  r^n, \ \forall\,0<r\leq\frac{\mathfrak{i}_M}2,
\]
where $\mathfrak{i}_M$ is the injectivity radius of $M$.
\end{lemma}

Now we have the following estimate.

\begin{theorem}\label{compBuse}
Let $(M,F,d\mathfrak{m})$ be an $n$-dimensional closed {\rm FMMM}, where $d\mathfrak{m}$ is either the Busemann-Hausdorff measure or the Holmes-Thompson measure. Given  $K\in \mathbb R$ and $V>0$, suppose
\[
 \mathbf{Ric}\geq(n-1)K,\ \mathfrak{m}(M)=V.
\]
Then there exist two constants $\mathfrak{C}_2=\mathfrak{C}_2(n)>0$ and $\mathfrak{C}_3=\mathfrak{C}_3(n)>0$ such that the elements in the spectrum of any faithful dimension pair $(\mathscr{C},\Dim)$ satisfy
\textcolor[rgb]{0.00,0.00,0.00}{\[
\overline{\lambda}_{k}\geq\left\{
\begin{array}{llll}
\frac{{\mathfrak{C}_2}^{1+\frac{\sqrt{|K|}\Lambda_F^{5n^2}V}{ \mathfrak{i}_M^{n-1}}}}{\Lambda^{21n^2}_F} \left(\frac{\mathfrak{i}_M^{n-1} }{V}\right)^2,&& \text{ if }k< \mathfrak{C}_3\,\Lambda_F^{5n^2}\frac{V}{\mathfrak{i}_M^n},\\
 \\
   \frac{\mathfrak{C}_2^{1+ \sqrt{|K|}\Lambda_F^{5n} {V}^{\frac1n}}}{\Lambda_F^{32n}}\left( \frac{k}{V} \right)^{\frac{2}{n}},&& \text{ if }k\geq\mathfrak{C}_3\,\Lambda_F^{5n^2}\frac{V}{\mathfrak{i}_M^n}.
\end{array}
\right.
\]}
\end{theorem}
\begin{proof}
The proof is divided into two steps.

 \textbf{Step 1.}
Given $\varepsilon>0$, let $\{p_i\}_{i=1}^m$ be a complete $\varepsilon$-package, where $m=\Ca_M(\varepsilon)$. We are going to show that
\[
m=m(\varepsilon)\leq \left\{
\begin{array}{llll}
  \frac{\Lambda_F^{5n^2}\, V}{C\,\varepsilon^n},&& \text{ if }\varepsilon\leq  {\mathfrak{i}_M}/{2},\\
 \tag{5.6}\label{mestimate}\\
 \frac{2^{n+1}\,\Lambda_F^{5n^2} \, V}{C\, \mathfrak{i}_M^{n-1}\, \varepsilon},&& \text{ if }\varepsilon>  {\mathfrak{i}_M}/{2},
\end{array}
\right.
\]
where $C=C(n)>0$ is defined in Lemma \ref{Crokineq}.

 \textbf{Case 1.} If $\varepsilon\leq  {\mathfrak{i}_M}/{2}$, then Lemma \ref{Crokineq} yields
$
\mathfrak{m}(B_{p_i}(\varepsilon))\geq  \frac{C}{ \Lambda_F^{5n^2}}\,\varepsilon^n$ which implies $ m=m(\varepsilon)\leq \frac{\Lambda_F^{5n^2}\, V}{C\,\varepsilon^n}.
$

 \textbf{Case 2.} Suppose  $\varepsilon> {\mathfrak{i}_M}/{2}$. An elementary construction implies that  each ball $B_{p_i}(\varepsilon)$ contains at least $s:=\lfloor \frac{\varepsilon}{\mathfrak{i}_M}+\frac12 \rfloor$ disjoint ${\mathfrak{i}_M}/{2}$-balls, where $\lfloor a \rfloor$ denotes the greatest integer not larger than $ a$; such a construction can be performed by placing ${\mathfrak{i}_M}/{2}$-balls with centers $q_l$, $l=1,...,s$, along a unit speed minimizing geodesic joining any element of the boundary $\partial B_{p_i}(\varepsilon)$ to $p_i$.
Hence, by Lemma \ref{Crokineq} one has that
\begin{align*}
\mathfrak{m}(B_{p_i}(\varepsilon))&\geq s\cdot\min_{1\leq l\leq s}\mathfrak{m}\left(B_{q_l}\left(\frac{\mathfrak{i}_M}{2}\right)\right)\geq   \frac{C\,\mathfrak{i}^{n-1}_M\, \varepsilon}{2^{n+1}\Lambda_F^{5n^2}},
\end{align*}
and
\begin{align*}
m=m(\varepsilon)\leq \frac{\mathfrak{m}(M)}{\min_{i}\mathfrak{m}(B_{p_i}(\varepsilon))}\leq \frac{2^{n+1}\,\Lambda_F^{5n^2} \, V}{C\, \mathfrak{i}_M^{n-1}\, \varepsilon}.
\end{align*}

 \textbf{Step 2.}
Let $C=C(n)>0$ be defined in Lemma \ref{Crokineq} and set
\[
\varepsilon(k):=\left\{
\begin{array}{llll}
 \frac{2^{n+1}\Lambda_F^{5n^2} V}{C\, \mathfrak{i}_M^{n-1}\, k},&& \text{ if }k <\frac{2^{n+2} \Lambda_F^{5n^2}V}{C\,\mathfrak{i}_M^n},
\tag{5.7}\label{new3.6}\\
 \left(\frac{ \Lambda_F^{5n^2} V}{C\, k}\right)^{\frac1n},&& \text{ if }k \geq \frac{2^{n+2} \Lambda_F^{5n^2}V}{C\,\mathfrak{i}_M^n}.
\end{array}
\right.
\]
We can easily  check that
\[
\left\{
\begin{array}{llll}
 \varepsilon(k)> \frac{\mathfrak{i}_M}{2},&& \text{ if }k < \frac{2^{n+2} \Lambda_F^{5n^2}V}{C\,\mathfrak{i}_M^n},\\
 \\
 \varepsilon(k)\leq \frac{\mathfrak{i}_M}{2},&& \text{ if }k \geq \frac{2^{n+2} \Lambda_F^{5n^2}V}{C\,\mathfrak{i}_M^n},
\end{array}
\right.
\]
which together with  (\ref{mestimate}) furnishes $m(\varepsilon(k))\leq k$, for any $k\in \mathbb{N}^+$.

The argument similar to the one of (\ref{3.3newcC*}) together with Lemma \ref{uniformBuHo} yields
\textcolor[rgb]{0.00,0.00,0.00}{\[
\overline{\lambda}_{k}\geq \overline{\lambda}_{m(\varepsilon(k))}\geq \frac{C'^{1+ \sqrt{|K|}\varepsilon(k)}}{\Lambda_F^{20n+4}}\frac1{\varepsilon^{2}(k)},\tag{5.8}\label{estimatelambdam}
\]}
where $C'=C'(n)<1$ is a positive constant.

\textcolor[rgb]{0.00,0.00,0.00}{The above inequality together with (\ref{new3.6}) yields
\[
\overline{\lambda}_{k}\geq \frac{{C_1}^{1+C_2\frac{\sqrt{|K|}\Lambda_F^{5n^2}V }{ \mathfrak{i}_M^{n-1}}}}{\Lambda^{10n^2+20n+4}_F} \left(\frac{\mathfrak{i}_M^{n-1} k}{V}\right)^2\geq  \frac{{C_3}^{1+\frac{\sqrt{|K|}\Lambda_F^{5n^2}V}{ \mathfrak{i}_M^{n-1}}}}{\Lambda^{21n^2}_F} \left(\frac{\mathfrak{i}_M^{n-1} k}{V}\right)^2, \text{ if }k < \frac{2^{n+2}\Lambda_F^{5n^2}V}{C\,\mathfrak{i}_M^n},
\]
while
\[
\overline{\lambda}_{k}\geq \frac{C_4^{1+C_5\sqrt{|K|}\Lambda_F^{5n}{V}^{\frac1n}}}{\Lambda_F^{30n+4}}\left( \frac{k}{V} \right)^{\frac{2}{n}}\geq \frac{C_6^{1+ \sqrt{|K|}\Lambda_F^{5n} {V} ^{\frac1n}}}{\Lambda_F^{32n}}\left( \frac{k}{V} \right)^{\frac{2}{n}}, \text{ if }k \geq \frac{2^{n+2}  \Lambda_F^{5n^2}V}{C\,\mathfrak{i}_M^n},
\]
where $C_1,\ldots,C_6$ are constants  depending only on $n$.
Therefore, we concludes the proof by choosing $\mathfrak{C}_2:=\min\{C_3,C_6\}$} and $\mathfrak{C}_3:=2^{n+2}/C$.
\end{proof}

\begin{proof}[Proof of Theorem \ref{Busertype}] \textcolor[rgb]{0.00,0.00,0.00}{
Let $\varepsilon(k)$  be defined as in (\ref{new3.6}).
Since $\mathbf{Ric}\geq -(n-1)|K|$, we have again relation (\ref{estimatelambdam}).  Now consider
\[
k\geq\max\left\{\frac{2^{n+2} \Lambda_F^{5n^2}V}{C\,\mathfrak{i}_M^n},\frac{|K|^{\frac{n}{2}}\Lambda_F^{5n^2}V}{C}  \right\},
\]
where $C$ is defined as in Lemma \ref{Crokineq}. Then it is easy to check that $\sqrt{|K|}\varepsilon(k)\leq1$ and hence (\ref{estimatelambdam}) together with (\ref{new3.6})  furnishes
\begin{align*}
\overline{\lambda}_{k}\geq  \frac{C'^2}{\Lambda_F^{20n+4}}\frac1{\varepsilon^{2}(k)}\geq \frac{C_3}{\Lambda_F^{32n}}\left( \frac{k}{V} \right)^{\frac{2}{n}}.
\end{align*}
The proof is concluded by choosing $C_4=2^{n+2}/C$.}
\end{proof}

\subsection{Lusternik-Schnirelmann dimension pair} \label{sectnew2}
In this section we provide better estimates for the eigenvalues for  Lusternik-Schnirel\-mann spectrum; see Section \ref{LSDP}. To do this, we need certain properties of the counting function that will be discussed for an arbitrary dimension pair.

\subsubsection{Counting function}
\begin{definition}\label{countfun}\it Let $(\mathscr{C},\Dim)$ be a dimension pair.
Given $\lambda>0$, the {\rm counting function} corresponding to $(\mathscr{C},\Dim)$ is defined by
\begin{align*}
N(\lambda)&:=\sup\left\{\Dim(A):\,A\in \mathscr{C},\, \sup_{u\in A}E(u)< \lambda\right\}.
\end{align*}
\end{definition}

The  relation between  the counting function and the spectrum of a dimension pair can be states as follows.
\begin{lemma}\label{basic}Let $(\mathscr{C},\Dim)$ be a dimension pair.
For any  $\lambda\in(0,+\infty)$, the following properties hold$:$

\begin{itemize}
	\item[(i)] If $N(\lambda)\leq k$ for some $k\in \mathbb{N}^+$, then
	\[
	\lambda_{N(\lambda)}<\lambda\leq \lambda_{N(\lambda)+1}\leq \lambda_{k+1};
	\]
	\item[(ii)] Suppose $\lambda_k\geq f(k)$ for any $k\in \mathbb{N}^+$, where $f$ is a strictly increasing nonnegative function. Thus,
	\[
	N(\lambda)< \left\lfloor f^{-1}(\lambda)\right\rfloor+1.
	\]
\end{itemize}
\end{lemma}
\begin{proof}
(i) It suffices to show  $\lambda_k<\lambda\leq\lambda_{k+1}$ if $N(\lambda)=k$. According to
Definition \ref{countfun}, there exists $B\in \mathscr{C}$ such that $\sup_{u\in B}E(u)<\lambda$ and $k=\Dim(B)$.  Then Theorem \ref{minmaxth} implies
\[
\lambda_k=\inf_{A\in \mathscr{C}_k}\sup_{u\in A}E(u)\leq \sup_{u\in B}E(u)<\lambda.
\]
On the other hand, if $\lambda_{k+1}<\lambda$, then $N(\lambda)\geq k+1$, which is a contradiction. Therefore, $\lambda_k<\lambda\leq\lambda_{k+1}$ as asserted.

(ii) We first claim that $N(f(k))<k$ for every $k\in \mathbb N^+$. In fact, if $N(f(k))\geq k$ for some $k\in \mathbb N^+$, then  there exists $A\in \mathscr{C}_k$ with $\lambda_k\leq \sup_{u\in A}E(u)<f(k)$, which contradicts the assumption $\lambda_k\geq f(k)$. Thus, for any $\lambda>0$, since $\lambda<f(\lfloor f^{-1}(\lambda)\rfloor+1)$, the claim implies $N(\lambda)\leq N(f(\lfloor f^{-1}(\lambda)\rfloor+1))<\lfloor f^{-1}(\lambda)\rfloor+1$.
\end{proof}

\begin{lemma}\label{countrelatoin}Given a compact {\rm FMMM}, for any $\alpha\in\{1,2\}$, we have
\[
\left\{
\begin{array}{llll}
m^\alpha_{LS}(k)=m^\alpha_K(k),\\
 \\
 N^\alpha_{LS}(\lambda)=N^\alpha_K(\lambda), \ \forall\,\lambda\in (0,+\infty),\\
 \\
 m^\alpha_{LS}(k)\leq N^\alpha_{LS}(\lambda^{LS,\alpha}_k+\epsilon),\ \forall\epsilon>0,
\end{array}
\right.
\]
where $m^\alpha_{LS}(k)$ $($resp., $m^\alpha_K(k)$$)$ stands for  the multiplicity  of the $k^{th}$ eigenvalue of $(\mathscr{C}^\alpha,\Dim_{LS})$ $($resp., $(\mathscr{D}^\alpha,\Dim_{K})$$),$ while
$N^\alpha_{LS}$ $($resp., $N^\alpha_K$$)$ denotes the counting function corresponding to $(\mathscr{C}^\alpha,\Dim_{LS})$ $($resp., $(\mathscr{D}^\alpha,\Dim_{K})$$).$
\end{lemma}
\begin{proof}
 Theorem
\ref{therealconnection} implies that $m^\alpha_K(k)=m^\alpha_{LS}(k)$ and
$N^\alpha_K(\lambda)=N^\alpha_{LS}(\lambda)$. Suppose that the multiplicity of the $k^{\rm th}$ eigenvalue $\lambda:=\lambda^{LS,\alpha}_k$ is $l\in\mathbb N$. According to Lemma \ref{muti2}, there is a compact set $\mathfrak{K}_\lambda\subset \mathcal {S}$ with $\Dim_{LS}(\mathfrak{K}_\lambda)\geq l$ and $\sup_{u\in \mathfrak{K}_\lambda}E(u)=\lambda$. This implies
$m^\alpha_{LS}(k)=l\leq \Dim_{LS}(\mathfrak{K}_\lambda)\leq N^\alpha_{LS}(\lambda+\epsilon)$ for every $\epsilon>0$.
\end{proof}

\subsubsection{Gromov type estimate}

In order to estimate the lower bounds for the Lusternik-Schnirelmann eigenvalues, we need some results concerning the weighted Ricci curvature.
By Lemma \ref{Rho} we immediately have the following result.
\begin{lemma}\label{coveringandmin}
Given $N\in [n,+\infty)$, $K\in \mathbb R$ and $d>0$, let $(M,F,d\mathfrak{m})$ be an $n$-dimensional closed {\rm FMMM} with $\mathbf{Ric}_N\geq(N-1)K$ and $\diam(M)=d$.
Given $r>0$, let $\{p_i\}_{i=1}^m$ be a complete $r$-package. The following properties hold$:$

\begin{itemize}
	\item[(i)]
	$
	m=\Ca_M(r)\leq \max\left\{ e^{(N-1)d\sqrt{|K|}}\left( \frac{d}{r} \right)^N, 1\right\};
	$
	\item[(ii)] For any $x\in M$, the number of balls $B_{p_i}(4r)$ containing $x$ is not larger then $12^Ne^{12(N-1)r\sqrt{|K|}}$, i.e.,
	\[
	{\rm card}\left\{p\in \{p_i\}_{i=1}^m:\, x\in B_{p}(4r)  \right\}\leq 12^Ne^{12(N-1)r\sqrt{|K|}}.
	\]
\end{itemize}
\end{lemma}

\begin{lemma}\label{grandientes}Given $N\in [n,+\infty)$ and $K\in \mathbb R$,
let $(M,F,d\mathfrak{m})$ be an $n$-dimensional closed {\rm FMMM} with  $\mathbf{Ric}_N\geq (N-1)K$.
Then   for  any  ball $B_p(R)$, we have
\[
\ds\int_{B_p(R)}|u-u_{p,R}|^2d\mathfrak{m}\leq 2^{N+2}  R^2 e^{ {(N-1) \sqrt{|K|}R} }\ds\int_{B_p(2R)}F^{*2}(du)d\mathfrak{m}, \ \forall u\in \mathscr{X},
\]
where $u_{p,R}$ is the mean value of $u$ on $B_p(R)$, i.e.,
\[
u_{p,R}:=\frac{1}{\mathfrak{m}(B_p(R))}\ds\int_{B_p(R)}u d\mathfrak{m}.
\]
\end{lemma}

\begin{proof}Without loss of generality, we can assume $u\in C^\infty(M)$.
A direct calculation yields
\[
|u-u_{p,R}|(z)\leq \frac{1}{\mathfrak{m}(B_p(R))}\ds\int_{B_p(R)}|u(z)-u(x)| d\mathfrak{m}(x),
\]
which together with the H\"older inequality implies
\[
|u-u_{p,R}|^2(z)\leq \frac{1}{\mathfrak{m}(B_p(R))}\ds\int_{B_p(R)}|u(z)-u(x)|^2 d\mathfrak{m}(x).
\]
Integrating the above inequality over $B_p(R)$, we obtain
\begin{align*}
\ds\int_{B_p(R)}|u-u_{p,R}|^2(z)d\mathfrak{m}(z)
\leq \frac{1}{\mathfrak{m}(B_p(R))}\ds\int_{B_p(R)}d\mathfrak{m}(z) \ds\int_{B_p(R)}|u(z)-u(x)|^2 d\mathfrak{m}(x).\tag{5.8}\label{2.3CC}
\end{align*}
Let $\gamma_{z,x}$ be a unit speed minimal geodesic from $z$ to $x$, both points belonging to $B_p(R)$. Since for a.e. $s\in [0,d_F(z,x)]$ we have $$\frac{d}{ds} u({\gamma}_{z,x}(s))=du({\gamma}_{z,x}(s))\dot{\gamma}_{z,x}(s)\leq F^*(du)\circ{\gamma}_{z,x}(s)F(\dot{\gamma}_{z,x}(s)) = F^*(du)\circ{\gamma}_{z,x}(s),$$ it follows that
\begin{align*}
|u(z)-u(x)|\leq \int^{d_F(z,x)}_0 F^*(du)\circ{\gamma}_{z,x}(s)ds\leq \left( \int^{d_F(z,x)}_0 F^{*2}(du)\circ{\gamma}_{z,x}(s)ds \right)^\frac12 \left( 2R \right)^{\frac{1}{2}},
\end{align*}
 which
together with (\ref{2.3CC}) yields
\begin{align*}
&\ds\int_{B_p(R)}|u-u_{p,R}|^2(z)d\mathfrak{m}(z)\\
\leq
&\frac{2R}{\mathfrak{m}(B_p(R))}\ds\int_{(z,x)\in B_p(R)\times B_p(R)}\left( \int^{d_F(z,x)}_0 F^{*2}(du)\circ{\gamma}_{z,x}(s)ds \right)d\mathfrak{m}_\times(z,x).\tag{5.9}\label{2.4}
\end{align*}
By letting $A_1=A_2:=B_p(R)$ and $W:=B_p(2R)$,
 Theorem \ref{CheegerColding} furnishes
\begin{align*}
&\ds\int_{(z,x)\in B_p(R)\times B_p(R)}\left( \int^{d_F(z,x)}_0 F^{*2}(du)\circ{\gamma}_{z,x}(s)ds \right)d\mathfrak{m}_\times(z,x)\\
\leq &2^{N+1 } R e^{(N-1) {\sqrt{|K|} R} }\mathfrak{m}(B_p(R))\ds\int_{B_p(2R)} F^{*2}(du) d\mathfrak{m},
\end{align*}
which together with (\ref{2.4}) concludes the proof.
\end{proof}

Inspired by Hassannezhad, Kokarev and Polterovich \cite{HKP}, we get the following estimate.
\begin{theorem}\label{lsestimate}Given $N\in [n,+\infty)$, $K\in \mathbb R$ and $d>0$,
let $(M,F,d\mathfrak{m})$ be an $n$-dimensional closed {\rm FMMM} with
 \[
 \mathbf{Ric}_N\geq (N-1)K,\ \diam(M)=d.
 \]
 Then there exists a positive constant $\mathfrak{C}_4=\mathfrak{C}_4(N)$ depending only on $N$ such that for any $\alpha\in \{1,2\}$,
\[
N^\alpha_{K}(\lambda)\leq \max\left\{\mathfrak{C}_4^{1+\sqrt{|K|}d}d^N\lambda^{\frac{N}2},\,1\right\},\ \forall\,\lambda>0.
\]
\end{theorem}
\begin{proof} Given $\lambda>0$, set
$E_\lambda:=\{u\in \mathcal {S}:\,E(u)<\lambda\}$.
For any $r>0$, let $\{p_i\}_{i=1}^m$ be a complete $r$-package. According to Lemma \ref{propertiesDirdom}, $\{B_i:=B_{p_i}(2r)\}_{i=1}^m$ is a covering of $M$. We define a linear, continuous and odd map $\Phi_{\lambda,r}:\mathscr{X}\rightarrow \mathbb{R}^m$  by
\begin{align*}
u\mapsto \left( \frac{1}{\mathfrak{m}(B_1)}\ds\int_{B_1}u d\mathfrak{m},\ldots,   \frac{1}{\mathfrak{m}(B_m)}\ds\int_{B_m}u d\mathfrak{m}      \right).
\end{align*}

We claim that $0\notin\Phi_{\lambda,r}(E_\lambda)$ provided that $r>0$ satisfies
\[
\lambda\leq \left[ 2^{N+4}\cdot 12^N\cdot r^2 e^{14(N-1)\sqrt{|K|}r}\right]^{-1}.\tag{5.10}\label{2.6}
\]
 By contradiction, assume that there exist $r>0$ and $u\in E_\lambda$ such that (\ref{2.6}) holds and $\Phi_{\lambda,r}(u)=0$. Hence,
\[
u_{B_i}:= \frac{1}{\mathfrak{m}(B_i)}\ds\int_{B_i}u d\mathfrak{m}=0,
\]
thus $|u-u_{B_i}|^2=u^2$ for all $ i\in\{1,\ldots,m\}.$
 Lemmas \ref{grandientes} and \ref{coveringandmin}/(ii) yield
\begin{align*}
\ds\int_Mu^2d\mathfrak{m}\leq& \sum_{i=1}^m\ds\int_{B_i}u^2 d\mathfrak{m}\leq 2^{N+2}\cdot (2r)^2 e^{2 {(N-1)\sqrt{|K|}r} }\sum_{i=1}^m\ds\int_{B_{p_i}(4r)}F^{*2}(du)d\mathfrak{m}\\
\leq &2^{N+4}\cdot 12^N\cdot r^2 e^{14(N-1)\sqrt{|K|}r}\ds\int_{M}F^{*2}(du)d\mathfrak{m}<2^{N+4}\cdot 12^N\cdot r^2 e^{14(N-1)\sqrt{|K|}r}\lambda\ds\int_M u^2 d\mathfrak{m},
\end{align*}
which implies
\[
\lambda>\left[2^{N+4}\cdot 12^N\cdot r^2 e^{14(N-1)\sqrt{|K|}r}\right]^{-1},
\]
contradicting (\ref{2.6}). Thus for every $\lambda,r>0$ verifying (\ref{2.6}), $\Phi_{\lambda,r}: E_\lambda\rightarrow \mathbb{R}^{m}\backslash\{0\}$ is continuous and odd.

Fix $\alpha\in\{0,1\}$. For every $A\in \mathscr{D}^\alpha$ with $A\subset E_\lambda$, by the map $\Phi_{\lambda,r}$ constructed above and   Definition \ref{Krasnodef}, we have
$
\Dim_K(A)\leq m$ which implies \[N^\alpha_K(\lambda)\leq   m.\tag{5.11}\label{3.4CC*}
\]

Let us choose
\[
r_0:=\left(\lambda\cdot2^{N+4}\cdot 12^N\cdot  e^{14(N-1)\sqrt{|K|}d}   \right)^{-\frac12}.\tag{5.12}\label{2.7}
\]

\textbf{Case 1.} If $r_0\leq d$, then
\begin{align*}
\lambda=\left[{  2^{N+4}\cdot 12^N\cdot r_0^2   e^{14(N-1)\sqrt{|K|}d}}\right]^{-1}\leq \left[{  2^{N+4}\cdot 12^N\cdot r_0^2  e^{14(N-1)\sqrt{|K|}r_0}}\right]^{-1},
\end{align*}
which together with (\ref{2.6}) implies  $0\notin \Phi_{\lambda,r_0}(E_\lambda)$. Note that $\Phi_{\lambda,r_0}$ is  constructed by  a  complete $r_0$-package.  Lemma \ref{coveringandmin}/(i) and relations (\ref{3.4CC*}) and (\ref{2.7}) furnish
\begin{align*}
N^\alpha_K(\lambda)\leq m\leq  e^{(N-1)d\sqrt{|K|}}\left( \frac{d}{r_0} \right)^N\leq \mathfrak{C}_4^{1+d\sqrt{|K|}}d^N\lambda^{\frac{N}2},\tag{5.13}\label{3.6CC*}
\end{align*}
where
\[
 \mathfrak{C}_4:=\max\left\{2^{\frac{N(N+4)}{2}}\cdot 12^{\frac{N^2}{2}},\, e^{(7N+1)(N-1)}\right\}.
\]

\textbf{Case 2.}
If $r_0>d=:r_*$, then
\[
B_1(r_0)=M=B_1(r_*),\ 1=m=\Ca_M(r_0)=\Ca_M(r_*).
\]
Now it follows from (\ref{2.7}) that
\[
\lambda<\left[{  2^{N+4}\cdot 12^N\cdot r_*^2  e^{14(N-1)\sqrt{|K|}r_*}}\right]^{-1}.
\]
Now we consider $r_*$ instead of $r_0$, since the complete $r_0$-package coincides with the complete $r_*$-package. The same argument yields $0\notin\Phi_{\lambda,r_*}(E_\lambda)$ and hence, $N^\alpha_K(\lambda)\leq m=1$.
\end{proof}

\begin{theorem}\label{thesecondtheorem}
Given  $N\in [n,+\infty)$, $K\in \mathbb R$ and $d>0$, let $(M,F,d\mathfrak{m})$ be an $n$-dimensional closed {\rm FMMM} with
\[
\mathbf{Ric}_N\geq (N-1)K,\ \diam(M)=d.
\]
Then there exists a  constant $\mathfrak{C}_5=\mathfrak{C}_5(N)>0$ depending only on $N$ such that for any $\alpha\in\{0,1\}$,
\[
\overline{\lambda}^{LS,\alpha}_{k}=\overline{\lambda}^{K,\alpha}_{k}\geq \frac{\mathfrak{C}_5^{1+d\sqrt{|K|}}}{d^2}\, k^{\frac{2}N},\ \forall\,k\in \mathbb{N}^+,\tag{5.14}\label{lseigenestimate}
\]
where $\overline{\lambda}^{LS,\alpha}_{k}$ $($resp., $\overline{\lambda}^{K,\alpha}_{k}$$)$ denotes the $k^{th}$ positive eigenvalue of $(\mathscr{C}^\alpha,\Dim_{LS})$ $($resp., $(\mathscr{D}^\alpha,\Dim_{K})$$).$
Moreover, if $N\in \mathbb N,$ there is a constant $\mathfrak{C}_6=\mathfrak{C}_6(N)>0$ depending only on $N$ such that
\[
m^\alpha_{LS}(k)=m^\alpha_K(k)\leq    \mathfrak{C}_6^{1+d\sqrt{|K|}}\left[ \left(d\sqrt{|K|} \right)^N+k^N \right], \ \forall\,k\in \mathbb{N}^+.\tag{5.15}\label{lsmutiestimate}
\]
\end{theorem}
\begin{proof}
On one hand, (\ref{lseigenestimate}) follows by Theorem \ref{lsestimate} and Lemma \ref{basic}/(i). On the other hand, (\ref{lsmutiestimate}) follows by Lemma \ref{countrelatoin}, Theorem \ref{lsestimate} and Theorem \ref{Chees}, respectively.
\end{proof}

\section{Application:  eigenvalues of weighted Riemannian manifolds}\label{section6}
Let $(M,g, e^{-f}d\vol_g)$ be a closed weighted Riemannian manifold, that is, $(M,g)$ is a closed Riemannian manifold and $f\in C^\infty(M)$ is a smooth function. The {\it Bakry-\'Emery Laplacian} is
\[
\Delta_fu:=\Delta_g u-g(\nabla f, \nabla u),\ \forall u\in C^\infty(M).
\]
According to Setti \cite{Se}, the spectrum  $\{\lambda^{\Delta_f}_k\}_{k=1}^\infty$ of $\Delta_f$ is purely discrete and satisfies
\[
0=\lambda^{\Delta_f}_1<\lambda^{\Delta_f}_2\leq \cdots \leq \lambda^{\Delta_f}_k\leq\ldots \nearrow +\infty,
\]
while the corresponding eigenfunctions are smooth and form a basis of $L^2(M)$.

The weighted Riemannian manifold $(M,g, e^{-f}d\vol_g)$ can be viewed as a compact FMMM $(M,F,d\mathfrak{m})$, where $F:=\sqrt{g}$ and $d\mathfrak{m}:=e^{-f}d\vol_g$. In particular, the gradient $\nabla$ of $(M,F,d\mathfrak{m})$ coincides the one of $(M,g, e^{-f}d\vol_g)$, whereas
the Laplacian $\Delta$ of $(M,F,d\mathfrak{m})$ is exactly the  Bakry-\'Emery Laplacian $\Delta_f$. Therefore, $\lambda^{\Delta_f}_k$ is a critical value of the canonical energy functional
\[
E(u)=\frac{\ds\int_MF^{*2}(du)d\mathfrak{m}}{\ds\int_Mu^2 d\mathfrak{m}}=\frac{\ds\int_M g(\nabla u, \nabla u) e^{-f}d\vol_g}{\ds\int_M u^2 e^{-f}d\vol_g}.
\]
This fact yields the following min-max principle.
\begin{theorem}\label{Countminmxatheorem} Given $k\in \mathbb{N}^+$,  we have
\[
\lambda^{\Delta_f}_k=\min_{V\in \mathscr{H}_k}\max_{u\in V\backslash\{0\}}E(u),
\]
where $\mathscr{H}_k=\{V\subset \mathscr{X}:\, V\text{ is a linear subspace with }\Dim_C(V)=k\}$.
\end{theorem}

We now show the following result.
\begin{theorem}\label{wegitedLS}Let $(M,g, e^{-f}d\vol_g)$ be a closed weighted Riemannian manifold.
The Lusternik-Schnirelmann spectrum is exactly the spectrum of the Bakry-\'Emery Laplacian.
\end{theorem}
\begin{proof}
Fix $\alpha\in\{1,2\}$, let $(\mathscr{C}^\alpha,\Dim_{LS})$ be defined as in Definition \ref{defls2}.
 We are going to show $ \lambda^{LS,\alpha}_k=\lambda_k^{\Delta_f}$ for all $k\in \mathbb{N}^+$.

We first claim that $\lambda^{LS,\alpha}_k\leq \lambda^{\Delta_f}_k$. In fact, since $\Delta_f$ is linear and self-adjoint w.r.t $(\cdot,\cdot)_{L^2}$ (see (\ref{innernormL2})),
  we can suppose that the eigenfunctions $\{u_i\}_{i=1}^k$ corresponding to $\{\lambda^{\Delta_f}_i\}_{i=1}^k$ are orthonormal  w.r.t $(\cdot,\cdot)_{L^2}$. Set $V_k:=\text{Span}\{u_1,\ldots,u_k\}$. Thus, for $u\in \mathcal {S}\cap V_k$, we have $u=\sum_{i=1}^k a_i u_i$, where $\sum_{i=1}^k a^2_i=1$.
Hence,
\begin{align*}
E(u)=-\frac{\ds\int_M u \Delta u d\mathfrak{m}}{\ds\int_Mu^2 d\mathfrak{m}}=\sum_{i=1}^k\lambda^{\Delta_f}_i a_i^2\leq \lambda^{\Delta_f}_k\Longrightarrow \sup_{u\in \mathcal {S}\cap V_k} E(u)=\lambda^{\Delta_f}_k.
\end{align*}
Since $\mathcal {S}\cap V_k\in  \mathscr{C}^{LS,\alpha}_k$ (see Proposition \ref{lsdimp}), we have
\[
\lambda^{LS,\alpha}_{k}=\inf_{A\in \mathscr{C}^{LS,\alpha}_k}\sup_{u\in A}E(u)\leq \sup_{u\in \mathcal {S}\cap V_k} E(u)=\lambda^{\Delta_f}_k.
\]

We now show  $\lambda^{\Delta_f}_k\leq \lambda^{LS,\alpha}_k$. Let $\{\phi_j\}_{j=1}^k$ be the eigenfunctions corresponding to $\{\lambda^{LS,\alpha}_j\}_{j=1}^k$. By suitable modification to the proof of Theorem \ref{compareRiemannen}, we can show that $\{\phi_j\}_{j=1}^k$ satisfy
\[
  \ds\int_M \phi_i\phi_j d \mathfrak{m}=\delta_{ij}, \ \ds\int_M g(\nabla \phi_i, \nabla \phi_j) d \mathfrak{m}=\lambda^{LS,\alpha}_i\delta_{ij}.
  \]
Let $W_k:=\text{Span}\{\phi_1,\ldots,\phi_k\}$. Then Theorem \ref{Countminmxatheorem} together with the same argument as above implies
\[
\lambda^{\Delta_f}_k\leq \sup_{u\in W_k\backslash\{0\}}E(u)=\sup_{u\in W_k\cap \mathcal {S}}E(u)=\lambda^{LS,\alpha}_k,
\]
which concludes the proof.
\end{proof}


A direct calculation yields that the weighted Ricci curvature $\mathbf{Ric}_N$ is exaclty
the  {\it $(N-n)$-Bakry-\'Emery Ricci tensor} of $(M,g, e^{-f}d\vol_g)$, i.e.,
\[
\mathbf{Ric}_N=\mathbf{Ric}+\text{Hess}(f)-\frac{1}{N-n}df\otimes df.
\]
This fact together with
Theorems \ref{wegitedLS},  \ref{Chees} and  \ref{thesecondtheorem} yields the following result.
\begin{theorem}Given $N\in (n,+\infty)\cap \mathbb N$, $K\in \mathbb R$ and $d>0$,
let $(M,g, e^{-f}d\vol_g)$ be an $n$-dimensional closed weighted Riemannian manifold with
\[
\mathbf{Ric}_{N}\geq (N-1)K,\ \diam(M)= d.
\]
Then there exist two positive constants $C_1=C_1(N)$ and $C_2=C_2(N)$ such that
\[
 \frac{C_1^{1+d\sqrt{|K|}}}{d^2}\, k^{\frac{2}N}\leq \overline{\lambda}^{\Delta_f}_k\leq \frac{(N-1)^2}{4}|K|+C_2\left(\frac{k}{d}\right)^2, \ \forall\,k\in \mathbb{N}^+.
\]
where $\overline{\lambda}^{\Delta_f}_k$ is the $k^{th}$ positive eigenvalue of $\Delta_f$.
\end{theorem}

\appendix
\section{}\label{section7}
\subsection{Properties of $\mathcal {S}$ and $\mathbb{P}(\mathscr{X})$}\label{B-Fsec}
In this section, we investigate $\mathcal {S}$ and $\mathbb{P}(\mathscr{X})$. First,
 we recall the definition of Banach-Finsler manifolds in the sense of Palais, see  Palais \cite[Definition 2.10, Definition 3.5]{PS} and Struwe \cite[p.\,77]{S}.
\begin{definition}[\cite{PS,S}]\label{PSFinslerde} {\it Given $r\geq1$,
let $X$ be a $C^r$-Banach manifold modeled on a Banach space $V$,  and let $\|\cdot\|:TX\rightarrow \mathbb{R}$ be a function. $(X,\|\cdot\|)$ is called a $C^r$-{\rm Banach-Finsler manifold}
  if for each $k>1$ and each $x_0\in X$, there exists a bundle chart $\varphi:O\times V\approx TX|_O$ for $TX$ with $O$ a neighborhood of $x_0\in X$ such that $\|\cdot\|\circ\varphi$ satisfies:

\begin{itemize}
	\item[(i)] for each $x\in O$, the function $v\in V\mapsto \|\varphi(x,v)\|$ is an admissible norm for $V;$
	\item[(ii)] $\frac{1}{k}\|\varphi(x,v)\|\leq \|\varphi(x_0,v)\|\leq k\|\varphi(x,v)\|$ for all $x\in O$ and $v\in V$.
\end{itemize}
\noindent A Banach-Finsler manifold $(X,\|\cdot\|)$ is said to be {\rm complete} if each component of $X$ is complete under the
 metric induced by $\|\cdot\|$.}
\end{definition}

We also  need the following result, see  Palais \cite[Theorem 8, Corollary, p.3]{PS2}, \cite[Theorem 3.6, Theorem 5.9]{PS} and Zeidler \cite[Theorem 73.C, Example 73.41]{Z} for the proofs.
\begin{lemma}[{\cite{PS2,PS,Z}}]\label{Banachco} The following properties  hold.

\begin{itemize}
	\item[(i)] Let $X$ be a Banach space and $f:X\rightarrow \mathbb{R}$ be a $C^k$-function, $k\geq 1$. If $Df(x)\neq0$ for all the solutions $x$ of the equation $f(x)=0$, then the solution set $S:=f^{-1}(0)$  is a closed submanifold of $X$ and especially, is a
	$C^k$-Banach manifold.
	\item[(ii)] If $(X,\|\cdot\|)$ is a complete $C^1$-Banach-Finsler manifold and $N$ is a closed $C^1$-submanifold  of $X$,  then $(N,\|\cdot\||_{TN})$
	a complete Banach-Finsler manifold as well.
	\item[(iii)] Every paracompact Banach manifold is an ANR $($i.e., absolute neighborhood retract$).$
	\item[(iv)] An ANR is an AR if and only if it is contractible.
\end{itemize}
\end{lemma}

\begin{proposition}\label{half2.5}
 $(\mathcal {S},\|\cdot\|\,|_{T\mathcal {S}})$ is a complete $C^\infty$-Banach-Finsler  manifold and an ANR.
\end{proposition}
\begin{proof}
Consider the function $h:\mathscr{X}\rightarrow \mathbb{R}$ defined by $h(u):=\|u\|_{L^2}^2-1$.
It is easy to see that
\begin{align*}
Dh(u)(\phi)&=\left.\frac{d}{dt}\right|_{t=0}h(u+t\phi)=2\ds\int_Mu\phi d\mathfrak{m};\tag{A.1}\label{4.2**!!}\\
D^2h(u)(\phi_1,\phi_2)&=\left.\frac{d}{dt}\right|_{t=0}Dh(u+t\phi_2)(\phi_1)=2\ds\int_M\phi_1\phi_2d\mathfrak{m};\\
D^3h(u)(\phi_1,\phi_2,\phi_3)&=\left.\frac{d}{dt}\right|_{t=0}D^2h(u+t\phi_3)(\phi_1,\phi_2)=0.
\end{align*}
The H\"older inequality together with the compactness of $M$ then yields $h\in C^\infty(\mathscr{X})$. Moreover, if $u\in \mathscr{X}$ satisfies $Dh(u)=0$,  (\ref{4.2**!!}) then implies $u=0$. Thus,  $Dh(u)\neq0$ for any $u\in h^{-1}(0)=\mathcal {S}$. It follows from  Lemma \ref{Banachco}/(i)(ii) that  $(\mathcal {S},\|\cdot\|\,|_{T\mathcal {S}})$ is a complete $C^\infty$-Banach-Finsler  manifold. In particular, ${\mathcal {S}}$ is   paracompact since it is metrizable. Thus,  Lemma \ref{Banachco}/(iii) furnishes that ${\mathcal {S}}$ is an ANR.
\end{proof}

In the sequel, we prove that ${\mathcal {S}}$ is an AR while $\mathbb{P}(\mathscr{X})$ is an ANR. Before doing this, we recall that the unit sphere in an infinite-dimensional Hilbert space is contractible (cf.  Kakutani \cite{K}).

\begin{proposition}\label{contr}
${\mathcal {S}}$ is contractible and hence, an AR.
\end{proposition}
\begin{proof}
 Recall that $(\mathscr{X},(\cdot,\cdot))$  is a separable Hilbert space, where $(\cdot,\cdot)$ is defined by (\ref{3.1}). Thus,   the unit sphere $\mathscr{S}:=\{u\in \mathscr{X}:\|u\|_1=1\}$ in $(\mathscr{X},(\cdot,\cdot))$
is contractible (cf. \cite{K}), where   $\|\cdot\|_1$  denotes the norm induced by $(\cdot,\cdot)$. It follows from (\ref{new new 5.2})  that
 $f\in \mathscr{X}$ with $\|f\|_{L^2}=0$ if and only if   $\|f\|_1=0$. From this fact, one can easily prove that $\mathcal {S}$ is homeomorphic to $\mathscr{S}$ by considering the map $t: \mathcal {S}\rightarrow \mathscr{S}$, $u\mapsto \frac{u}{\|u\|_{1}}$.
 Hence, ${\mathcal {S}}$ is contractible. Since $\mathcal {S}$ is an ANR,
 the statement follows by Lemma \ref{Banachco}/(iv).
\end{proof}

\begin{proof}[Proof of Proposition \ref{compS}] The first part of  the proposition follows from Propositions \ref{half2.5} and \ref{contr}, which together with
 Proposition \ref{srongDE} furnishes that $i^*E$ is a $C^1$-function on $\mathcal {S}$.
\end{proof}

\begin{proposition}\label{finalPX}
$\mathbb{P}(\mathscr{X})$ is a paracompact Banach topological manifold and hence, a normal ANR.
\end{proposition}
\begin{proof}
Since $\mathcal {S}$ is a $C^\infty$-Banach manifold and $\mathfrak{p}:\mathcal {S}\rightarrow \mathbb{P}(\mathscr{X})$ is a twofold covering, a standard argument yields that $\mathbb{P}(\mathscr{X})$ is a topological Banach-Finsler   manifold.

We now show that $\mathbb{P}(\mathscr{X})$ is paracompact. Given any open covering $\{U_\alpha\}$ of $\mathbb{P}(\mathscr{X})$, we can obtain a refinement $\{V_\beta\}$ of $\{U_\alpha\}$  and an open covering $\{+\mathcal {V}_\beta, -\mathcal {V}_\beta\}$ of $\mathcal {S}$ such that $\mathfrak{p}|_{\pm \mathcal {V}_\beta}:\pm \mathcal {V}_\beta\rightarrow V_\beta$ are homeomorphisms. Since $\mathcal {S}$ is paracompact (see Proposition \ref{half2.5}), there exists a locally finite refinement $\{\mathcal {O}_\gamma\}$ of $\{\pm\mathcal {V}_\beta\}$. Thus, each $\mathfrak{p}|_{\mathcal {O}_\gamma}:\mathcal {O}_\gamma\rightarrow \mathfrak{p}(\mathcal {O}_\gamma)$  is a homeomorphism. In particular,  $\{\mathfrak{p}(\mathcal {O}_\gamma)\}$ is a refinement of $\{U_\alpha\}$.

 On the other hand, for each $[u]\in \mathbb{P}(\mathscr{X})$, there are two open neighbourhoods $N_\pm\subset \mathcal {S}$ of $\pm u$ such that each of them intersects  only finitely many of the sets in $\{\mathcal {O}_\gamma\}$. Let $N_{[u]}:=\mathfrak{p}(N_+)\cap \mathfrak{p}(N_-)$, which is an open set. Thus, if $N_{[u]}$ intersects some $\mathfrak{p}(\mathcal {O}_\gamma)$, then $\mathcal {O}_\gamma$ must intersect at least one of $N_\pm$ (but not vice versa), which implies
\[
{\rm card}\{\mathfrak{p}(\mathcal {O}_\gamma):\, \mathfrak{p}(\mathcal {O}_\gamma)\cap N_{[u]}\neq\emptyset\}\leq {\rm card}\{\mathcal {O}_\gamma:\, \mathcal {O}_\gamma\cap (N_+\cup N_-)\neq\emptyset\}<+\infty,
\]
where ${\rm card}$ denotes the cardinality of a set.
Hence, $\{\mathfrak{p}(\mathcal {O}_\gamma)\}$ is locally finite and therefore, $\mathbb{P}(\mathscr{X})$ is paracompact and normal. Now it follows from  Lemma \ref{Banachco}/(iii) that $\mathbb{P}(\mathscr{X})$ is an ANR.
\end{proof}

\subsection{Properties of Dirichlet regions}\label{propertiesDirdomAb}
\begin{proof}[Proof of Lemma \ref{propertiesDirdom}]

 We first show that $B_{p_i}\left({r}  \right)\subset D_i\subset B_{p_i}(2r)$ for each $i\in\{1,\ldots,m\}$. Since $\{B_{p_i}(r)\}_{i=1}^m$ is   a maximal family of disjoint $r$-balls, one gets $B_{p_i}\left({r}  \right)\subset D_i$ for each $i\in\{1,\ldots,m\}$.
On the other hand,
note that $\{B_{p_i}(2r)\}_{i=1}^m$ is a covering of $M$. Thus, given $i\in\{1,\ldots,m\}$, for any $q\in D_i$, we claim $d_F(p_i,q)<2r$. Otherwise,  there would exist a point $p_j\neq p_i$ such that $q\in B_{p_j}(2r)$ and hence, $d_F(p_j,q)< d_F(p_i,q)$, which is a contradiction, hence $D_i\subset B_{p_i}(2r)$.

 By Definition \ref{Dirdef} we have that $\{D_i\}_{i=1}^m$ is a  covering of $M$.  We show
 $\mathfrak{m} (D_i\cap D_j)=0$ if $i\neq j.$
 In order to do this, set
\[
f_{ji}(x):=d_F(p_j,x)-d_F(p_i,x),\ A:=f_{ji}^{-1}(0)\cap (\text{Cut}_{p_i}\cup \text{Cut}_{p_j}),\ B:=f_{ji}^{-1}(0)-A.\tag{A.2}\label{definfij}
\]
Let us equip $f_{ji}^{-1}(0)$ with the induced topology from $M$.
Since $B$ is an open set of $f_{ji}^{-1}(0)$, there exists an open subset $N$ of $M$ such that $B=N\cap f_{ji}^{-1}(0)$.
In the sequel, we show that $B$ is an $(n-1)$-dimensional submanifold of $N$ and hence, $\mathfrak{m}(B)=0$.
Note that $f_{ji}|_N$ is smooth.
Once we show $df_{ji}(q)\neq0$ for any $q\in B$, the claim follows. By contrary, if $df_{ji}(q)=0$, one has that $\nabla d_F(p_i,x)|_{x=q}=\nabla d_F(p_j,x)|_{x=q}$, which yields $p_i=p_j$ due to $q\notin (\text{Cut}_{p_i}\cup \text{Cut}_{p_j})$ and $d_F(p_i,q)=d_F(p_j,q)$.  Since both $A$ and $B$ are zero-measurable, $\mathfrak{m} (D_i\cap D_j)\leq\mathfrak{m} (f_{ji}^{-1}(0))=\mathfrak{m} (A)+\mathfrak{m} (B)=0$ for $i\neq j.$

We now show that
$\mathfrak{m}(\text{int}(D_i))=\mathfrak{m}(D_i)$ for each $i\in\{1,\ldots,m\}$.
For each $j\neq i$, set $D_{ij}:=\{q\in M:\, d_F(p_i,q)\leq d_F(p_j,q)\}$.
Then $D_i=\bigcap_{j\neq i}D_{ij}$ and hence,
\[
\text{int}(D_i)=\bigcap_{j\neq i}\text{int}(D_{ij})=\bigcap_{j\neq i}\{ q\in M:\, d_F(p_i,q)<d_F(p_j,q) \},\tag{A.3}\label{3.1CC**}
\]
which implies
\[
\displaystyle D_i-\text{int}(D_i)={\bigcup}_{ j\neq i} \left(D_i\bigcap f^{-1}_{ji}(0)\right),
\]
where $f_{ji}$ is defined by (\ref{definfij}). The claim follows by $\mathfrak{m}(f^{-1}_{ji}(0))=0$.

Finally we show that
$\text{int}(D_i)$ is starlike with respect to $p_i$ for each $i\in\{1,\ldots,m\}$.
Given any $q\in \text{int}(D_i)$, let $\gamma(t)$ be a unit speed minimal geodesic from $p_i$ to $q$. For any $j\neq i$, consider
\[
f_{ji}(\gamma(t))=d_F(p_j,\gamma(t))-d_F(p_i,\gamma(t))=d_F(p_j,\gamma(t))-t=:\rho_j(\gamma(t))-t.
\]
If $q$ is not a cut point of $p$ along $\gamma(t)$, we have
\begin{align*}
\frac{d}{dt} f_{ji}(\gamma(t))=g_{\nabla \rho_j}(\nabla \rho_j,\dot{\gamma}(t))-1\leq F(\nabla \rho_j)F(\dot{\gamma}(t))-1= 0,\ \forall\, t\in (0,d_F(p_i,q)]. \tag{A.4}\label{4.4new*}
\end{align*}
Since $q\in  \text{int}(D_{ij})$ for all $j\neq i$ (see (\ref{3.1CC**})), we have
$
f_{ji}(\gamma(d_F(p_i,q)))=d_F(p_j,q)-d_F(p_i,q)> 0,
$
which together with (\ref{4.4new*}) yields $f_{ji}(\gamma(t))> 0$ for $t\in [0,d_F(p_i,q)]$. Hence, $\gamma(t)\subset \text{int}(D_{ij})$. Then (\ref{3.1CC**}) implies $\gamma(t)\subset \text{int}(D_i)$. If $q\in {\rm Cut}_{p_i}$, then for any small $\epsilon>0$, the above proof yields that $f_{ji}(\gamma(t))\geq f_{ji}(\gamma(d_F(p_i,q)-\epsilon))$, $t\in [0,d_F(p_i,q)-\epsilon)$ and the same statement follows by the continuity of $f_{ji}$.
\end{proof}

\subsection{Properties of Cheeger's constant}\label{cheegerconst}
In this subsection we study  Cheeger's constant and prove Lemma \ref{infact}.
The co-area formula (cf.\,Shen \cite[Theorem 3.3.1]{Sh1}) yields the following result, which is useful to prove Lemma \ref{Th1old}.
\textcolor[rgb]{0.00,0.00,0.00}{\begin{lemma}\label{chele}Let $(M,F,d\mathfrak{m})$ be an $n$-dimensional complete {\rm FMMM}
and let $D$ be  an open subset of $M.$
Given a  positive function $f\in C^\infty(M)$, we have
$${\mathbbm{h}}(D)\int^\infty_{0}\min\{\mathfrak{m}(\Omega(t)),\mathfrak{m}(D)-\mathfrak{m}(\Omega(t))\}dt\leq \int_DF^*(df)d\mathfrak{m},
$$
where $\Omega(t):=\{x\in D:f(x)\geq t\}$.
\end{lemma}}

\begin{proof}Without loss of generality, we assume $f$ is nonconstant. For almost every $t$ with $0< \min f\leq t \leq \max f<\infty$, $\Omega(t)$ is a domain
in $D$, with compact closure and smooth boundary.
Note that $\mathbf{n}:={\frac{\nabla f}{F(\nabla f)}}$ is a unit normal vector along $\partial\Omega(t)$. The co-area formula yields that
\begin{align*}
\int_DF^*(df)d\mathfrak{m}=\int^\infty_{0}\A_{\mathbf{n}}(\partial\Omega(t))dt\geq {\mathbbm{h}}(D)\int^\infty_{0}\min\{\mathfrak{m}(\Omega(t)),\mathfrak{m}(D)-\mathfrak{m}(\Omega(t))\}dt,
\end{align*}
which concludes the proof.
\end{proof}

Lemma \ref{distorlemma-0}/(i) implies the following result.
\begin{lemma}\label{le2}Given $K\leq 0$ and $\Theta\geq 1$,
Let $(M,F,d\mathfrak{m})$ be   an $n$-dimensional complete {\rm FMMM} with
\[
\mathbf{Ric}\geq (n-1)K,\ |\tau|\leq \log \Theta.
\]
 Then for any $y\in S_pM$, we have
\begin{align*}
(i)\ &{\hat{\sigma}_p(\min\{i_y,r\},y)}\geq  \Theta^{-2}\frac{A_{n,K}(r)}{V_{n,K}(R)-V_{n,K}(r)}{\ds\int_r^R\hat{\sigma}_p(\min\{i_y,t\},y)dt}, \ \forall\, 0<r\leq R;\\
(ii)\ &{\ds\int_{r_0}^{r_1}\hat{\sigma}_p(\min\{i_y,t\},y)dt}\geq   \Theta^{-2} \frac{V_{n,K}(r_1)-V_{n,K}(r_0)}{V_{n,K}(r_2)-V_{n,K}(r_1)}{\ds\int_{r_1}^{r_2}\hat{\sigma}_p(\min\{i_y,t\},y)dt}, \ \forall\, 0<r_0<r_1<r_2,
\end{align*}
where $A_{n,K}(r)$ and $V_{n,K}(r)$  are defined by {\rm (\ref{volume2.5K})}.
\end{lemma}

Let
$i:\Gamma{\hookrightarrow}M$ be a smooth hypersurface
embedded in $(M,F,d\mathfrak{m})$.
Given $p\in M$, let $(r,y)$ denote the polar coordinate system around $p$. For any $x\in  \Gamma\backslash \text{Cut}_p$, one can define a local measure on $\Gamma$ around $x$ by
\[
d\mathfrak{A}:=i^*(\nabla r\rfloor d\mathfrak{m}).
\]
\begin{lemma}\label{le3}
Let $(M,F,d\mathfrak{m})$ be a complete {\rm FMMM} and let
 $i: \Gamma \hookrightarrow  M$ be a smooth hypersurface.
Then for any $x\in \Gamma\backslash \text{Cut}_p$, we have
$dA|_{x}\geq   d\mathfrak{A}|_{x}$.

\end{lemma}
\begin{proof} Let
$\textbf{n}$ denote a unit normal vector field on $\Gamma$. Then we have
\[
d\mathfrak{A}=|i^*(\nabla r\rfloor d\mathfrak{m})|=|g_{\textbf{n}}(\textbf{n},\nabla r)|dA \leq dA,
\]
which is the required relation.
\end{proof}

\begin{proof}[Proof of Lemma \ref{infact}] The proof is almost the same as the one of Chavel \cite[Theorem 6.8]{Ch2} (also see Buser \cite[Lemma 5.1]{B}) and hence we just  sketch it. Let $\Gamma$ be a smooth hypersurface embedded in $D$ which divides $D$ into disjoint open sets $D_1$, $D_2$ in $D$ with common boundary $\partial D_1=\partial D_2=\Gamma$. Without loss of generality, we assume that $\mathfrak{m}(D_1\cap
B_p(r/2))\leq \frac12 \mathfrak{m}(B_p(r/2))\leq \mathfrak{m}(D_2\cap
B_p(r/2))$.
Let $\alpha\in (0,1)$ be a constant which
will be chosen later.

 \textbf{Case 1}: Suppose $\mathfrak{m}(D_1\cap
B_p(r/2))\leq \alpha \mathfrak{m}(D_1)$.
For each $q\in
D_1-\text{Cut}_p$, Let $q^*$ be the last point on the minimal
geodesic segment $\gamma_{pq}$ from $p$ to $q$, where this ray intersects
$\Gamma$. If the whole segment $\gamma_{pq}$ is contained in $D_1$, set
$q^*:=p$. Fix a positive number $\beta\in (0,r/2)$. Let $(t,y)$ denote the polar coordinate system around $p$. Given a point $q=(\rho,y)\in D_1-\text{Cut}_p-B_p(r/2)$, set
\[
\text{rod}(q):=\{(t,y): \beta\leq t\leq \rho\}.
\]
Define
\begin{align*}
&\mathcal {D}^1_1:=\{q\in D_1-\text{Cut}_p-\overline{B_p(r/2)}: q^* \notin B_p(\beta)\};\\
&\mathcal {D}^2_1:=\{q\in D_1-\text{Cut}_p-\overline{B_p(r/2)}: \text{rod}(q)\subset D_1\};\\
&\mathcal {D}^3_1:=\{q\in B_p(r/2)-\overline{B_p(\beta)}: \exists\, x\in \mathcal {D}^2_1 \text{ such that } q\in \text{rod}(x)\}.
\end{align*}

\begin{figure*}[ht]\centering \label{figtt}
\includegraphics[width=4in,height=2.5in]{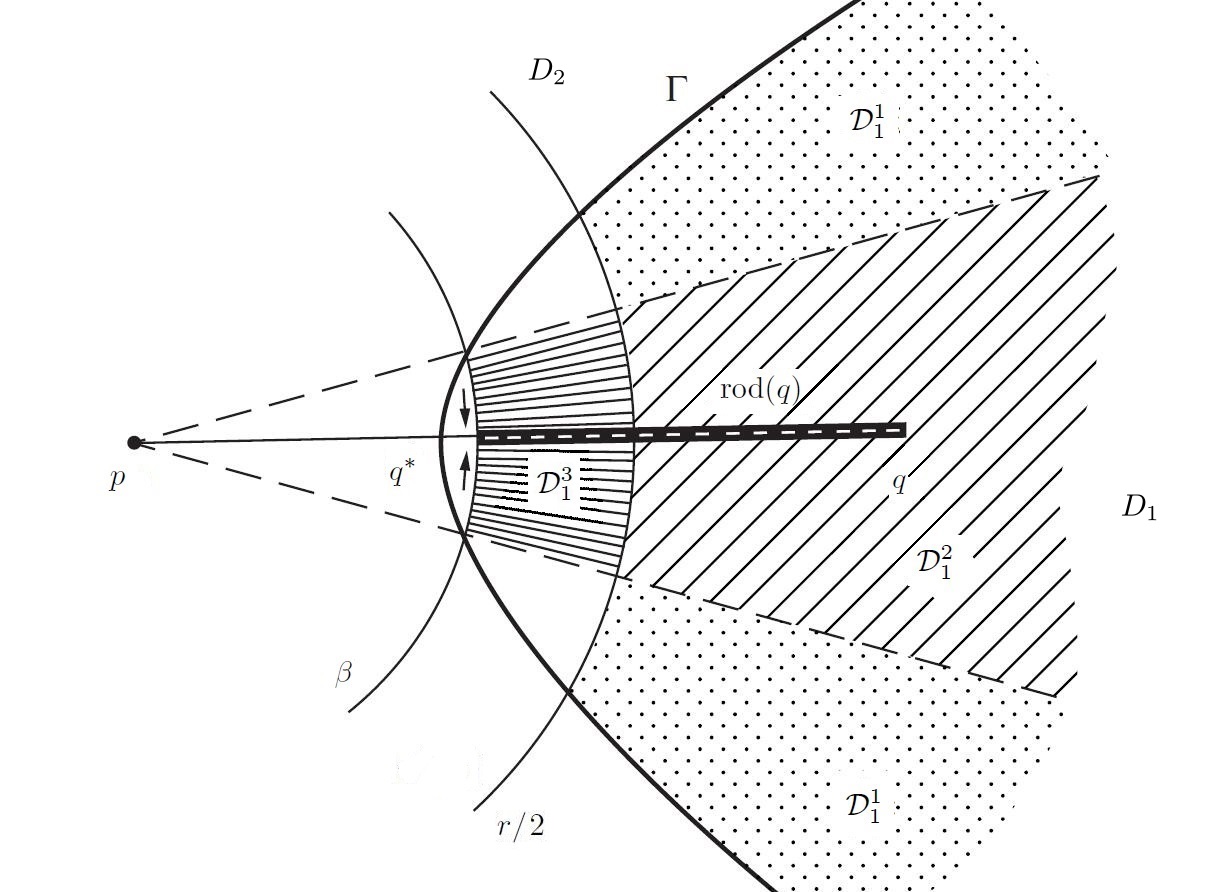}
\caption{}
\end{figure*}

\noindent
By Lemma \ref{le2}/(ii), we obtain that
\[
\frac{\mathfrak{m}(\mathcal {D}^3_1)}{\mathfrak{m}(\mathcal {D}^2_1)}\geq \Theta^{-2}\frac{V_{n,K}(r/2)-V_{n,K}(\beta)}{V_{n,K}(R)-V_{n,K}(r/2)}=:\gamma^{-1}.
\]
It follows from the assumption that
\begin{align*}
\mathfrak{m}(\mathcal {D}^1_1)\geq \left(1-\alpha(1+\gamma)\right)\mathfrak{m}(D_1).\tag{A.5}\label{5.1}
\end{align*}

Set $\mathfrak{D}_1^1 :=\{y\in S_pM: \,\exists\, t>0  \text{ such that } (t,y)\in \mathcal {D}^1_1\}$.
Clearly,
\[
\mathfrak{m}(\mathcal {D}^1_1)=\ds\int_{\mathfrak{D}_1^1}d\nu_p(y)\ds\int_{r/2}^{\min\{R, i_y\}}\chi_{\mathcal {D}^1_1}(\exp_p(ty))\cdot\hat{\sigma}_p(t,y)dt,\tag{A.6}\label{A.2}
\]
where $\chi_{\mathcal {D}^1_1}(x)$ is the characteristic function of $\mathcal {D}^1_1$ and $\exp_p:TM\rightarrow M$ is the usual exponential map at $p$.
The same argument as  Step 3 in the proof of Chavel \cite[Theorem 6.8]{Ch2} together with (\ref{A.2}), Lemmas \ref{le2}/(i) and \ref{le3} then furnishes
\begin{align*}
\mathfrak{m}(\mathcal {D}^1_1)\leq  \Theta^2\frac{V_{n,K}(R)-V_{n,K}(\beta)}{A_{n,K}(\beta)}A(\Gamma). \tag{A.7}\label{5.2}
\end{align*}
Combining (\ref{5.1}) and (\ref{5.2}), we obtain
\begin{align*}
\frac{A(\Gamma)}{\mathfrak{m}(D_1)}=\frac{A(\Gamma)}{\mathfrak{m}(\mathcal {D}_1^1)}\frac{\mathfrak{m}(\mathcal {D}_1^1)}{\mathfrak{m}(D_1)}\geq \Theta^{-2}(1-\alpha(1+\gamma))\frac{A_{n,K}(\beta)}{V_{n,K}(R)-V_{n,K}(\beta)}.\tag{A.8}\label{firstbound}
\end{align*}

\textbf{Case 2}: Suppose $\mathfrak{m}(D_1\cap B_p
({r}/2))\geq \alpha \mathfrak{m}(D_1)$.
For simplicity, set
$W_i:=D_i\cap B_p({r}/2)$, $i=1,2$. Consider the
product space $W_1\times W_2$ with the product measure
$d\mathfrak{m}_{\times}:=d\mathfrak{m} \times d\mathfrak{m}$. Let
\[
N:=\{(q,w)\in W_1\times W_2:\, q\in \text{Cut}_w\}.
\]
Since   the cut locus is a null set,
 Fubini's theorem   yields $\mathfrak{m}_{\times}(N)=0$.
For each $(q,w)\in (W_1\times W_2)\backslash N$, there exists a unique minimal geodesic $\gamma_{wq}$ from $w$ to $q$ with the length $L_F(\gamma_{wq})\leq r$.
The triangle inequality implies  $\gamma_{wq}\subset B_p(r)$. Denote by $q^\natural$ the last point on $\gamma_{wq}$ where $\gamma_{wq}$ intersects $\Gamma$.
Now define
\begin{align*}
V_1&:=\{(q,w)\in W_1\times W_2-N:\, d_F(w,q^\natural)\geq d_F(q^\natural,q)\},\\
V_2&:=\{(q,w)\in W_1\times W_2-N:\, d_F(w,q^\natural)\leq d_F(q^\natural,q)\}.
\end{align*}
Since $\mathfrak{m}_\times(V_1\cup V_2)=\mathfrak{m}_{\times}(W_1\times W_2)$, we have
\[
\mathfrak{m}_{\times}(V_1)\geq \frac12 \mathfrak{m}_\times(W_1\times W_2) \text{ or } \mathfrak{m}_{\times}(V_2)\geq \frac12 \mathfrak{m}_\times(W_1\times W_2).\tag{A.9}\label{twomeasuresmatter}
\]

Since $F$ is reversible, the reverse of a geodesic is still a geodesic. Thus, no matter which one in (\ref{twomeasuresmatter}) holds, a similar argument to  Step 5 in the proof of Chavel \cite[Theorem 6.8]{Ch2} together with Lemmas \ref{le2}/(i) and \ref{le3} yields
\[
\frac{A(\Gamma)}{\mathfrak{m}(D_1)}\geq \frac{\alpha}{2\Theta^2}\frac{A_{n,K}(r/2)}{V_{n,K}(r)-V_{n,K}(r/2)}.\tag{A.10}\label{A.4}
\]

 From (\ref{firstbound}) and (\ref{A.4}), we choose
 \[
 \alpha=\frac{2 \mathcal {A}}{\mathcal {B}+2\mathcal {A}(1+\Theta^2\mathcal {C})},
 \]
where
\[
\mathcal {A}:=\frac{A_{n,K}(\beta)}{V_{n,K}(R)-V_{n,K}(\beta)},\ \mathcal {B}:=\frac{A_{n,K}(r/2)}{V_{n,K}(r)-V_{n,K}(r/2)}, \ \mathcal {C}:=\frac{V_{n,K}(R)-V_{n,K}(r/2)}{V_{n,K}(r/2)-V_{n,K}(\beta)}.
\]
Then
a direct calculation yields
\[
\frac{A(\Gamma)}{\min\{\mathfrak{m}(D_1),\mathfrak{m}(D_2)\}}\geq\frac{A(\Gamma)}{\mathfrak{m}(D_1)}\geq
\sup_{0<\beta<\frac{r}2}\frac{A_{n,K}(\beta)\left[V_{n,K}\left(\frac{r}{2 }\right)-V_{n,K}(\beta)\right]}{4\Theta^4 V_{n,K}(r)V_{n,K}(R)}\geq\frac{C^{1+\sqrt{|K|}R} }{\Theta^4}\frac{r^{n-1}}{R^n},
\]
where $C=C(n)<1$ is a positive number only depending on $n$.
\end{proof}

\proof[Acknowledgements]

The research of A. Krist\'aly is supported by the National Research, Development and Innovation Fund of Hungary, financed under the K$\_$18 funding scheme, Project No.  127926.
This work is also supported by the National Natural Science Foundation of China (No. 11501202, No. 11761058, No. 11671352),  the Natural Science Foundation of Shanghai (No. 17ZR1420900, No. 19ZR1411700) and the grant of China Scholarship Council (No. 201706745006). Work initiated while W. Zhao  was a visiting scholar at
IUPUI.


\begin{thebibliography}{10}




\bibitem{A} L. Ambrosio, S.  Honda and J. W. Portegies, \textsl{Continuity of nonlinear eigenvalues in  $\mathrm{CD}(K,+\infty )$ spaces with respect to measured Gromov-Hausdorff convergence},  Calc. Var. (2018) 57: 34. https://doi.org/10.1007/s00526-018-1315-0.


\bibitem{AlB} J. Alvarez-Paiva, G. Berck, \textsl{What is wrong with the
	Hausdorff measure in Finsler spaces}. Adv.  Math.
{\textbf{204}}(2006), 647--663.

\bibitem{AlT}J. Alvarez-Paiva and A. C. Thompson, \textsl{Volumes in normed and Finsler
	spaces}, A Sampler of Riemann-Finsler geometry (Cambridge) (D. Bao,
R. Bryant, S.S. Chern, and Z. Shen, eds.), Cambridge University
Press, 2004, pp. 1--49.







\bibitem{BCS} D. Bao, S. S. Chern and Z. Shen, \textsl{An introduction
to Riemannian-Finsler geometry}, GTM {\bf{200}}, Springer-Verlag,
2000.

\bibitem{BI} D. Burago and  S. Ivanov, \textit{On asymptotic volume of Finsler tori, minimal surfaces in normed spaces, and symplectic filling volume}. Ann. of Math. (2) \textbf{156}(2002), no. 3, 891--914.


\bibitem{B} P. Buser, \textsl{A note on the isoperimetric constant}, Ann. Sci. \'Ec. Norm. Sup.  \textbf{15}(1982), 213-230.





\bibitem{C4} M. Craioveanu, M. Puta, and Th. M. Rassias, \textsl{Old and new aspects in spectral geometry},
volume \textbf{534} of Mathematics and its Applications, Kluwer Academic Publishers, Dordrecht,
2001.

\bibitem{Ch} I. Chavel, \textsl{Eigenvalues in Riemannian geometry}, Academic Press, New York, 1984.

\bibitem{Ch2} I. Chavel, \textsl{Riemannian geometry: A
modern introduction}, Cambridge Univ., 1993.

\bibitem{Cheeger} J. Cheeger, \textsl{A lower bound for smallest eigenvalue of the Laplacian}, In Problems in
Analysis, pages 195--199, Princeton University Press, 1970.

\bibitem{CC} J. Cheeger and T. H. Colding, \textsl{Lower bounds on Ricci curvature and the almost rigidity
of warped products}, Ann. of Math.  \textbf{144}(1996), 189-237.




\bibitem{C} S. Cheng, \textsl{Eigenvalue comparison theorems and its geometric applications}. Math. Z. \textbf{143}(1975), 289--297.

\bibitem{Chern} S. S. Chern, \textsl{Finsler geometry
	is just Riemannian
	geometry without the
	quadratic restriction}.  Notices Amer. Math. Soc. \textbf{43} (1996), no. 9, 959--963.



\bibitem{CLOT}  O. Cornea,
G. Lupton,
J. Oprea, and
D. Tanr\'e,
\textsl{Lusternik-Schnirelmann
Category}, Mathematical Surveys and Monographs, Volume \textbf{103}, 2003.


\bibitem{E} D. Egloff, \textsl{Uniform Finsler Hadamard manifolds}, Ann. Inst. Henri Poincar\'e  \textbf{66}(1997), 323--357.

\bibitem{F} E. Fadell, \textsl{The relationship between Lusternik-Schnirelmann category and the concept of genus}, Pacific J. Math.  \textbf{89} (1980), 33--42.


\bibitem{G3} M. Gromov, \textsl{Paul Levy's isoperimetric inequality}, Preprint, Inst. Hautes Etudes Sci., Publ. Math., 1980.





\bibitem{G} M. Gromov, \textsl{Dimension, nonlinear spectra and width}, Geometric aspects of functional
analysis, Israel seminar (1986-87), Lecture Notes in Math., \textbf{1317}, Springer, Berlin (1988),
132--184.


\bibitem{G2} M. Gromov, \textsl{Metric structures for Riemannian and non-Riemannian spaces}. Based on the 1981 French original.
With appendices by M. Katz, P. Pansu and S. Semmes. Translated from the French by Sean Michael Bates.
Progress in Mathematics, \textbf{152}. Birkh¡§auser Boston, Inc., Boston, MA, 1999. xx+585 pp.



\bibitem{GS} Y. Ge and Z. Shen, \textsl{Eigenvalues and eigenfuncitons of metric measure manifolds},  Proc.
London Math. Soc. (3) \textbf{82}(2001), 725--746.

\bibitem{HKP} A. Hassannezhad, G. Kokarev, and I. Polterovich, \textsl{Eigenvalue inequalities on Riemannian manifolds with a lower Ricci curvature bound}, Yuri Safarov's memorial volume. J. Spectr. Theory \textbf{16}(2016), 807--835.


\bibitem{H} E. Hebey, \textsl{Sobolev Spaces on Riemannian Manifolds}. Springer, 1996.






\bibitem{HW} W. Hurwicz and H. Wallman, \textsl{Dimension Theory}, Princeton University Press, 1948.




\bibitem{K} S. Kakutani, \textsl{Topological properties of the unit sphere of a Hilbert space}, Proc.
Imp. Acad. Tokyo, \textbf{19}(1943),  269--271.

\bibitem{K3}   S. Kronwith, \textsl{Convex manifolds of nonnegative curvature}, Journal of Differential Geometry
\textbf{14}(1979), 621--628.


\bibitem{K2} M. A. Krasnoselskii, \textsl{Topological Methods in the Theory of Nonlinear Integral
Equations}, MacMillan, N. Y., (1965).

\bibitem{K-R} A. Krist\'aly and I. Rudas, \textsl{Elliptic problems on the ball endowed with Funk-type metrics}.
Nonlinear Anal. \textbf{119}(2015), 199--208.

\bibitem{LL} E.H. Lieb and M. Loss, Analysis. Second edition. Graduate Studies in Mathematics, 14. American Mathematical Society, Providence, RI, 2001.

\bibitem{O} S.-I. Ohta and K.-T. Sturm, \textsl{Heat flow on Finsler manifolds}, Comm. Pure Appl.
Math. \textbf{62}(2009), 1386--1433.


\bibitem{PS2} R.S. Palais, \textsl{Homotopy theory of infinite dimensional manifolds}, Topology \textbf{5}(1966), 1--16.

\bibitem{Se} A.G. Setti, \textsl{Eigenvalue estimates for the weighted Laplacian on a Riemannian manifold},
Rend. Sem. Mat. Univ. Padova. \textbf{100}(1998), 27-55.



\bibitem{PS} R.S. Palais, \textsl{Lusternik-Schnirelman theory on Banach manifolds},  Topology \textbf{5}(1966), 115--132.



\bibitem{Shen-1} Z. Shen, \textit{Projectively flat Finsler metrics of constant flag curvature}. Trans.  Amer. Math. Soc. \textbf{355}(4)(2003), 1713--1728.

\bibitem{Shen-2} Z. Shen, \textit{Finsler metrics with K=0 and S=0}. Canad. J. Math. \textbf{55}(2003), no.1, 112--132.

\bibitem{Shen_Adv_Math} Z. Shen, \textsl{Volume comparison and its applications in Riemann-Finsler geometry}. Adv. Math. \textbf{128}(1997), no. 2, 306--328.



\bibitem{Sh1} Z. Shen, \textsl{Lectures on Finsler geometry}, World
Sci., Singapore, 2001.






\bibitem{S} M. Struwe, \textsl{Variational methods}, volume 34 of Ergebnisse derMathematik und ihrer Grenzgebiete.
3. Folge. A Series of Modern Surveys in Mathematics [Results in Mathematics and Related
Areas. 3rd Series. A Series of Modern Surveys in Mathematics]. Springer-Verlag, Berlin,
fourth edition, 2008. Applications to nonlinear partial differential equations and Hamiltonian
systems.



\bibitem{YZ} L. Yuan and W. Zhao, \textsl{Some formulas of Santal\'o type in Finsler geometry and its applications}, Publicationes Mathematicae Debrecen  \textbf{87}(2015),  79--101.




\bibitem{Z} E. Zeidler, \textsl{Nonlinear Functional Analysis and Its Applications IV: Applications to Mathematical Physics},  Springer-Verlag. Berlin, Germany, 1997.

\bibitem{ZS} W. Zhao and Y. Shen, \textsl{A Universal Volume comparison Theorem
for Finsler Manifolds and Related Results},  Can. J. Math.  \textbf{65}(2013), 1401--1435.

\bibitem{Z22} W. Zhao, \textsl{Integral curvature bounds and diameter estimates on
Finsler manifolds}, Sci. China Math.  (accepted).

%

\end{thebibliography}
\end{document}